\documentclass[10pt]{amsart}

\makeatletter
\@namedef{subjclassname@2020}{\emph{2020} Mathematics Subject Classification}
\makeatother

\newcommand{\doi}[1]{\href{https://doi.org/#1}{\nolinkurl{doi:#1}}}
\newcommand{\arXiv}[2]{Available at: \href{https://arxiv.org/abs/#1}{\nolinkurl{arXiv:#1}}, #2 pp.}

\usepackage{amsmath, amsthm, amsfonts, amssymb} 
\usepackage{mathrsfs} 
\usepackage{enumerate} 
\usepackage{bbm} 
\usepackage{bm} 
\usepackage{xcolor}
\usepackage[colorlinks,citecolor=blue!50!cyan,linkcolor=black]{hyperref} 

\usepackage{chngcntr}
\counterwithin{equation}{section} 

\usepackage{thmtools}
\usepackage{thm-restate} 

	\newtheorem{theorem}{Theorem}[section]
	\newtheorem{lemma}[theorem]{Lemma}
	\newtheorem{proposition}[theorem]{Proposition}
	\newtheorem{corollary}[theorem]{Corollary}
	\newtheorem{conjecture}[theorem]{Conjecture}
\theoremstyle{definition}
	\newtheorem{definition}[theorem]{Definition}
	\newtheorem{remark}[theorem]{Remark}
	\newtheorem{question}[theorem]{Question}

\newcommand{\N}{\mathbb{N}}
\newcommand{\Z}{\mathbb{Z}}
\newcommand{\Q}{\mathbb{Q}}
\newcommand{\R}{\mathbb{R}}
\newcommand{\C}{\mathbb{C}}
\newcommand{\T}{\mathbb{T}}

\newcommand{\bbS}{\mathbb{S}}

\newcommand{\A}{\mathbb{A}}

\newcommand{\mZ}{\mathcal{Z}}
\newcommand{\mY}{\mathcal{Y}}
\newcommand{\mX}{\mathcal{X}}

\newcommand{\mH}{\mathcal{H}}

\newcommand{\mA}{\mathcal{A}}
\newcommand{\mB}{\mathcal{B}}

\newcommand{\mM}{\mathcal{M}}
\newcommand{\mE}{\mathcal{E}}

\newcommand{\EFS}{\mathbf{EFS}}

\newcommand{\eps}{\varepsilon}

\newcommand{\es}{\emptyset}

\renewcommand{\tilde}{\widetilde}
\renewcommand{\hat}{\widehat}

\newcommand{\UClim}{\textup{UC-}\lim}

\newcommand{\gen}{\textup{gen}}

\newcommand{\id}{\textup{id}}

\newcommand{\ind}{\mathbbm{1}}

\newcommand{\supp}[1]{\textup{supp} \left( #1 \right)}

\newcommand{\diam}[1]{\textup{diam} \left( #1 \right)}

\newcommand{\E}[2]{\mathbb{E} \left[ #1 \mid #2 \right]}

\newcommand{\innprod}[2]{\left\langle #1, #2 \right\rangle}
\newcommand{\norm}[2]{\left\| #2 \right\|_{#1}}
\newcommand{\seminorm}[2]{{\left\vert\kern-0.25ex\left\vert\kern-0.25ex\left\vert #2 
    \right\vert\kern-0.25ex\right\vert\kern-0.25ex\right\vert}_{#1}}

\title[Infinite polynomial patterns in $\Q$]{Infinite polynomial patterns in large subsets of the rational numbers}

\author{Ethan Ackelsberg}
\address{Institute of Mathematics, EPFL, 1015-Lausanne, Switzerland}
\email{ethan.ackelsberg@epfl.ch}

\date{\today}

\keywords{}

\subjclass[2020]{Primary: 05D10; Secondary: 11B13, 37A15}

\begin{document}


\begin{abstract}
	Inspired by a question of Kra, Moreira, Richter, and Robertson \cite{kmrr_survey}, we prove two new results about infinite polynomial configurations in large subsets of the rational numbers.
	First, given a finite coloring of $\Q$, we show that there exists an infinite set $B = \{b_n : n \in \N\} \subseteq \Q$ such that
	\begin{equation*}
		\{b_i, b_i^2 + b_j : i < j\}
	\end{equation*}
	is monochromatic.
	Second, we prove that every subset of positive density in the rational numbers contains a translate of such an infinite configuration.
	The corresponding results in the integers are both known to be false, so our results provide natural and relatively simple examples of combinatorial structures that distinguish between the Ramsey-theoretic properties of the rational numbers and the integers.
	
	The proofs of our main results build upon methods developed in a series of papers by Kra, Moreira, Richter, and Robertson \cite{kmrr1, kmrr2, kmrr_FS} to translate from combinatorics into dynamics, where the core of the argument reduces to understanding the behavior of certain polynomial ergodic averages.
	The new dynamical tools required for this analysis are a Wiener--Wintner theorem for polynomially-twisted ergodic averages in $\Q$-systems and a structure theorem for Abramov $\Q$-systems.
\end{abstract}

\maketitle


\section{Introduction}


\subsection{Infinite sumset configurations}

The breakthrough work of Moreira, Richter, and Robertson resolving the Erd\H{o}s sumset conjecture \cite{mrr} opened up a new active area of research into which \emph{infinite} combinatorial patterns can be found in large subsets of the integers or other groups or semigroups.
Following the simplified dynamical proof of the Erd\H{o}s sumset conjecture given by Host \cite{host}, a series of papers by Kra, Moreira, Richter, and Robertson \cite{kmrr1, kmrr2, kmrr_FS} recently culminated in the ``density finite sums theorem,'' stating that every positive density\footnote{By ``density,'' we always mean the \emph{upper Banach density}, which is defined for a set $A \subseteq \N$ by
\begin{equation*}
	d^*(A) = \lim_{N \to \infty} \sup_{x \in \N} \frac{|A \cap \{x+1, x+2, \ldots, x+N\}|}{N}.
\end{equation*}} subset of the positive integers contains, for every $k \in \N$, a configuration of the form
\begin{equation*}
	t + \left\{ \sum_{x \in F} x : F \subseteq B, 0 < |F| \le k \right\}
\end{equation*}
for some $t \in \Z$ and some infinite set $B \subseteq \N$ (see \cite[Theorem 1.1]{kmrr_FS}).
A subsequent paper of Hern\'{a}ndez \cite{hernandez} provides a full classification of the infinite linear patterns that can be found in arbitrary sets of positive density in the integers.

The original work of Moreira, Richter, and Robertson \cite{mrr} applies to general countable amenable groups, and some of the other sumset results have also been extended beyond the integers: in \cite{cm}, Charamaras and Mountakis proved the $k=2$ case of the density finite sums theorem for an appropriate class of abelian groups and certain additional amenable groups; and in \cite{aj}, Jamneshan and the author extended the $k=3$ case to abelian groups.


\subsection{Polynomial sumsets}

While the aforementioned work addresses quite thoroughly the behavior of linear patterns, it leaves open many natural questions about which (if any) nonlinear patterns one should expect in arbitrary large sets.
One of the simplest such questions, inspired by the Furstenberg--S\'{a}rk\"{o}zy theorem \cite{furstenberg, sarkozy} and posed by Kra, Moreira, Richter and Robertson in their survey paper on infinite sumset configurations (see \cite[Question 3.15]{kmrr_survey}), is whether for any set $A \subseteq \N$ with positive density there exist an infinite set $B = \{b_1 < b_2 < \ldots\}$ and a shift $t \in \Z$ such that
\begin{equation} \label{eq: sumset configuration with B}
	\left\{ b_i, b_i^2 + b_j : i < j \right\}
\end{equation}
is contained in $A - t$.
Note that taking $x = b_j + t$, $y = b_i^2 + b_j + t$, and $n = b_i$ gives a pair of points $x, y \in A$ with $y - x = n^2$, so an affirmative answer to this question strengthens the Furstenberg--S\'{a}rk\"{o}zy theorem.
A weaker question that was also asked in \cite[Question 3.13]{kmrr_survey} is whether every set $A \subseteq \N$ of positive density contains a shift of a set of the form
\begin{equation} \label{eq: sumset configuration}
	\left\{ b_i^2 + b_j : i < j \right\}
\end{equation}
for some infinite set $B = \{b_1 < b_2 < \ldots\}$.

There is substantial evidence in favor of an affirmative answer to both of the questions from the previous paragraph.
To enable this discussion, we recall some definitions.
A set $A \subseteq \N$ is
\begin{itemize}
	\item	\emph{thick} if it contains arbitrary long intervals,
	\item	\emph{syndetic} if it has bounded gaps, and
	\item	\emph{piecewise syndetic} if it can be written as the intersection of a thick set with a syndetic set.
\end{itemize}
Piecewise syndetic sets are an important class of sets in Ramsey theory; for example, van der Waerden's theorem can be rephrased as the statement that every piecewise syndetic set contains arbitrarily long arithmetic progressions.\footnote{This formulation of van der Waerden's theorem appears to have first been observed by Furstenberg and Glasner \cite{fg}, who provide a short argument in the opening paragraph of their paper.
Furstenberg and Glasner use the term ``big'' instead of piecewise syndetic and refer to thick sets as ``replete'' sets.
For a rendering of their argument in modern terminology, see the first page of \cite{beiglbock}.}
We will now show that every piecewise syndetic set contains a configuration of the form \eqref{eq: sumset configuration} up to a shift.
Suppose that $A \subseteq \N$ is a piecewise syndetic set.
Then there exists $k \in \N$ such that $T = A \cup (A - 1) \cup \ldots \cup (A - k)$ is thick.
Define a sequence $c_1 < c_2 < \ldots$ as follows.
Let $c_1 = 1$.
Given $c_1 < \ldots < c_n$, we use the fact that $T$ is thick to choose $c_{n+1}$ such that $c_{n+1} + \{c_1^2, \ldots, c_n^2\} \subseteq T$.
We thus have an infinite sequence $(c_n)_{n \in \N}$ such that $\{c_i^2 + c_j : i < j\} \subseteq T$.
Now we define a $(k+1)$-coloring of the 2-element subsets of $\N$ by
\begin{equation*}
	\chi(\{i,j\}) = \min\{t \in \{0, 1, \ldots, k\} : c_i^2 + c_j \in A - t\}
\end{equation*}
for $i < j$.
By Ramsey's theorem, there is an infinite subset $S \subseteq \N$ and a color $t \in \{0, 1, \ldots, k\}$ such that $\chi(\{i,j\}) = t$ for every $i, j \in S$, $i < j$.
Enumerating $S = \{s_1 < s_2 < \ldots\}$ and putting $b_n = c_{s_n}$, we see that $A - t$ contains a set of the form \eqref{eq: sumset configuration}.
Thus, the second question in the previous paragraph (\cite[Question 3.13]{kmrr_survey}) has a positive answer for piecewise syndetic sets.

The above discussion does not rely on any particular properties of the set of squares, which may be hiding potential issues that arise when trying to upgrade from piecewise syndetic sets to arbitrary sets of positive upper Banach density or from the more limited configuration in \eqref{eq: sumset configuration} to the configuration in \eqref{eq: sumset configuration with B}.
This is where the Furstenberg--S\'{a}rk\"{o}zy theorem plays a role: using the fact that $\{n^2 : n \in \N\}$ is a set of strong recurrence\footnote{Subsequent work by Hetzel \cite{hetzel} establishes a similar result for sets of recurrence without assuming that they are also sets of strong recurrence.}, it was shown in \cite[Corollary 3.33]{kmrr_survey} that for any set $A \subseteq \N$ of positive density, there exist infinite sets $B, C \subseteq \N$ such that $\{c, b^2 + c : b \in B, c \in C, b < c\} \subseteq A$.
Asking for a configuration of the form \eqref{eq: sumset configuration with B} is equivalent to asking for $B$ and $C$ to be equal (at the cost of allowing $A$ to be shifted).

Despite the evidence in its favor, even the weaker question of Kra, Moreira, Richter, and Robertson has a negative answer, as shown in \cite{counterexample}:

\begin{theorem}[{\cite[Theorem 1.5]{counterexample}}] \label{thm: counterexample}
	For any $\eps > 0$, there exists a set $A \subseteq \N$ with $d^*(A) > 1 - \eps$ such that if $B \subseteq \N$, $t \in \Z$, and
	\begin{equation*}
		\{b_1^2 + b_2 : b_1, b_2 \in B, b_1 < b_2\} \subseteq A - t,
	\end{equation*}
	then $B$ is finite.
\end{theorem}

In \cite{kmrr_survey}, a related question about finite colorings of the integers was also formulated; namely, if $\N$ is finitely colored, can one can always find an infinite set $B = \{b_1 < b_2 < \ldots\}$ such that $\{b_i, b_i^2 + b_j : i < j\}$ is monochromatic (see \cite[Remark 3.17]{kmrr_survey})?
We note that by a clever argument of Bergelson (see \cite[p. 53]{ert_update}) one can always find a monochromatic triple of the form $\{b_1, b_2, b_1^2 + b_2\}$.
However, a recent paper of Di Nasso, Luperi Baglini, Mamino, Mennuni, and Ragosta \cite{ramsey_witness} answers the coloring question of Kra, Moreira, Richter, and Robertson again in the negative\footnote{The phrasing of \cite[Corollary 5.10]{ramsey_witness} uses terminology that may not be familiar to the reader.
In \cite{ramsey_witness}, the authors say that an equation $f(x) + g(y) = h(z)$ is \emph{Ramsey partition regular} if for any finite coloring of $\N$, there is an infinite monochromatic set $B \subseteq \N$ such that for any $b_1, b_2 \in B$ with $b_1 < b_2$, there exists $c$ of the same color as the elements of $B$ such that $f(b_1) + f(b_2) = h(c)$.
So, for example, Theorem \ref{thm: not ramsey pr} is equivalent to the statement that the equation $x^2+y=z$ fails to be Ramsey partition regular.}:

\begin{theorem}[cf. {\cite[Corollary 5.10]{ramsey_witness}}] \label{thm: not ramsey pr}
	There exists a finite coloring of $\N$ with the property that for any infinite set $B = \{b_1 < b_2 < \ldots\}$, the set $\{b_i, b_i^2 + b_j : i < j\}$ contains elements of at least two different colors.
\end{theorem}

The proof given in \cite{ramsey_witness} uses methods from nonstandard analysis and does not provide an explicit coloring.
In Section \ref{sec: 5-coloring}, we translate the argument into a simple combinatorial proof using 5 colors.

\bigskip

Theorems \ref{thm: counterexample} and \ref{thm: not ramsey pr} seem to eliminate the possibility of meaningfully incorporating polynomials into the framework of infinite sumset configurations.
We show in this paper, however, that not all hope is lost: both Theorem \ref{thm: counterexample} and Theorem \ref{thm: not ramsey pr} depend on particular number-theoretic aspects of the integers, and the questions of Kra, Moreira, Richter, and Robertson in fact have positive answers when one moves to the rational numbers, as we shall see in the main results formulated below.
As a consequence, we exhibit new combinatorial patterns that have divergent Ramsey-theoretic behaviors with respect to the integers and the rationals, a phenomenon for which only limited families of examples were previously known (see \cite{bhl,bhls}).


\subsection{New results}

To give a precise formulation of our first main result, we need a definition.
A sequence $\Phi = (\Phi_N)_{N \in \N}$ of finite subsets of $\Q$ is an (additive) \emph{F{\o}lner sequence} if
\begin{equation*}
	\lim_{N \to \infty} \frac{\left| (\Phi_N + x) \triangle \Phi_N \right|}{|\Phi_N|} = 0
\end{equation*}
for every $x \in \Q$.
An example of a F{\o}lner sequence in $\Q$ is the sequence
\begin{equation*}
	\Phi_N = \left\{ \sum_{n=1}^N \frac{a_n}{n} : a_n \in \Z, -N \le a_n \le N \right\}.
\end{equation*}
The \emph{upper Banach density} of a set $A \subseteq \Q$ is the quantity
\begin{equation*}
	d^*(A) = \sup_{\Phi} \limsup_{N \to \infty} \frac{|A \cap \Phi_N|}{|\Phi_N|},
\end{equation*}
where the supremum is taken over all F{\o}lner sequences $\Phi$ in $\Q$.

\begin{theorem} \label{thm: density}
	Let $A \subseteq \Q$ with $d^*(A) > 0$.
	Then there exists $t \in \Q$ and an infinite set $\{b_n : n \in \N\} \subseteq \Q$ such that
	\begin{equation*}
		\{b_i, b_i^2 + b_j : i < j\} \subseteq A - t.
	\end{equation*}
\end{theorem}

We can replace the polynomial term $b_i^2$ by a more general polynomial expression $P(b_i)$ but with some restriction on the polynomial $P(x) \in \Q[x]$.
To aid with the discussion, call a polynomial $P(x) \in \Q[x]$ \emph{good for sumsets} if Theorem \ref{thm: density} holds with $P(b_i)$ in place of $b_i^2$.
We will now describe the family of polynomials that are good for sumsets.

One relatively simple obstruction to being good for sumsets is when $P = c$ for some nonzero constant $c \in \Q \setminus \{0\}$.
Indeed, we can find a set $A \subseteq \Q$ such that $A \cap (A+c) = \es$ and $\Q = A \cup (A+c)$.
Such a set $A$ must have density $\frac{1}{2}$, so we see that $P=c$ is bad for sumsets.

A more subtle obstacle arises from certain linear polynomials.
By modifying a construction of Straus (never published by Straus but described, e.g., in \cite{bbhs} and \cite[Section 2]{counterexample}), we show in Section \ref{sec: straus example rationals} that there exists a subset $A \subseteq \Q$ with $d^*(A) > 0$ such that no translate of $A$ contains a $\Delta$-set, i.e. an infinite set of the form $\{b_j - b_i : i < j\}$.
As a consequence, the polynomial $P(x) = -x + c$ is bad for sumsets for every $c \in \Q$.

As we will show, there are no other obstacles to producing polynomial sumsets.
That is, $P(x) \in \Q[x]$ is good for sumsets if and only if $P = 0$ or both $P(x)$ and $P(x) + x$ are nonconstant, as demonstrated by the following theorem (we do not account for the trivial $P = 0$ case explicitly):

\begin{theorem} \label{thm: density P}
	Let $P(x) \in \Q[x]$ such that both $P(x)$ and $P(x) + x$ are nonconstant.
	Let $A \subseteq \Q$ with $d^*(A) > 0$.
	Then there exists $t \in \Q$ and an infinite set $\{b_n : n \in \N\} \subseteq \Q$ such that
	\begin{equation*}
		\{b_i, P(b_i) + b_j : i < j\} \subseteq A - t.
	\end{equation*}
\end{theorem}

Using the methods developed for proving Theorem \ref{thm: density P}, we are also able to prove a result about monochromatic infinite polynomial sumset configurations in finite colorings of $\Q$.

\begin{theorem} \label{thm: coloring P}
	Let $P(x) \in \Q[x]$ with $\deg{P} \ge 2$.
	Then for any finite coloring of $\Q$, there exists an infinite set $B = \{b_n : n \in \N\}$ such that
	\begin{equation*}
		\{b_i, P(b_i) + b_j : i < j\}
	\end{equation*}
	is monochromatic.
\end{theorem}

We do not know what happens in the linear case $\deg{P} = 1$, and the following question appears to be open:

\begin{question}
	For which linear polynomials $P(x) = cx+d \in \Q[x]$ does the following hold: for any finite coloring of $\Q$, there is an infinite set $B = \{b_n : n\in \N\}$ such that $\{b_i, P(b_i) + b_j : i < j\}$ is monochromatic?
\end{question}

This problem is surprisingly recalcitrant, and the only linear polynomial for which we know the answer is $P(x) = x$: by Hindman's theorem \cite{hindman}, if $\Q$ is finitely colored, then there exists an infinite set $B = \{b_n : n \in \N\}$ such that $\{b_i, b_i + b_j : i < j\}$ is monochromatic.

\bigskip

The asymmetric restriction $i < j$ is a necessary feature of both Theorem \ref{thm: density P} and Theorem \ref{thm: coloring P}.
Both theorems fail if the condition $i < j$ is replaced by the seemingly more natural condition $i \ne j$.
In fact, the configuration $\{b_i, b_i^d + b_1 : i > 1\}$ already loses the Ramsey-theoretic properties enjoyed by $\{b_i, b_i^d + b_j : i < j\}$, as shown by the next two theorems, which we prove in Section \ref{sec: i>j examples}.

\begin{restatable}{theorem}{ReversedOrderDensity} \label{thm: i>j bad for density}
	Let $d \ge 2$, and let $\Phi$ be a F{\o}lner sequence in $\Q$.
	There exists a set $A \subseteq \Q$ with $d_{\Phi}(A) = 1$ such that for each $t \in \Q$, there are only finitely many $q \in \Q$ such that $q^d + t \in A$.
\end{restatable}

\begin{restatable}{theorem}{ReversedOrderColoring} \label{thm: i>j bad for coloring}
	Let $d \ge 2$.
	There exists a 3-coloring $\chi : \Q \to \{0,1,2\}$ such that for every $t \in \Q$, there are only finitely many $q \in \Q$ such that $\chi(q) = \chi(q^d+t)$.
\end{restatable}


\subsection{Overview of the paper}

We will prove Theorem \ref{thm: density P} and Theorem \ref{thm: coloring P} using techniques from ergodic theory.
For problems concerning patterns in sets of positive density, dynamical tools have in most cases been deployed in the language of abstract ergodic theory with no reference to an underlying topological structure (see, e.g., \cite[Chapters 3--7]{furstenberg_book}).
As highlighted by Host's dynamical proof of the Erd\H{o}s sumset conjecture \cite{host}, topology suddenly plays a prominent role when one is interested in infinite patterns as opposed to the finite patterns handled by classical results in ergodic Ramsey theory.
The combination of topological and ergodic-theoretic arguments has since been used to great effect in the various extensions of the Erd\H{o}s sumset conjecture mentioned at the beginning of the introduction.
Our proof of Theorem \ref{thm: density P} builds upon these recent developments and utilizes topological as well as measure-theoretic tools.

On the other hand, in most applications of dynamical methods to coloring problems (first established in a systematic way in \cite{fw}; see also \cite[Chapters 1, 2, and 8]{furstenberg_book}), the natural setting is in the realm of topological dynamics.
In keeping with this tradition, we give a reformulation of Theorem \ref{thm: coloring P} as a statement in topological dynamics in Section \ref{sec: Furstenberg systems} (see Theorem \ref{thm: coloring dynamical}).
However, in the course of the proof, we introduce a measure and carry out the majority of the argument in the context of measure-preserving systems in order to make use of many of the same ingredients that appear in the proof of Theorem \ref{thm: density P}.

\bigskip

The paper is organized as follows.
Section \ref{sec: Furstenberg systems} uses the framework introduced in \cite{kmrr2} to translate Theorem \ref{thm: density P} and Theorem \ref{thm: coloring P} into dynamical language.
There is substantial overlap in the arguments for both of the main theorems, and the common tools for both results are developed in Sections \ref{sec: characteristic factor}--\ref{sec: progressive measure}.
We complete the proof of Theorem \ref{thm: density P} in Section \ref{sec: proof finish density}; the finishing touches for the proof of Theorem \ref{thm: coloring P} appear in Section \ref{sec: proof finish coloring}.
Two important technical results about polynomial structures in measure-preserving systems play a central role in our argument: (1) a Wiener--Wintner-type theorem for ergodic averages twisted by polynomial phases (Theorem \ref{thm: polynomial ww adelic}), and (2) a structure theorem for Abramov systems (Theorem \ref{thm: Abramov structure}).
Both of these theorems have known analogues in the case of (totally ergodic) $\Z$-actions but are new for $\Q$-actions.
The proofs of these new results, which may be of independent interest, appear in Section \ref{sec: ww} and Section \ref{sec: Abramov}, respectively.
As promised in the discussion following Theorem \ref{thm: density}, an adaptation of Straus's example to the rational numbers is detailed in Section \ref{sec: straus example rationals}.
Examples demonstrating the importance of the asymmetric $i < j$ restriction in the polynomial sumsets that we have considered are given in Section \ref{sec: i>j examples}.
We end the paper with a discussion of infinite polynomial sumset configurations in the integers in Section \ref{sec: integers}.


\section*{Acknowledgements}

This work is supported by the Swiss National Science Foundation under grant TMSGI2-211214.
Thanks to Florian K. Richter for helpful comments on an earlier draft and to Mauro Di Nasso for drawing the author's attention to the recent paper \cite{ramsey_witness} and explaining the relevance of \emph{Ramsey witnesses} from nonstandard analysis.
Thanks also to the anonymous referees for their careful reading of the paper and productive feedback, which led to a correction in the statement of Theorem \ref{thm: polynomial ww adelic} and prompted the inclusion of Section \ref{sec: i>j examples}.


\section{Dynamical reformulation of the main theorems} \label{sec: Furstenberg systems}

The Furstenberg correspondence principle, originating in Furstenberg's ergodic-theoretic proof of Szemer\'{e}di's theorem \cite{furstenberg}, is a tool that allows one to encode combinatorial statements in a dynamical framework.
Before stating the version of the Furstenberg correspondence principle that we will need, we introduce basic objects from ergodic theory.

A \emph{topological dynamical $\Q$-system} is a pair $(X, T)$, where $X$ is a compact metric space and $T$ is an action of $(\Q,+)$ by homeomorphisms $T^q : X \to X$.
By a (measure-preserving) \emph{$\Q$-system}, we will mean a triple $(X, \mu, T)$, where $(X,T)$ is a topological dynamical $\Q$-system and $\mu$ is a Borel probability measure on $X$ that is $T$-invariant, meaning that $\mu(T^{-q}A) = \mu(A)$ for every Borel set $A \subseteq X$ and every $q \in \Q$.
(Note that the notion of a measure-preserving system is often more general than we have stated here, allowing for $(X, \mu)$ to be an arbitrary probability space and $T$ a measurable map.
Several arguments in this paper rely on simultaneously keeping track of both topological and measure-theoretic properties, which is why we impose additional topological constraints in the definition of a system.)

A system $(X, \mu, T)$ is called \emph{ergodic} if whenever $A \subseteq X$ is a Borel set satisfying $\mu(A \triangle T^{-q}A) = 0$ for every $q \in \Q$, one has $\mu(A) \in \{0,1\}$.
A point $a \in X$ is \emph{transitive} if the orbit $\{T^q a : q \in \Q\}$ is dense in $X$.
We say that $a$ is \emph{generic for $\mu$ along a F{\o}lner sequence $\Phi$}, abbreviated $a \in \gen(\mu, \Phi)$, if $\frac{1}{|\Phi_N|} \sum_{q \in \Phi_N} \delta_{T^q a} \to \mu$ in the weak$^*$ topology.

\begin{theorem}[Topological Furstenberg correspondence principle] \label{thm: topological correspondence}
	Let $\chi : \Q \to \{1, \ldots, r\}$ be an $r$-coloring of $\Q$ for some $r \in \N$.
	Then there exists a topological dynamical $\Q$-system $(X, T)$, a transitive point $a \in X$, and a partition $X = \bigcup_{k=1}^r E_k$ of $X$ into clopen sets such that $T^q a \in E_{\chi(q)}$ for every $q \in \Q$.
\end{theorem}

\begin{proof}
	Encoding a finite coloring in a topological dynamical system has its origins in \cite{fw}, where a corresponding statement for $\Z$ appears implicitly.
	We give a short self-contained proof for the convenience of the reader.
	
	Let $\Omega = \{1, \ldots, r\}^{\Q}$ be the space of all $r$-colorings of $\Q$.
	Equipping $\Omega$ with the product topology makes it into a compact set with compatible metric given, for example, by
	\begin{equation*}
		d(\omega, \omega') = \begin{cases}
			\frac{1}{\min\{j \in \N : \omega(q_j) \ne \omega'(q_j)\}}, & \text{if}~\omega \ne \omega' \\
			0, & \text{if}~\omega = \omega',
		\end{cases}
	\end{equation*}
	where $q_1, q_2, \ldots$ is an enumeration of $\Q$.
	
	For $q \in \Q$, let $T^q : \Omega \to \Omega$ be the shift map $(T^q \omega)(s) = \omega(s+q)$, and let $X = \overline{\{T^q \chi : q \in \Q\}} \subseteq \Omega$.
	Then $(X, T)$ is a topological dynamical $\Q$-system, and $a = \chi$ is a transitive point by construction.
	Let $E_k = \{x \in X : x(0) = k\}$ for each $k \in \{1, \ldots, r\}$.
	Then the sets $E_1, \ldots, E_r$ form a partition of $X$ into clopen sets, and for $q \in \Q$ and $k \in \{1, \ldots, r\}$, we have
	\begin{equation*}
		T^q a \in E_k \iff (T^qa)(0) = k \iff \chi(q) = k.
	\end{equation*}
\end{proof}

\begin{theorem}[Furstenberg correspondence principle] \label{thm: correspondence}
	Let $A \subseteq \Q$ with $d^*(A) > 0$.
	Then there exists an ergodic $\Q$-system $(X, \mu, T)$, a transitive point $a \in X$, a F{\o}lner sequence $\Phi$ such that $a \in \gen(\mu, \Phi)$, and a clopen set $E \subseteq X$ with $\mu(E) > 0$ such that $A = \{q \in \Q : T^q a \in E\}$.
\end{theorem}

Theorem \ref{thm: correspondence} can be proved similarly to the proof of Theorem \ref{thm: topological correspondence}, with the added step that one must construct a $T$-invariant ergodic measure $\mu$ on the space $X$.
A full proof of Theorem \ref{thm: correspondence} (in the wider generality of amenable groups) appears in \cite[Theorem 2.15]{cm}.

\bigskip

In order to complete the translation of Theorems \ref{thm: density P} and \ref{thm: coloring P} into dynamical language, we need an appropriate notion of dynamical progressions, which generalizes the notion of Erd\H{o}s progressions from \cite{kmrr2}.

\begin{definition} \label{defn: EFS}
	Let $(X, T)$ be a topological dynamical $\Q$-system, and let $P(x) \in \Q[x]$.
	A triple $(x_0, x_1, x_2) \in X^3$ is an \emph{Erd\H{o}s--Furstenberg--S\'{a}rk\"{o}zy $P$-progression} if there exists a sequence of distinct elements $(s_n)_{n \in \N}$ such that
	\begin{equation*}
		\lim_{n \to \infty} \left( T^{s_n} x_0, T^{P(s_n)} x_1 \right) = (x_1, x_2).
	\end{equation*}
\end{definition}

We will denote the set of Erd\H{o}s--Furstenberg--S\'{a}rk\"{o}zy $P$-progressions by $\EFS^P_X \subseteq X^3$.
For a point $x_0 \in X$, we let $\EFS^P_X(x_0) = \left\{ (x_1, x_2) \in X^2 : (x_0, x_1, x_2) \in \EFS^P_X \right\}$.
We have the following variant of \cite[Theorem 2.2]{kmrr2}.

\begin{proposition} \label{prop: EFS implies sumset}
	Let $(X, T)$ be a topological dynamical $\Q$-system, $x_0 \in X$, and $U, V \subseteq X$ open sets.
	If $\EFS^P_X(x_0) \cap (U \times V) \ne \es$, then there exists a sequence $(b_n)_{n \in \N}$ of distinct elements such that
	\begin{align*}
		\{b_i : i \in \N\} & \subseteq \{q \in \Q : T^q x_0 \in U\} \\
		\{P(b_i) + b_j : i < j\} & \subseteq \{q \in \Q : T^q x_0 \in V\}.
	\end{align*}
\end{proposition}

\begin{proof}
	Let $(x_1, x_2) \in U \times V$ with $(x_0, x_1, x_2) \in \EFS^P_X$.
	Let $(s_n)_{n \in \N}$ be a sequence such that
	\begin{equation*}
		\lim_{n \to \infty} \left( T^{s_n} x_0, T^{P(s_n)} x_1 \right) = (x_1, x_2).
	\end{equation*}
	By refining the sequence $(s_n)_{n \in \N}$, we may assume $(T^{s_n}x_0, T^{P(s_n)}x_1) \in U \times V$ for every $n \in \N$.
	
	We will construct an increasing sequence $(n_i)_{i \in \N}$ by induction and take $b_i = s_{n_i}$.
	Let $n_1 = 1$.
	Suppose we have chosen $n_1 < n_2 < \ldots < n_k$ such that the elements $b_i = s_{n_i}$ satisfy
	\begin{equation*}
		x_0 \in \bigcap_{1 \le i < j \le k} T^{-P(b_i) - b_j } V.
	\end{equation*}
	(Note that this condition is vacuously satisfied for $k = 1$.)
	The set $\bigcap_{1 \le i \le k} T^{-P(b_i)} V$ is an open neighborhood of $x_1$, so we may choose $n_{k+1} > n_k$ such that
	\begin{equation*}
		T^{s_{n_{k+1}}} x_0 \in \bigcap_{1 \le i \le k} T^{-P(b_i)} V.
	\end{equation*}
	Let $b_{k+1} = s_{n_{k+1}}$.
	By the induction hypothesis, this implies
	\begin{equation*}
		x_0 \in \bigcap_{1 \le i < j \le k+1} T^{-P(b_i) - b_j} V.
	\end{equation*}
	
	By construction, $T^{b_i}x_0 \in U$ for every $i \in \N$, and $T^{P(b_i) + b_j}x_0 \in V$ for $i < j$, so this completes the proof.
\end{proof}

We can thus rephrase our main results in dynamical terms.
First, we have the following reformulation of Theorem \ref{thm: density P}:

\begin{theorem} \label{thm: density dynamical}
	Let $P(x) \in \Q[x]$ such that both $P(x)$ and $P(x) + x$ are nonconstant.
	Let $(X, \mu, T)$ be an ergodic $\Q$-system, $a \in \gen(\mu, \Phi)$, and $E \subseteq X$ an open set with $\mu(E) > 0$.
	Then there exist $t \in \Q$ and points $x_1, x_2 \in E$ such that $(T^ta, x_1, x_2)$ is an Erd\H{o}s--Furstenberg--S\'{a}rk\"{o}zy $P$-progression.
\end{theorem}

\begin{proof}[Proof that Theorem \ref{thm: density dynamical}$\implies$Theorem \ref{thm: density P}]
	Suppose Theorem \ref{thm: density dynamical} holds.
	Let $P(x) \in \Q[x]$ such that both $P(x)$ and $P(x) + x$ are nonconstant, and let $A \subseteq \Q$ with $d^*(A) > 0$.
	By the Furstenberg correspondence principle, let $(X, \mu, T)$ be an ergodic $\Q$-system, $a \in X$ a transitive point, $\Phi$ a F{\o}lner sequence such that $a \in \gen(\mu, \Phi)$, and $E \subseteq X$ a clopen set with $\mu(E) > 0$ such that $A = \{q \in \Q : T^q a \in E\}$.
	Then by Theorem \ref{thm: density dynamical}, there exists $t \in \Q$ such that $\EFS^P_X(T^ta) \cap (E \times E) \ne \es$.
	Applying Proposition \ref{prop: EFS implies sumset}, we obtain a sequence $(b_n)_{n \in \N}$ of distinct rational numbers such that
	\begin{equation*}
		\{b_i, P(b_i) + b_j : i < j\} \subseteq \{q \in \Q : T^q(T^ta) \in E\} = A - t.
	\end{equation*}
\end{proof}

We similarly give a dynamical formulation of Theorem \ref{thm: coloring P}:

\begin{restatable}{theorem}{ColoringDynamical} \label{thm: coloring dynamical}
	Let $P(x) \in \Q[x]$ with $\deg{P} \ge 2$.
	Let $(X, T)$ be a topological dynamical $\Q$-system.
	Suppose $a \in X$ is a transitive point.
	If $E_1, \ldots, E_r \subseteq X$ is a finite open cover of $X$, then there exists $k \in \{1, \ldots, r\}$ with points $x_1, x_2 \in E_k$ such that $(a, x_1, x_2)$ is an Erd\H{o}s--Furstenberg--S\'{a}rk\"{o}zy $P$-progression.
\end{restatable}

\begin{proof}[Proof that Theorem \ref{thm: coloring dynamical}$\implies$Theorem \ref{thm: coloring P}]
	Suppose Theorem \ref{thm: coloring dynamical} holds.
	Let $P(x) \in \Q[x]$ be a nonconstant polynomial, and let $\chi : \Q \to \{1, \ldots, r\}$ be a finite coloring of $\Q$.
	Then by Theorem \ref{thm: topological correspondence}, let $(X, T)$ be a topological dynamical $\Q$-system, $a \in X$ a transitive point, and $X = \bigcup_{k=1}^r E_k$ a partition into clopen sets such that $T^q a \in E_{\chi(q)}$ for every $q \in \Q$.
	By Theorem \ref{thm: coloring dynamical}, there exists $k \in \{1, \ldots, r\}$ and points $(x_1, x_2) \in \EFS^P_X(a) \cap (E_k \times E_k)$.
	Hence, by Proposition \ref{prop: EFS implies sumset}, there exists an infinite sequence $(b_n)_{n \in \N}$ such that
	\begin{equation*}
		\{b_i, P(b_i) + b_j : i < j\} \subseteq \{q \in \Q : T^q a \in E_k\} = \{q \in \Q : \chi(q) = k\}.
	\end{equation*}
	That is, $\{b_i, P(b_i) + b_j : i < j\}$ is monochromatic of color $k$, so Theorem \ref{thm: coloring P} holds.
\end{proof}

The basic strategy for proving these dynamical statements is as follows.
First, we show that Erd\H{o}s--Furstenberg--S\'{a}rk\"{o}zy $P$-progressions are ``controlled'' in some sense by the Abramov factor of order $d = \deg{P}$.
The main tool here is a polynomial variant of the Wiener--Wintner theorem (Theorem \ref{thm: polynomial ww adelic}).
Using a structure theorem for Abramov systems (Theorem \ref{thm: Abramov structure}), we are able to define a measure on Erd\H{o}s--Furstenberg--S\'{a}rk\"{o}zy progressions starting at a fixed point in the Abramov factor of order $d$.
We then lift this measure to the original system and prove that it remains supported on the closure of the space of Erd\H{o}s--Furstenberg--S\'{a}rk\"{o}zy progressions.
From this point on, the proofs of Theorem \ref{thm: density dynamical} and Theorem \ref{thm: coloring dynamical} diverge, as we need to establish different properties of the measure.

It turns out that the proof of Theorem \ref{thm: density dynamical} is simpler, so we deal with it first.
The remaining step is to average the measures corresponding to the spaces $\EFS^P_X(T^ta)$ over $t \in \Q$ and show that this averaged measure give positive mass to the product set $E \times E$.
By introducing the average over $t$, we are able to obtain this positivity result by an application of a version of the Furstenberg--S\'{a}rk\"{o}zy theorem over $\Q$.

To prove Theorem \ref{thm: coloring dynamical}, we need to show that the measure supported on the closure of $\EFS^P_X(a)$ gives positive mass to $\bigcup_{k=1}^r (E_k \times E_k)$.
The inability to shift the point $a$ means that recurrence results from ergodic theory are no longer helpful.
However, we can compute polynomial orbits in Abramov systems rather explicitly in order to reduce to an expression that can be controlled by an application of Ramsey's theorem for hypergraphs.


\section{Characteristic factors for polynomial averages} \label{sec: characteristic factor}

The goal of this section is to show that polynomial ergodic averages related to Erd\H{o}s--Furstenberg--S\'{a}rk\"{o}zy progressions are ``controlled'' by certain structured factors (the Abramov factors, to be defined later in this section).


\subsection{Factors of measure-preserving systems}

A system $(Y, \nu, S)$ is a \emph{(measurable) factor} of another system $(X, \mu, T)$ if there exists a measurable map $\pi : X \to Y$ such that $\pi_*\mu = \nu$ and $\pi \circ T^q = S^q \circ \pi$ almost everywhere for every $q \in \Q$.
We will say that $\pi$ is a \emph{topological factor map} if $\pi$ is continuous and surjective and the identity $\pi \circ T^q = S^q \circ \pi$ holds everywhere.

Measurable factors of $(X, \mu, T)$ correspond to $\sigma$-algebras that are invariant under the action of $T$.
For this reason, while dealing with factors, it is often convenient to make explicit reference to $\sigma$-algebras.
We will denote the Borel $\sigma$-algebra on a space by calligraphic letters; for example, we use $\mX$ to denote the $\sigma$-algebra on $X$, $\mY$ to denote the $\sigma$-algebra on $Y$, and so on.
In the other direction, given a $T$-invariant $\sigma$-algebra $\mY \subseteq \mX$, we will denote the corresponding space by $Y$.
In general, such a set $Y$ is an abstract space about which we can say very little.
However, for the structured factors of interest, we will be able to produce a concrete realization of $Y$.

We caution the reader that in a standard terminological abuse, we will use the term \emph{factor} to refer both to the system $(Y, \nu, S)$ and to the corresponding sub-$\sigma$-algebra $\mY \subseteq \mX$, using whichever notion is more convenient for the specific problem at hand.

Each factor comes with a corresponding conditional expectation operator.
Given a sub-$\sigma$-algebra $\mY \subseteq \mX$, the \emph{conditional expection} $\E{\cdot}{\mY} : L^1(\mu) \to L^1(\mu)$ is the operator defined by the properties: for any $f \in L^1(\mu)$, 
\begin{itemize}
	\item	$\E{f}{\mY}$ is $\mY$-measurable, and
	\item	if $g \in L^{\infty}(\mu)$ is $\mY$-measurable, then
		\begin{equation*}
			\int_X fg~d\mu = \int_X \E{f}{\mY} \cdot g~d\mu.
		\end{equation*}
\end{itemize}
Restricting $\E{\cdot}{\mY}$ to $L^2(\mu)$ produces the orthogonal projection onto the subspace of $L^2(\mu)$ of $\mY$-measurable square-integrable functions.

Suppose $\pi : (X, \mX, \mu, T) \to (Y, \mY, \nu, S)$ is a measurable factor map.
(Here, we have included the Borel $\sigma$-algebras explicitly for clarity in the ensuing discussion.)
A function $f : X \to \C$ is measurable with respect to the sub-$\sigma$-algebra $\pi^{-1}(\mY)$ if and only if there is a measurable function $g : Y \to \C$ such that $f = g \circ \pi$.
We will write $\E{f}{Y}$ for the function on $Y$ determined by $\E{f}{Y} \circ \pi = \E{f}{\pi^{-1}(\mY)}$.


\subsection{A variation on characteristic factors}

Let us now give meaning to ``control'' by factors.
Let $(X, \mu, T)$ be an ergodic $\Q$-system and $a \in X$.
The space of progressions $\EFS^P_X(a)$ is intimately related to polynomial orbits of the form $(T^qa, T^{P(q)}x)_{q \in \Q}$ for $x \in X$ (see Definition \ref{defn: EFS}).
In order to understand such orbits from a measure-theoretic point of view, it is natural to study ergodic averages taking the shape
\begin{equation*}
	A_N(x) = \frac{1}{|\Phi_N|} \sum_{q \in \Phi_N} f(T^qa) \cdot g(T^{P(q)}x).
\end{equation*}
Here, we will assume (as we may in Theorem \ref{thm: density dynamical}) that $\Phi$ is a F{\o}lner sequence in $\Q$ for which $a \in \gen(\mu, \Phi)$.
To take advantage of genericity of $a$, we will also assume that $f$ is a continuous function.
It turns out that an appropriate and convenient space to work in is the space $L^2(\mu)$: we assume $g \in L^2(\mu)$ and study convergence of $A_N$ in $L^2$ norm.
What we will mean by factors ``controlling'' the average $A_N$ is a result of the following kind: there exists a pair of factors $(\mB_1, \mB_2)$ such that for any $\eps > 0$, there exists a decomposition $f = f_s + f_u$ as a sum of continuous functions and a decomposition $g = g_s + g_u$ as a sum of measurable functions satisfying:
\begin{itemize}
	\item	the ``structured'' functions $f_s$ and $g_s$ are measurable with respect to factors $\mB_1$ and $\mB_2$ respectively,
	\item	the ``uniform'' function $f_u$ is $\eps$-almost orthogonal to $\mB_1$, i.e.
        \begin{equation*}
            \norm{L^1(\mu)}{\E{f_u}{\mB_1}} < \eps,
        \end{equation*}
	\item	the ``uniform'' function $g_u$ is orthogonal to $\mB_2$, and
	\item	\begin{equation*}
			\limsup_{N \to \infty} \norm{L^2(\mu)}{\frac{1}{|\Phi_N|} \sum_{q \in \Phi_N} \left( f(T^qa) \cdot T^{P(q)}g - f_s(T^qa) \cdot T^{P(q)}g_s \right)} < \eps.
		\end{equation*}
\end{itemize}
This closely resembles the notion of \emph{characteristic factors} for multiple ergodic averages as introduced by Furstenberg and Weiss \cite{fw_polynomials}.
The difference is that here we must allow for a small error tolerance $\eps$ in order to ensure that the function $f_s$ is continuous.
Without continuity, we would be unable to say anything about the evaluation of the function $f_s$ along the specific orbit $(T^qa)_{q \in \Q}$.

We will carry out our argument in two steps, first identifying the factor $\mB_2$ and then using the structure of $g_s$ to more easily identify $\mB_1$.


\subsection{The van der Corput Lemma}

One of the main tools we will use for producing the factors $\mB_1$ and $\mB_2$ described in the previous subsection is the van der Corput differencing trick.

\begin{lemma}[van der Corput lemma, cf. {\cite[Lemma 2.1]{shalom}}] \label{lem: vdC}
	Let $\mH$ be a Hilbert space, and let $u : \Q \to \mH$ be a bounded sequence.
	Let $\Phi$ be a F{\o}lner sequence in $\Q$.
	Suppose that:
	\begin{itemize}
		\item	the limit
			\begin{equation*}
				z(r) = \lim_{N \to \infty} \frac{1}{|\Phi_N|} \sum_{q \in \Phi_N} \innprod{u(q+r)}{u(q)}
			\end{equation*}
			exists for every $r \in \Q$, and
		\item	there exists $K < \infty$ such that for any F{\o}lner sequence $\Psi$ in $\Q$,
			\begin{equation*}
				\limsup_{M \to \infty} \frac{1}{|\Psi_M|} \left| \sum_{r \in \Psi_M} z(r) \right| \le K.
			\end{equation*}
	\end{itemize}
	Then
	\begin{equation*}
		\limsup_{N \to \infty} \norm{}{\frac{1}{|\Phi_N|} \sum_{q \in \Phi_N} u(q)}^2 \le K.
	\end{equation*}
\end{lemma}


\subsection{The Kronecker factor is partially characteristic}

The first structured factor of importance for the present analysis is the Kronecker factor.
Let $(X, \mu, T)$ be an ergodic $\Q$-system.
The \emph{Kronecker factor} is the maximal factor of $(X, \mu, T)$ that is measurably isomorphic to a compact group rotational system, i.e. a system of the form $(Z, m_Z, R)$, where $Z$ is a compact abelian group with Haar measure $m_Z$ and $R^q z = z + \psi(q)$ for some homomorphism $\psi : \Q \to Z$.
There are many equivalent characterizations of the Kronecker factor, of which we reproduce a selection below (cf. \cite[Chapter 4]{hk_book}):

\begin{proposition} \label{prop: Kronecker}
	Let $(X, \mu, T)$ be an ergodic $\Q$-system, and let $f, g\in L^2(\mu)$.
	Consider the statements:
	\begin{enumerate}[(i)]
		\item	$f$ is measurable with respect to the Kronecker factor;
		\item	the orbit $\{T^q f : q \in \Q\}$ is pre-compact in $L^2(\mu)$;
		\item $f$ can be approximated arbitrarily well in $L^2(\mu)$ by linear combinations of eigenfunctions\footnote{We say that $f \in L^2(\mu)$ is an \emph{eigenfunction} if there is a group character $\chi : \Q \to \bbS^1$ such that for every $q \in \Q$, $T^qf = \chi(q)f$.};
		\item	$g$ is orthogonal to the Kronecker factor;
		\item	$g$ is orthogonal to every eigenfunction;
		\item	there exists a F{\o}lner sequence $\Phi$ such that
			\begin{equation*}
				\lim_{N \to \infty} \frac{1}{|\Phi_N|} \sum_{q \in \Phi_N} \left| \innprod{T^qg}{g} \right| = 0;
			\end{equation*}
		\item	for every F{\o}lner sequence $\Phi$, one has
			\begin{equation*}
				\lim_{N \to \infty} \frac{1}{|\Phi_N|} \sum_{q \in \Phi_N} \left| \innprod{T^qg}{g} \right| = 0.
			\end{equation*}
	\end{enumerate}
	Each of the statements (i)--(iii) is equivalent, and the statements (iv)--(vii) are equivalent to each other.
\end{proposition}

From now on, we will denote the Kronecker factor as a sub-$\sigma$-algebra by $\mZ$ and as a system by $(Z, m_Z, R)$.

\begin{proposition} \label{prop: Kronecker partially characteristic}
	Let $(X, \mu, T)$ be an ergodic $\Q$-system and $a \in \gen(\mu, \Phi)$.
	Let $P(x) \in \Q[x]$ be a nonconstant polynomial.
	For any $f \in C(X)$ and $g \in L^2(\mu)$, if $\E{g}{\mZ} = 0$, then
	\begin{equation*}
		\lim_{N \to \infty} \frac{1}{|\Phi_N|} \sum_{q \in \Phi_N} f(T^qa) \cdot T^{P(q)}g = 0.
	\end{equation*}
\end{proposition}

\begin{proof}
	We prove the claim by induction on the degree of $P$.
	First, suppose $\deg{P} = 1$, and write $P(q) = cq + d$ with $c, d \in \Q$, $c \ne 0$.
	Let $u(q) = f(T^qa) \cdot T^{cq+d}g \in L^2(\mu)$.
	Then
	\begin{align*}
		\innprod{u(q+r)}{u(q)} & = \overline{f(T^qa)} \cdot f(T^{q+r}a) \int_X T^{cq+d}\overline{g} \cdot T^{c(q+r)+d}g~d\mu \\
		 & = (\overline{f} \cdot T^rf)(T^qa) \int_X \overline{g} \cdot T^{cr}g~d\mu,
	\end{align*}
	so
	\begin{equation*}
		z(r) = \lim_{N \to \infty} \frac{1}{|\Phi_N|} \sum_{q \in \Phi_N} \innprod{u(q+r)}{u(q)} = \left( \int_X \overline{f} \cdot T^rf~d\mu \right) \left( \int_X \overline{g} \cdot T^{cr}g~d\mu \right).
	\end{equation*}
	Therefore, by the assumption $\E{g}{\mZ} = 0$, given any F{\o}lner sequence $\Psi$ in $\Q$, we have (see (vii) in Proposition \ref{prop: Kronecker})
	\begin{equation*}
		\limsup_{M \to \infty} \frac{1}{|\Psi_M|} \sum_{r \in \Psi_M} |z(r)| \le \norm{u}{f}^2 \cdot \lim_{M \to \infty} \frac{1}{|\Psi_M|} \sum_{r \in \Psi_M} \left| \int_X \overline{g} \cdot T^{cr}g~d\mu \right| = 0.
	\end{equation*}
	Thus, by Lemma \ref{lem: vdC}, we conclude
	\begin{equation*}
		\lim_{N \to \infty} \frac{1}{|\Phi_N|} \sum_{q \in \Phi_N} f(T^qa) \cdot T^{P(q)}g = \lim_{N \to \infty} \frac{1}{|\Phi_N|} \sum_{q \in \Phi_N} u(q) = 0.
	\end{equation*}
	
	Now suppose the proposition holds for polynomials of degree at most $d$, and assume $\deg{P} = d+1$.
	As before, let $u(q) = f(T^qa) \cdot T^{P(q)}g \in L^2(\mu)$.
	Then
	\begin{align*}
		\innprod{u(q+r)}{u(q)} & = \overline{f(T^qa)} \cdot f(T^{q+r}a) \int_X T^{P(q)}\overline{g} \cdot T^{P(q+r)}g~d\mu \\
		 & = (\overline{f} \cdot T^rf)(T^qa) \int_X \overline{g} \cdot T^{P(q+r)-P(q)}g~d\mu.
	\end{align*}
	For fixed $r \in \Q \setminus \{0\}$, the polynomial $Q_r(q) = P(q+r) - P(q)$ has degree at most $d$.
	Therefore, by the induction hypothesis,
	\begin{equation*}
		\lim_{N \to \infty} \frac{1}{|\Phi_N|} \sum_{q \in \Phi_N} (\overline{f} \cdot T^rf)(T^qa) \cdot T^{Q_r(q)} g = 0.
	\end{equation*}
	Hence,
	\begin{equation*}
		\lim_{N \to \infty} \frac{1}{|\Phi_N|} \sum_{q \in \Phi_N} \innprod{u(q+r)}{u(q)} = 0
	\end{equation*}
	for every $r \in \Q \setminus \{0\}$.
	By Lemma \ref{lem: vdC}, it follows that
	\begin{equation*}
		\lim_{N \to \infty} \frac{1}{|\Phi_N|} \sum_{q \in \Phi_N} f(T^qa) \cdot T^{P(q)}g = \lim_{N \to \infty} \frac{1}{|\Phi_N|} \sum_{q \in \Phi_N} u(q) = 0.
	\end{equation*}
\end{proof}


\subsection{The Abramov factor is partially characteristic} \label{sec: Abramov characteristic}

We now wish to describe a factor that will allow us also to replace $f$ by a structured function when dealing with the limiting behavior of averages of the form
\begin{equation} \label{eq: polynomial average}
	\frac{1}{|\Phi_N|} \sum_{q \in \Phi_N} f(T^qa) \cdot T^{P(q)}g.
\end{equation}
Since the Kronecker factor is generated by eigenfunctions (see item (iii) in Proposition \ref{prop: Kronecker}), Proposition \ref{prop: Kronecker partially characteristic} shows that the key to understanding the behavior of \eqref{eq: polynomial average} is to analyze ``polynomially-twisted'' ergodic averages
\begin{equation} \label{eq: polynomial twist}
	\frac{1}{|\Phi_N|} \sum_{q \in \Phi_N} f(T^qa) \cdot \chi(P(q))
\end{equation}
for characters $\chi : \Q \to \bbS^1$.

Let us begin by introducing some notation for describing characters on $\Q$.
For each prime number $p$, we define the $p$-adic metric on $\Q$ by $d_p(x,y) = p^{-n}$ if $x - y = p^n \cdot \frac{a}{b}$ with $n \in \Z$ and $a, b \in \Z$ coprime to $p$.
The completion of $\Q$ with respect to $d_p$ is the space of $p$-adic numbers $\Q_p$, which may be expressed as formal series $\sum_{j=N}^{\infty} c_j p^j$ with $N \in \Z$ and $c_j \in \{0, 1, \ldots, p-1\}$.
The $p$-adic integers are the subspace $\Z_p = \{\sum_{j=0}^{\infty} c_j p^j : c_j \in \{0, 1, \ldots, p-1\}\} \subseteq \Q_p$.
The ring of adeles is defined by
\begin{equation*}
	\A = \left\{ (a_{\infty}, a_2, a_3, \ldots) \in \R \times \prod_p \Q_p : a_p \in \Z_p~\text{for all but finitely many}~p \right\}.
\end{equation*}
We embed $\Q$ as a subspace of $\A$ by identifying $q \in \Q$ with the element $(q, q, q, \ldots) \in \A$.
The ring of adeles $\A$ is a topological space with the subspace topology inherited from the product space $\R \times \prod_p \Q_p$.
One can show that $\Q \subseteq \A$ is discrete and co-compact, so $\A/\Q$ is a compact group.
In fact, $\A/\Q$ is isomorphic to the Pontryagin dual $\hat{\Q}$.
We let $e_{\Q} : \A/\Q \to \bbS^1$ such that each character $\chi \in \hat{\Q}$ is of the form $\chi(q) = e_{\Q}(qx)$ for some $x \in \A/\Q$.
(For a concrete description of the function $e_{\Q}$, see, e.g., \cite[Section 5]{bm_vdC}.)
We may thus express the averages \eqref{eq: polynomial twist} in the more concrete form
\begin{equation} \label{eq: polynomial twist adelic}
	\frac{1}{|\Phi_N|} \sum_{q \in \Phi_N} f(T^qa) \cdot e_{\Q}(P(q))
\end{equation}
with $P$ a polynomial taking values in $\A/\Q$.

As may be expected, averages of the form \eqref{eq: polynomial twist adelic} are controlled in a certain sense by polynomial structures in the system $(X, \mu, T)$.
Let $(X, \mu, T)$ be an ergodic $\Q$-system.
For a function $g \in L^{\infty}(\mu)$ and an element $q \in \Q$, we define the multiplicative derivative $\Delta_q g = \overline{g} \cdot T^q g$.
Given $q_1, \ldots, q_k \in \Q$, we define $\Delta^k_{q_1, \ldots, q_k}$ inductively by $\Delta^k_{q_1, \ldots, q_k} g = \Delta_{q_k} \left( \Delta^{k-1}_{q_1, \ldots, q_{k-1}} g \right)$.
We say that $g \in L^{\infty}(\mu)$ is a \emph{phase polynomial of degree at most $k$} if for each $(q_1, \ldots, q_{k+1}) \in \Q^{k+1}$,
\begin{equation*}
	\Delta^{k+1}_{q_1, \ldots, q_{k+1}} g = 1~\text{a.e.}
\end{equation*}
Let $\mE_k(T)$ be the set of phase polynomials of degree at most $k$.
The \emph{Abramov factor of order $k$} is the factor $\mA_k$ for which $L^2(X, \mA_k, \mu)$ is the closure of the linear span of $\mE_k(T)$, and we denote the factor as a system by $(A_k, m_{A_k}, T_k)$.
(We have thus far reserved the letter $m$ for the Haar measure on a compact group.
We will show in Theorem \ref{thm: Abramov structure} that (up to a measurable isomorphism) $A_k$ is indeed a compact abelian group, so this notation is consistent with our usage.)
Note that $\mE_1(T)$ is the set of eigenfunctions of $(X, \mu, T)$, so the order 1 Abramov factor $\mA_1$ is equal to the Kronecker factor $\mZ$ by item (iii) in Proposition \ref{prop: Kronecker}.)

The next theorem relates polynomially twisted averages to Abramov factors under a technical assumption on the eigenvalues of the system.
Given an ergodic $\Q$-system, we denote by $\sigma(T)$ the group of eigenvalues
\begin{equation*}
	\sigma(T) = \left\{ t \in \A/\Q : \exists g \in L^2(\mu), g \ne 0,~\text{such that}~\forall q \in \Q, T^q g = e_{\Q}(qt) g \right\}
\end{equation*}
We then let $\sigma'(T) = \{qt : q \in \Q, t \in \sigma(T)\}$ be the $\Q$-vector space spanned by $\sigma(T)$.
The sets $\sigma(T)$ and $\sigma'(T)$ are countable, since $L^2(\mu)$ is a separable Hilbert space.

\begin{definition}
	We say that an ergodic $\Q$-system $(X, \mu, T)$ has \emph{divisible discrete spectrum} if the discrete spectrum $\sigma(T)$ is a divisible group.
\end{definition}

Equivalently, $T$ has divisible discrete spectrum if and only if $\sigma(T) = \sigma'(T)$.
We will see in Lemma \ref{lem: divisible} that every ergodic system has an ergodic extension with divisible discrete spectrum, so the condition of having divisible discrete spectrum is not overly stringent.

\begin{restatable}{theorem}{WW} \label{thm: polynomial ww adelic}
	Let $(X, \mu, T)$ be an ergodic $\Q$-system with divisible discrete spectrum, and let $d \in \N$.
	Let $\Phi$ be a tempered\footnote{A F{\o}lner sequence $\Phi$ is \emph{tempered} if
	\begin{equation*}
		\sup_{N \in \N} \frac{\left| \bigcup_{j=1}^{N-1} (\Phi_N - \Phi_j) \right|}{|\Phi_N|} < \infty.
	\end{equation*}
	Tempered F{\o}lner sequences play an important role in pointwise ergodic theory, as demonstrated in \cite{lindenstrauss}.} F{\o}lner sequence in $\Q$.
	If $f \in L^1(\mu)$ and $\E{f}{\mA_d} = 0$, then there exists a set $X_0 \subseteq X$ with $\mu(X_0) = 1$ such that for any polynomial $P(x) \in \A[x]$ with $\deg{P} \le d$ and any $x \in X_0$,
	\begin{equation*}
		\lim_{N \to \infty} \frac{1}{|\Phi_N|} \sum_{q \in \Phi_N} f(T^qx) \cdot e_{\Q} \left( P(q) \right) = 0.
	\end{equation*}
\end{restatable}

We reserve the proof of Theorem \ref{thm: polynomial ww adelic} until Section \ref{sec: ww}.

\begin{remark} \label{rem: ww kronecker}
	The $d=1$ case of Theorem \ref{thm: polynomial ww adelic} follows from \cite[Lemma 1.1 and Corollary 3.2]{zk}, and in this case it is not necessary to assume that the discrete spectrum is divisible.
\end{remark}

Theorem \ref{thm: polynomial ww adelic} does not immediately address how one may control the averages \eqref{eq: polynomial twist adelic}, since the F{\o}lner sequence in Theorem \ref{thm: polynomial ww adelic} is required to be tempered and the set $X_0$ may not contain every generic point.
However, we can overcome these complications with the following consequence of Theorem \ref{thm: polynomial ww adelic}:

\begin{theorem} \label{thm: Abramov partially characteristic}
	Let $(X, \mu, T)$ be an ergodic $\Q$-system with divisible discrete spectrum and $a \in \gen(\mu, \Phi)$.
	Let $P(x) \in \Q[x]$ be a nonconstant polynomial of degree $d$, and let $\alpha \in \A/\Q$.
	If $f \in C(X)$, then
	\begin{equation} \label{eq: Abramov control}
		\limsup_{N \to \infty} \left| \frac{1}{|\Phi_N|} \sum_{q \in \Phi_N} f(T^qa) \cdot e_{\Q}(P(q)\alpha) \right| \le \norm{1}{\E{f}{\mA_d}}.
	\end{equation}
\end{theorem}

\begin{proof}
	If $\alpha = 0$, then by the assumption $a \in \gen(\mu,\Phi)$, the left-hand side of \eqref{eq: Abramov control} is equal to $\left| \int_X f~d\mu \right|$, and in this case the inequality is trivial.
	We will therefore assume $\alpha \ne 0$.
	
	We define an action of $\Q$ on $(\A/\Q)^d$ by
	\begin{equation*}
		S^q(v_1, \ldots, v_d) = \left( v_1 + q\alpha, v_2 + 2qv_1 + q^2\alpha, \ldots, v_d + \sum_{j=1}^{d-1} \binom{d}{j} q^{d-j} v_j + q^d \alpha \right).
	\end{equation*}
	One can show that $((\A/\Q)^d, S)$ is minimal and uniquely ergodic.
	The argument appears in \cite[Lemma 1.25 and Theorem 3.12]{furstenberg_book} for a corresponding action of $\Z$ on $\T^d$ and is easily adapted to the $\Q$-action on $(\A/\Q)^d$.
	Writing $P(x) = \sum_{j=0}^d c_j x^j$, we have
	\begin{equation*}
		e_{\Q}(P(q)\alpha) = g \left( S^q(\bm{0}) \right)
	\end{equation*}
	for the function $g(v_1, \ldots, v_d) = e_{\Q} \left( c_0\alpha + \sum_{j=1}^d c_j v_j \right)$.
	
	Now let $(N_k)_{k \in \N}$ be a sequence such that
	\begin{equation*}
		\limsup_{N \to \infty} \left| \frac{1}{|\Phi_N|} \sum_{q \in \Phi_N} f(T^qa) \cdot e_{\Q}(P(q)\alpha) \right| = \left| \lim_{k \to \infty} \frac{1}{|\Phi_{N_k}|} \sum_{q \in \Phi_{N_k}} f(T^qa) \cdot e_{\Q}(P(q)\alpha) \right|.
	\end{equation*}
	Refining to a further subsequence if necessary, we may assume that the weak* limit
	\begin{equation*}
		\lambda = \lim_{k \to \infty} \frac{1}{|\Phi_{N_k}|} \sum_{q \in \Phi_{N_k}} \delta_{T^qa} \times \delta_{S^q(\bm{0})}
	\end{equation*}
	exists.
	Since $a \in \gen(\mu, \Phi)$ by assumption and $\bm{0} \in \gen(m_{\A/\Q}^d, \Phi)$ by unique ergodicity of the system $((\A/\Q)^d, S)$, the measure $\lambda$ is a joining of the systems $(X, \mu, T)$ and $((\A/\Q)^d, m_{\A/\Q}^d, S)$.
	
	\bigskip
	
	\noindent \textbf{Claim.} For $h \in L^2(\mu)$ and $k \in L^2((\A/\Q)^d)$,
	\begin{equation} \label{eq: joining comes from Abramov}
		\int_{X \times (\A/\Q)^d} h \otimes k~d\lambda = \int_{X \times (\A/\Q)^d} \E{h}{\mA_d} \otimes k~d\lambda.
	\end{equation}
	
	\begin{proof}[Proof of Claim] \renewcommand\qedsymbol{$\blacksquare$}
		By linearity, it suffices to prove the identity when $\E{h}{\mA_d} = 0$ and $k(\bm{v}) = e_{\Q} \left( \sum_{j=1}^d r_j v_j \right)$ for some $\bm{r} \in \Q^d$.
		Moreover, if $\lambda = \int_{\Omega} \lambda_{\omega}~d\rho(\omega)$ is the ergodic decomposition of $\lambda$, then each of the ergodic components $\lambda_{\omega}$ is an ergodic joining of $(X, \mu, T)$ and $((\A/\Q)^d, m_{\A/\Q}^d, S)$ (see, e.g., \cite[Proposition 1]{dlR_joinings}), and it suffices to prove \eqref{eq: joining comes from Abramov} with $\lambda_{\omega}$ in place of $\lambda$.

		Fix a tempered F{\o}lner sequence $\Psi$ in $\Q$ (such a sequence exists by \cite[Proposition 1.4]{lindenstrauss}).
		By Theorem \ref{thm: polynomial ww adelic}, let $X_0 \subseteq X$ with $\mu(X_0) = 1$ such that
		\begin{equation*}
			\lim_{N \to \infty} \frac{1}{|\Psi_N|} \sum_{q \in \Psi_N} h(T^qx) \cdot e_{\Q} \left( Q(q) \right) = 0
		\end{equation*}
		for every $Q(x) \in \A[x]$ of degree at most $d$ and every $x \in X_0$.
		By Lindenstrauss's pointwise ergodic theorem (see \cite[Theorem 1.2]{lindenstrauss}), let $Y \subseteq X \times (\A/\Q)^d$ with $\lambda_{\omega}(Y) = 1$ such that
		\begin{equation*}
			\lim_{N \to \infty} \frac{1}{|\Psi_N|} \sum_{q \in \Psi_N} h(T^qx) \cdot k(S^qy) = \int_{X \times (\A/\Q)^d} h \otimes k~d\lambda_{\omega}
		\end{equation*}
		for every $(x,y) \in Y$.
		
		Let $(x, y) \in Y \cap (X_0 \times (\A/\Q)^d)$.
		Note that
		\begin{equation*}
			k(S^qy) = e_{\Q} \left( Q_y(q) \right)
		\end{equation*}
		for some polynomial $Q_y$ of degree at most $d$ depending on the point $y$.
		Thus,
		\begin{align*}
			\int_{X \times (\A/\Q)^d} h \otimes k~d\lambda_{\omega} & = \lim_{N \to \infty} \frac{1}{|\Psi_N|} \sum_{q \in \Psi_N} h(T^qx) \cdot k(S^qy) \\
             & = \lim_{N \to \infty} \frac{1}{|\Psi_N|} \sum_{q \in \Psi_N} h(T^qx) \cdot e_{\Q} \left( Q_y(q) \right) = 0.
		\end{align*}
		This proves the claim.
	\end{proof}
	
	Now, using the claim, we have
	\begin{align*}
		\limsup_{N \to \infty} \left| \frac{1}{|\Phi_N|} \sum_{q \in \Phi_N} f(T^qa) \cdot e_{\Q}(P(q)\alpha) \right| & = \left| \int_{X \times (\A/\Q)^d} f \otimes g~d\lambda \right| \\
         & = \left| \int_{X \times (\A/\Q)^d} \E{f}{\mA_d} \otimes g~d\lambda \right|,
	\end{align*}
	and the inequality \eqref{eq: Abramov control} follows by an application of H\"{o}lder's inequality.
\end{proof}


\subsection{An approximation theorem for averages in systems with topological Abramov factors}

We now have nearly everything in place to obtain the desired type of ``control'' described at the beginning of the section.
The final missing piece is a small technical assumption on the system $(X, \mu, T)$.
We say that $(X, \mu, T)$ \emph{has topological Abramov factors} if there are topological (continuous) factor maps $\pi_d : X \to A_d$ for every $d \in \N$.
We will see later on that this is not overly restrictive: every system has an extension that has topological Abramov factors as well as having divisible discrete spectrum (see Proposition \ref{prop: divisible and topological factors} below).
Having topological Abramov factors ensures that continuous functions on $A_d$ lift to continuous functions on $X$ via the factor map $\pi_d$.

We collect the main results of this section in the following theorem:

\begin{theorem} \label{thm: characteristic factor}
	Let $(X, \mu, T)$ be an ergodic $\Q$-system with divisible discrete spectrum, and let $a \in \gen(\mu, \Phi)$.
	Assume that $(X, \mu, T)$ has topological Abramov factors.
	Let $P(x) \in \Q[x]$ be a nonconstant polynomial of degree $d$.
	Let $f \in C(X)$ and $g \in L^2(\mu)$.
	Then given $\eps > 0$, there exists a decomposition $f = f_s \circ \pi_d + f_u$ with $f_s \in C(A_d)$ and $f_u \in C(X)$ such that $\norm{L^1(\mu)}{\E{f_u}{\mA_d}} < \eps$ and
	\begin{equation} \label{eq: approximate characteristic factor}
		\limsup_{N \to \infty} \norm{L^2(\mu)}{\frac{1}{|\Phi_N|} \sum_{q \in \Phi_N} \left( f(T^qa) \cdot T^{P(q)}g - f_s(T_d^q\pi_d(a)) \cdot T^{P(q)}\E{g}{\mZ} \right)} < \eps.
	\end{equation}
\end{theorem}

\begin{proof}
	We will assume $\norm{L^2(\mu)}{g} = 1$.
	To deduce the general case, one can scale and apply the theorem with $\eps' = \norm{L^2(\mu)}{g}^{-1} \cdot \eps$.
	
	Since the space of continuous functions $C(A_d)$ is dense in $L^1(m_{A_d})$, we may find $f_s \in C(A_d)$ such that $\norm{L^1(m_{A_d})}{\E{f}{A_d} - f_s} < \eps$.
	Let $f_u = f - f_s \circ \pi_d$.
	The factor map $\pi_d$ is continuous by assumption, so $f_u$ is continuous.
	Moreover, $\norm{L^1(\mu)}{\E{f_u}{\mA_d}} < \eps$.
	It remains to check \eqref{eq: approximate characteristic factor}.
	
	For functions $h \in C(X)$ and $k \in L^2(\mu)$, let
	\begin{equation*}
		A_N(h,k; x) = \frac{1}{|\Phi_N|} \sum_{q \in \Phi_N} h(T^qa) \cdot k(T^{P(q)}x).
	\end{equation*}
	We want to show
	\begin{equation*}
		\limsup_{N \to \infty} \norm{L^2(\mu)}{A_N(f,g) - A_N(f_s, \E{g}{\mZ})} < \eps.
	\end{equation*}
	By Proposition \ref{prop: Kronecker partially characteristic}, we have
	\begin{equation*}
		\limsup_{N \to \infty} \norm{L^2(\mu)}{A_N(f,g) - A_N(f, \E{g}{\mZ})} = 0,
	\end{equation*}
	so it is enough to show that
	\begin{equation*}
		\limsup_{N \to \infty} \norm{L^2(\mu)}{\underbrace{A_N(f, \E{g}{\mZ}) - A_N(f_s, \E{g}{\mZ})}_{A_N(f_u, \E{g}{\mZ})}} < \eps.
	\end{equation*}
	
	Let $	\sigma \subseteq \A/\Q$ be the countable set of eigenvalues of $T$, and let $h_{\alpha} : X \to \bbS^1$ be an eigenfunction with eigenvalue $\alpha$ for each $\alpha \in \sigma$, i.e. $T^q h_{\alpha} = e_{\Q}(q\alpha) \cdot h_{\alpha}$.
	By item (iii) in Proposition \ref{prop: Kronecker} and orthogonality of eigenfunctions, we may write $g = \sum_{\alpha \in \sigma} c_{\alpha} h_{\alpha}$ for some coefficients $c_{\alpha}$ satisfying $\sum_{\alpha \in \sigma} |c_{\alpha}|^2 = 1$.
	Then
	\begin{equation*}
		A_N(f_u, \E{g}{\mZ}) = \frac{1}{|\Phi_N|} \sum_{q \in \Phi_N} f_u(T^qa) \cdot \sum_{\alpha \in \sigma} c_{\alpha} e_{\Q}(P(q)\alpha) h_{\alpha},
	\end{equation*}
	so by Theorem \ref{thm: Abramov partially characteristic} and Fatou's lemma,
	\begin{multline*}
		\limsup_{N \to \infty} \norm{L^2(\mu)}{A_N(f_u, \E{g}{\mZ})}^2 \\
         = \limsup_{N \to \infty} \sum_{\alpha \in \sigma} |c_{\alpha}|^2 \left| \frac{1}{|\Phi_N|} \sum_{q \in \Phi_N} f_u(T^qa) \cdot e_{\Q}(P(q)\alpha) \right|^2 < \eps^2.
	\end{multline*}
\end{proof}


\section{Equidistribution of polynomial orbits in Abramov systems}

Theorem \ref{thm: characteristic factor} shows that the Abramov factors of $(X, \mu, T)$ play a fundamental role in the behavior of polynomial orbits in $X$ and hence in the structure of the space of Erd\H{o}s--Furstenberg--S\'{a}rk\"{o}zy progressions.
The goal of this section is to show that polynomial orbits in Abramov systems have a simple algebraic description.
This will later allow us to construct a measure supported on Erd\H{o}s--Furstenberg--S\'{a}rk\"{o}zy progressions that we can lift back to the original space $X$.

The main tool that allows us to get a handle on polynomial orbits in Abramov systems is the following structure theorem.

\begin{theorem} \label{thm: Abramov structure}
	Suppose $(X, \mu, T)$ is an ergodic Abramov $\Q$-system of order at most $k$.
	Then there exists a compact $\Q$-vector space\footnote{The notion of a compact $\Q$-vector space may sound rather exotic, but we have in fact already encountered an example of one in Section \ref{sec: characteristic factor}: the space $\A/\Q$.
	Indeed, $\A/\Q$ is a compact group that is both torsion-free (since $\Q$ is divisible) and divisible (since $\Q$ is torsion-free), so there is a natural multiplicative action of $\Q$ on $\A/\Q$ that makes $\A/\Q$ into a vector space over $\Q$.} $V$ and an element $\alpha \in V$ such that $(X, \mu, T)$ is a factor of the system $(V^k, m_V^k, S)$, where
	\begin{multline*}
		S^q(v_1, \ldots, v_k) \\
         = \left( v_1 + q\alpha, v_2 + qv_1 + \binom{q}{2} \alpha, \ldots, v_k + q v_{k-1} + \ldots + \binom{q}{k-1} v_1 + \binom{q}{k} \alpha \right).
	\end{multline*}
	Moreover, up to replacing $(X, \mu, T)$ by a measurably isomorphic system $(\tilde{X}, \tilde{\mu}, \tilde{T})$, we may assume that the factor map is continuous.
\end{theorem}

A similar structure theorem was established by Abramov \cite{abramov} for totally ergodic $\Z$-systems.
We give a proof of Theorem \ref{thm: Abramov structure} in Section \ref{sec: Abramov}.

Observe that for a system of the form given in Theorem \ref{thm: Abramov structure}, the orbit of any point along any polynomial sequence will have polynomial coordinates.
Consequently, we can prove an equidistribution theorem for orbits in Abramov systems, which will be essential for defining a measure on the space of Erd\H{o}s--Furstenberg--S\'{a}rk\"{o}zy progressions.
Before giving a statement of the relevant equidistribution results, we recall some necessary terminology and notation.
Given a function $v : \Q \to V$ taking values in a topological vector space $V$, we say that the \emph{uniform Ces\`{a}ro limit} of $v$ is equal to $L \in V$ if
\begin{equation*}
	\lim_{N \to \infty} \frac{1}{|\Phi_N|} \sum_{q \in \Phi_N} v(q) = L
\end{equation*}
for every F{\o}lner sequence $\Phi$.
If the uniform Ces\`{a}ro limit exists, we denote it by $\UClim_{q \in \Q} v(q)$.
A function $x : \Q \to X$ taking values in a compact metric space $X$ is said to be \emph{well-distributed} with respect to a Borel probability measure $\mu$ if $\UClim_{q \in \Q} \delta_{x(q)} = \mu$ in the weak* topology.
That is, for every F{\o}lner sequence $\Phi$ and every continuous function $f \in C(X)$,
\begin{equation*}
	\lim_{N \to \infty} \frac{1}{|\Phi_N|} \sum_{q \in \Phi_N} f(x(q)) = \int_X f~d\mu.
\end{equation*}
When $X$ is a group and $\mu$ is the Haar measure on a subgroup $H \subseteq X$, we will simply say that $x$ is well-distributed in $H$.

We have the following variant of Weyl's equidistribution theorem for polynomials over $\Q$:

\begin{proposition} \label{prop: polynomial equidistribution}
	Let $V$ be a compact $\Q$-vector space, and let $k \in \N$.
	Suppose $\alpha_1, \ldots, \alpha_k \in V$, and let $H_j \subseteq V$ be the closed $\Q$-vector subspace generated by $\alpha_j$ for each $j \in \{1, \ldots, k\}$.
	Given any family $P_1(x), \ldots, P_k(x) \in \Q[x]$ of nonconstant linearly independent polynomials with zero constant term, the $\Q$-sequence
	\begin{equation*}
		\left( P_1(q) \alpha_1, P_2(q) \alpha_2, \ldots, P_k(q) \alpha_k \right)
	\end{equation*}
	is well-distributed in the closed subspace $H_1 \times H_2 \times \ldots \times H_k \subseteq V^k$.
\end{proposition}

\begin{proof}
	First we reduce to the case $P_j(x) = x^j$.
	Suppose Proposition \ref{prop: polynomial equidistribution} holds for all $k \in \N$ for the polynomials $P_j(x) = x^j$.
	Now for a family $P_1, \ldots, P_k$ of nonconstant linearly independent polynomials with zero constant term, let us write $P_i(x) = \sum_{j=1}^d a_{ij} x^j$ with $d = \max_{1 \le i \le k} \deg{P_i} \ge k$.
	The condition that $P_1, \ldots, P_k$ are linearly independent means that the matrix $A = (a_{ij})_{1 \le i \le k, 1 \le j \le d}$ has trivial cokernel, so $A(\Q^d) = \Q^k$ by the fundamental theorem of linear algebra.
	
	Let $V' = V^k$ and $\alpha' = (\alpha_1, \ldots, \alpha_k) \in V'$.
	Let $H' \subseteq V'$ be the closed vector space generated by $\alpha'$.
	By assumption, the sequence
	\begin{equation*}
		s(q) = (q \alpha', q^2 \alpha', \ldots, q^d \alpha')
	\end{equation*}
	is well-distributed in $(H')^d$.
	Let $\pi : (V')^d \to V^k$ be the linear map
	\begin{equation*}
		\pi(\bm{v}^1, \ldots, \bm{v}^d) = \left( \sum_{j=1}^d a_{ij} v^j_i \right)_{i=1}^k.
	\end{equation*}
	Then
	\begin{equation*}
		\left( P_1(q) \alpha_1, P_2(q) \alpha_2, \ldots, P_k(q) \alpha_k \right) = \pi(s(q))
	\end{equation*}
	is well-distributed in $\pi((H')^d)$.
	It remains to check that $\pi((H')^d) = H_1 \times \ldots \times H_k$.
	On the one hand, if $(q_1, \ldots, q_d) \in \Q^d$, then
	\begin{equation*}
		\pi(q_1\alpha', \ldots, q_d\alpha') = \left( \left( \sum_{j=1}^d a_{ij}q_j \right) \alpha_i \right)_{i=1}^k \in H_1 \times \ldots \times H_k,
	\end{equation*}
	so by continuity of $\pi$, we have $\pi((H')^d) \subseteq H_1 \times \ldots \times H_k$.
	On the other hand, if $(q_1, \ldots, q_k) \in \Q^k$, then there exist $(r_1, \ldots, r_d) \in \Q^d$ such that
	\begin{equation*}
		q_i = \sum_{j=1}^d a_{ij} r_j
	\end{equation*}
	for every $i \in \{1, \ldots, k\}$.
	Hence, $(q_1\alpha_1, \ldots, q_k\alpha_k) = \pi(r_1\alpha', \ldots, r_d\alpha') \in \pi((H')^d)$.
	Thus, since $\pi((H')^d)$ is closed, we have $H_1 \times \ldots \times H_k \subseteq \pi((H')^d)$.
	
	\bigskip
	
	Now we want to show: if $f_1, \ldots, f_d : V \to \C$ are continuous functions, then
	\begin{equation} \label{eq: equidistribution definition}
		\UClim_{q \in \Q} \prod_{j=1}^d f_j(q^j\alpha_j) = \prod_{j=1}^d \int_{H_j} f_j~dm_{H_j}.
	\end{equation}
	Approximating the functions $f_j$ by linear combinations of group characters, it suffices to prove \eqref{eq: equidistribution definition} for $f_j = \chi_j \in \hat{V}$.
	Note that if $\chi_j \in \hat{V}$, then $\lambda_j(q) = \chi_j(q\alpha_j)$ is a group character on $\Q$, and $\lambda_j$ is the trivial character if and only if $\chi_j \in H_j^{\perp}$.
	Therefore, writing $\lambda_j(q) = e_{\Q}(q\beta_j)$ and disregarding any high degree trivial terms, it suffices to prove: if $\beta_1, \ldots, \beta_d \in \A/\Q$ and $\beta_d \ne 0$, then
	\begin{equation*}
		\UClim_{q \in \Q} e_{\Q} \left( \sum_{j=1}^d q^j \beta_j \right) = 0.
	\end{equation*}
	This was proved in \cite[Theorem 5.2]{bm_vdC}.
\end{proof}

Using the observation that (polynomial) orbits in Abramov systems have polynomial coordinates, we obtain the following corollary.

\begin{corollary} \label{cor: lambda exists}
	Let $(X, \mu, T)$ be an Abramov $\Q$-system, and let $(Z, m_Z, R)$ be its Kronecker factor.
	Assume that the factor map provided by Theorem \ref{thm: Abramov structure} is continuous.
	Then the weak* limit
	\begin{equation*}
		\UClim_{q \in \Q} \delta_{T^qx} \times \delta_{R^{P(q)} z}
	\end{equation*}
	exists for every $(x, z) \in X \times Z$ and every polynomial $P$.
\end{corollary}

Let $(X, \mu, T)$ be an ergodic $\Q$-system with divisible discrete spectrum, and fix a polynomial $P$ of degree $d$.
We fix a topological model for the Abramov factor $A_d$ such that the factor map in Theorem \ref{thm: Abramov structure} is continuous.
Assume moreover that there is a continuous factor map $\pi_d : X \to A_d$.
Let $\pi_1$ be the factor map $\pi_1 : X \to Z$.
This factor map will automatically be continuous, since it can be written as a composition of $\pi_d$ with a topological factor map $\pi^d_1 : A_d \to Z$.
For points $(u, z) \in A_d \times Z$, define (by Corollary \ref{cor: lambda exists})
\begin{equation*}
	\tilde{\lambda}^P_{(u,z)} = \UClim_{q \in \Q} \delta_{T_d^qu} \times \delta_{R^{P(q)} z}.
\end{equation*}
Then define for $(x_1, x_2) \in X$ a measure $\lambda^P_{(x_1, x_2)}$ on $X^2$ by
\begin{equation*}
	\int_{X \times X} f \otimes g~d\lambda^P_{(x_1, x_2)} = \int_{A_d \times Z} \E{f}{A_d} \otimes \E{g}{Z}~d\tilde{\lambda}^P_{(\pi_d(x_1),\pi_1(x_2))}
\end{equation*}
for $f, g \in C(X)$.
In order for this expression to be well-defined (since the conditional expectation is only determined almost everywhere), we require both marginals of the measure $\tilde{\lambda}_{(\pi_d(x_1), \pi_1(x_2))}^P$ to be absolutely continuous (with respect to the Haar measures on $A_d$ and $Z$, respectively).
The first marginal is given by $\UClim_{q \in \Q} \delta_{T_d^q \pi_d(x)}$.
This is equal to $m_{A_d}$, since $(A_d, T_d)$ is uniquely ergodic (see \cite[Theorem 3.12]{furstenberg_book} for a proof for a closely related $\Z$-system that is easily adapted to our context of $\Q$-systems).
The second marginal is given by $\UClim_{q \in \Q} \delta_{R^{P(q)} \pi_1(x)}$, which is equal to $m_Z$ unless $P(q)$ is constant by Proposition \ref{prop: polynomial equidistribution}.

\begin{theorem} \label{thm: limit formula}
	Let $(X, \mu, T)$ be an ergodic $\Q$-system with divisible discrete spectrum and topological Abramov factors, and let $a \in \gen(\mu, \Phi)$.
	Let $f \in C(X)$ and $g \in L^2(\mu)$.
	Then for any nonconstant polynomial $P$,
	\begin{equation*}
		\lim_{N \to \infty} \frac{1}{|\Phi_N|} \sum_{q \in \Phi_N} f(T^qa) g \left( T^{P(q)}x \right) = \int_{X^2} f \otimes g~d\lambda^P_{(a,x)}
	\end{equation*}
	in $L^2(\mu)$.
\end{theorem}

\begin{proof}
	By scaling if necessary, we may assume $\norm{L^2(\mu)}{g} = 1$.
	Let $\eps > 0$.
	By Theorem \ref{thm: characteristic factor}, we may decompose $f = f_s \circ \pi_d + f_u$ and $g = g_s \circ \pi_1 + g_u$ such that $f_s \in C(A_d)$, $f_u \in C(X)$ with $\norm{L^1(\mu)}{\E{f_u}{\mA_d}} < \frac{\eps}{2}$, $g_s = \E{g}{Z}$, and
	\begin{equation*}
		\limsup_{N \to \infty} \norm{L^2(\mu)}{\frac{1}{|\Phi_N|} \sum_{q \in \Phi_N} \left( f(T^qa) \cdot T^{P(q)}g - f_s(T_d^q\pi_d(a)) \cdot T^{P(q)}(g_s \circ \pi_1) \right)} < \frac{\eps}{2}.
	\end{equation*}
	
	Now, by the definition of the measures $\tilde{\lambda}^P$, we have
	\begin{equation*}
		\lim_{N \to \infty} \frac{1}{|\Phi_N|} \sum_{q \in \Phi_N} f_s(T_d^q\pi_d(a)) \cdot h(R^{P(q)}z) = \int_{A_d \times Z} f_s \otimes h~d\tilde{\lambda}^P_{(\pi_d(a),z)}
	\end{equation*}
	for every $h \in C(X)$ and $z \in Z$.
	In particular, approximating $g_s$ by continuous functions in $L^2(m_Z)$, we have
	\begin{equation*}
		\lim_{N \to \infty} \frac{1}{|\Phi_N|} \sum_{q \in \Phi_N} f_s(T_d^q\pi_d(a)) \cdot (g_s \circ \pi_1)(T^{P(q)}x) = \int_{A_d \times Z} f_s \otimes g_s~d\tilde{\lambda}^P_{(\pi_d(a),\pi_1(x))}
	\end{equation*}
	in $L^2(\mu)$ as a function of $x \in X$.
	From the definition of $\lambda^P$, we also have
	\begin{equation*}
		\int_{X^2} f \otimes g~d\lambda^P_{(a,x)} = \int_{A_d \times Z} \E{f}{A_d} \otimes g_s~d\tilde{\lambda}^P_{(\pi_d(a),\pi_1(x))},
	\end{equation*}
	so
	\begin{multline*}
		\limsup_{N \to \infty} \norm{L^2(\mu)}{\frac{1}{|\Phi_N|} \sum_{q \in \Phi_N} f(T^qa) g \left( T^{P(q)}x \right) - \int_{X^2} f \otimes g~d\lambda^P_{(a,x)}} \\
		 \le \frac{\eps}{2} + \norm{L^2(m_{A_d})}{\int_{A_d \times Z} (f_s - \E{f}{A_d}) \otimes g_s~d\tilde{\lambda}^P_{(\pi_d(a),\pi_1(x))}} \\
		 \le \frac{\eps}{2} + \norm{L^1(\mu)}{\E{f_u}{\mA_d}} < \eps.
	\end{multline*}
\end{proof}


\section{Progressive measures} \label{sec: progressive measure}

Motivated by the method in \cite{kmrr_FS}, we define a notion of progressive measures for detecting Erd\H{o}s--Furstenberg--S\'{a}rk\"{o}zy progressions.
For a compact metric space $Y$, we denote by $\mM(Y)$ the space of Borel probability measures on $Y$.

\begin{definition}
	Let $(X, \mu, T)$ be an ergodic $\Q$-system, and let $P(x) \in \Q[x]$.
	A probability measure $\tau \in \mM(X^2)$ is \emph{$P$-progressive from a point $a \in X$} if $\supp{\tau} \subseteq \overline{\EFS^P_X(a)}$.
\end{definition}

The goal of this section is to construct a progressive measure.
Let $(X, \mu, T)$ be an ergodic $\Q$-system with divisible discrete spectrum and topological Abramov factors, and let $P$ be a polynomial of degree $d$.
For a point $u \in A_d$, we define a measure  $\tilde{\sigma}^P_u$ by
\begin{equation*}
	\tilde{\sigma}^P_u = \UClim_{q \in \Q} \delta_{T_d^q u} \times \delta_{R^{P(q)+q} \pi^d_1(u)}.
\end{equation*}
This measure is well-defined by Corollary \ref{cor: lambda exists}.
We then define $\sigma^P_x \in \mM(X^2)$ for $x \in X$ by
\begin{equation*}
	\int_{X \times X} f \otimes g~d\sigma^P_x = \int_{A_d \times Z} \E{f}{A_d} \otimes \E{g}{Z}~d\tilde{\sigma}^P_{\pi_d(x)}
\end{equation*}
for $f, g \in C(X)$.
Once again, we need to ensure that the marginals of $\tilde{\sigma}^P_{\pi_d(x)}$ are absolutely continuous so that $\sigma^P_x$ is well-defined.
Noting that $\tilde{\sigma}^P_{\pi_d(x)} = \tilde{\lambda}^Q_{(\pi_d(x),\pi_1(x))}$ for the polynomial $Q(x) = P(x) + x$, we see that $\sigma^P_x$ is well-defined whenever $P(x) + x$ is nonconstant.

The main result of this section is:

\begin{theorem} \label{thm: sigma is progressive}
	Let $(X, \mu, T)$ be an ergodic $\Q$-system with divisible discrete spectrum and topological Abramov factors, and let $P(x) \in \Q[x]$ such that both $P(x)$ and $P(x) + x$ are nonconstant.
	If $a \in \gen(\mu, \Phi)$ for some F{\o}lner sequence $\Phi$, then $\sigma^P_a$ is $P$-progressive from $a$.
\end{theorem}

For the remainder of this section, we fix an ergodic $\Q$-system $(X, \mu, T)$ with divisible discrete spectrum and topological Abramov factors, a F{\o}lner sequence $\Phi$, a generic point $a \in \gen(\mu,\Phi)$, and a polynomial $P(x) \in \Q[x]$ such that $P(x)$ and $P(x)+x$ are nonconstant.

We have the following criterion for a measure to be progressive:

\begin{proposition} \label{prop: progressive criterion sets}
	Let $\tau \in \mM(X^2)$ and $a \in X$.
	Suppose that for any open sets $U, V \subseteq X$ with $\tau(U \times V) > 0$, there are infinitely many $q \in \Q$ such that $T^q a \in U$ and
	\begin{equation*}
		\tau \left( (U \times V) \cap (T^{-P(q)}V \times X) \right) > 0.
	\end{equation*}
	Then $\tau$ is $P$-progressive from $a$.
\end{proposition}

\begin{proof}
	A closely related result is proved in \cite[Proposition 3.2]{kmrr_FS}.
	We adapt the argument to our setting of polynomial progressions.
	
	In order to show that $\supp{\tau} \subseteq \overline{\EFS^P_X(a)}$, it suffices to prove that if $U, V \subseteq X$ are open sets and $\tau(U \times V) > 0$, then $\EFS^P_X(a) \cap (U \times V) \ne \es$.
	Let $U, V \subseteq X$ be open set with $\tau(U \times V) > 0$.
	We will construct sequences $(U_n)_{n \ge 0}$ and $(V_n)_{n \ge 0}$ of open subsets of $X$ with $U_0 = U$ and $V_0 = V$ and a sequence $(s_n)_{n \in \N}$ of distinct elements of $\Q$ such that the following properties hold for each $n \in \N$:
	\begin{enumerate}[(1)]
		\item	$\overline{U}_n \subseteq U_{n-1}$ and $\overline{V}_n \subseteq V_{n-1}$,
		\item	$\diam{U_n} \le \frac{1}{2} \diam{U_{n-1}}$ and $\diam{V_n} \le \frac{1}{2} \diam{V_{n-1}}$,
		\item	$\tau(U_n \times V_n) > 0$,
		\item	$T^{s_n}a \in U_{n-1}$, and
		\item	$U_n \subseteq T^{-P(s_n)}V_{n-1}$.
	\end{enumerate}
	
	First we handle $n = 1$.
	Since $\tau(U_0 \times V_0) = \tau(U \times V) > 0$, we use the hypothesis for $\tau$ to find $s_1 \in \Q$ such that $T^{s_1}a \in U_0$ and
	\begin{equation*}
		\tau \left( (U_0 \times V_0) \cap (T^{-P(s_1)}V_0 \times X) \right) > 0.
	\end{equation*}
	We can then use regularity of the measure $\tau$ to find open sets $U_1, V_1 \subseteq X$ such that $\overline{U}_1 \subseteq U_0 \cap T^{-P(s_1)}V_0$ and $\overline{V}_1 \subseteq V_0$ with $\diam{U_1} \le \frac{1}{2} \diam{U_0}$, $\diam{V_1} \le \frac{1}{2} \diam{V_0}$, and $\tau(U_1 \times V_1) > 0$.
	
	Suppose now that we have constructed $U_1, \ldots, U_n$, $V_1, \ldots, V_n$, and $s_1, \ldots s_n$ for some $n \in \N$.
	By the induction hypothesis, $\tau(U_n \times V_n) > 0$, so by assumption, there exists $s_{n+1} \in \Q \setminus \{s_1, \ldots, s_n\}$ such that $T^{s_{n+1}}a \in U_n$ and
	\begin{equation*}
		\tau \left( (U_n \times V_n) \cap (T^{-P(s_{n+1})}V_n \times X) \right) > 0.
	\end{equation*}
	As in the base case, we may therefore find open sets $U_{n+1}, V_{n+1} \subseteq X$ such that $\overline{U}_{n+1} \subseteq U_n \cap T^{-P(s_{n+1})}V_n$ and $\overline{V}_{n+1} \subseteq V_n$ with $\diam{U_{n+1}} \le \frac{1}{2} \diam{U_n}$, $\diam{V_{n+1}} \le \frac{1}{2} \diam{V_n}$, and $\tau(U_{n+1} \times V_{n+1}) > 0$.
	
	\bigskip
	
	By (1) and (2), the sets $\bigcap_{n \ge 0} U_n$ and $\bigcap_{n \ge 0} V_n$ are singletons.
	Let $x_1$ be the unique element of $\bigcap_{n \ge 0} U_n$ and $x_2$ the unique element of $\bigcap_{n \ge 0} V_n$.
	Then by (4), $T^{s_n}a \to x_1$ as $n \to \infty$.
	Moreover, by (5),
	\begin{equation*}
		T^{P(s_n)}x_1 \in V_{n-1}
	\end{equation*}
	for every $n \in \N$, so $T^{P(s_n)}x_1 \to x_2$ as $n \to \infty$.
	That is, $(a, x_1, x_2) \in \EFS^P_X$.
	Since $x_1 \in U_0 = U$ and $x_2 \in V_0 = V$, this proves $\EFS^P_X(a) \cap (U \times V) \ne \es$ as desired.
\end{proof}

The following variant of Proposition \ref{prop: progressive criterion sets} will be more directly applicable with the methods we employ.

\begin{proposition} \label{prop: progressive criterion}
	Let $\tau \in \mM(X^2)$ and $a \in X$.
	Let $\Phi$ be a F{\o}lner sequence in $\Q$.
	Suppose the following property holds: for all $f, g \in C(X)$, if $f, g \ge 0$ and $\int_{X^2} f \otimes g~d\tau > 0$, then
	\begin{equation*}
		\limsup_{N \to \infty} \frac{1}{|\Phi_N|} \sum_{q \in \Phi_N} f(T^q a) \int_{X^2} (f \cdot T^{P(q)} g) \otimes g~d\tau > 0.
	\end{equation*}
	Then $\tau$ is $P$-progressive from $a$.
\end{proposition}

\begin{proof}
	We reduce to Proposition \ref{prop: progressive criterion sets}.
	Let $U, V \subseteq X$ be open sets with $\tau(U \times V) > 0$.
	Then there exist continuous functions $f, g : X \to [0,1]$ such that $\supp{f} \subseteq U$, $\supp{g} \subseteq V$, and $\int_{X^2} f \otimes g~d\tau > 0$.
	By assumption, we have
	\begin{equation*}
		\limsup_{N \to \infty} \frac{1}{|\Phi_N|} \sum_{q \in \Phi_N} f(T^q a) \int_{X^2} (f \cdot T^{P(q)} g) \otimes g~d\tau > 0.
	\end{equation*}
	In particular,
	\begin{equation} \label{eq: positive with recurrence}
		f(T^q a) \int_{X^2} (f \cdot T^{P(q)} g) \otimes g~d\tau > 0.
	\end{equation}
	for infinitely many $q \in \Q$.
	Suppose $q$ satisfies \eqref{eq: positive with recurrence}.
	Then $f(T^qa) > 0$, so $T^qa \in U$.
	Moreover,
	\begin{align*}
		\tau \left( (U \times V) \cap (T^{-P(q)}V \times X) \right) & = \int_{X^2} (\ind_U \cdot T^{P(q)} \ind_V) \otimes \ind_V~d\tau \\
         & \ge \int_{X^2} (f \cdot T^{P(q)} g) \otimes g~d\tau > 0.
	\end{align*}
	Thus, $\tau$ satisfies the hypothesis of Proposition \ref{prop: progressive criterion sets}, so $\tau$ is $P$-progressive from $a$.
\end{proof}

The main combinatorial input required to prove Theorem \ref{thm: sigma is progressive} is the following version of the Ajtai--Szemer\'{e}di theorem in compact abelian groups:

\begin{theorem} \label{thm: uniform Ajtai Szemeredi}
	Let $G$ be a compact abelian group.
	For any $\eps > 0$, there exists $\delta > 0$ with the following property: if $F : G^2 \to [0,1]$ is a measurable function such that $\int_{G^2} F~dm_{G^2} \ge \eps$, then
	\begin{equation*}
		\int_{G^3} F(x,y) F(x+t,y) F(x,y+t)~dx~dy~dt \ge \delta.
	\end{equation*}
\end{theorem}

\begin{proof}
	Let us first prove a corresponding statement for sets: for any $\eps > 0$, there exists a positive constant $ASz(\eps) > 0$ such that if $E \subseteq G^2$ is a measurable subset of measure at least $\eps$, then the set
	\begin{equation*}
		C(E) = \{(x,y,t) \in G^3 : \{(x,y), (x+t,y), (x,y+t)\} \subseteq E\}
	\end{equation*}
	has measure at least $ASz(\eps)$.
	(The notation $ASz$ stands for ``Ajtai--Szemer\'{e}di'' and $C$ stands for ``corners.'')
	This is essentially a statement of the Ajtai--Szemer\'{e}di theorem \cite{asz} and can be proved using standard regularity methods.
	For a reference, the exact result claimed here follows from the much stronger popular difference variant in \cite[Theorem 1.1]{berger}.
	
	Now we deduce the functional statement.
	Let $F : G^2 \to [0,1]$ be a measurable function with $\int_{G^2} F~d\mu \ge \eps$.
	Let $E = \left\{ (x,y) : F(x,y) \ge \frac{\eps}{2} \right\}$.
	Then $m_{G^2}(E) \ge \frac{\eps}{2}$ by Markov's inequality:
	\begin{equation*}
		m_{G^2}{(G^2 \setminus E)} \le \frac{1}{1 - \frac{\eps}{2}} \int_{G^2} (1 - F)~dm_{G^2} = \frac{1 - \eps}{1 - \frac{\eps}{2}} = 1 - \frac{\frac{\eps}{2}}{1 - \frac{\eps}{2}} \le 1 - \frac{\eps}{2}.
	\end{equation*}
	Therefore,
	\begin{align*}
		\int_{G^3} F(x,y) F(x+t,y) & F(x,y+t)~dx~dy~dt \\
        & \ge \left( \frac{\eps}{2} \right)^3 \int_{G^3} \ind_E(x,y) \ind_E(x+t,y) \ind_E(x,y+t)~dx~dy~dt \\
		 & = \left( \frac{\eps}{2} \right)^3 \cdot m_{G^3}(C(E)) \\
		 & \ge \left( \frac{\eps}{2} \right)^3 ASz \left( \frac{\eps}{2} \right).
	\end{align*}
	Thus, we can take $\delta = \left( \frac{\eps}{2} \right)^3 ASz \left( \frac{\eps}{2} \right)$.
\end{proof}

\begin{proposition} \label{prop: uniform bound on factor}
	For any $\eps > 0$, there exists $\delta > 0$ such that the following holds: if $f : A_d \to [0,1]$ and $g : Z \to [0,1]$ are measurable functions, $b \in A_d$, and
	\begin{equation*}
		\int_{A_d \times Z} f \otimes g~d\tilde{\sigma}^P_b \ge \eps,
	\end{equation*}
	then
	\begin{equation*}
		\int_{A_d \times Z} f(u) g(z) \left( \int_{A_d \times Z} f \otimes g~d\tilde{\lambda}^P_{(b, \pi^d_1(u))} \right)~d\tilde{\sigma}^P_b(u,z) \ge \delta.
	\end{equation*}
\end{proposition}

\begin{proof}
	We claim that it suffices to prove the proposition for continuous functions $f$ and $g$.
	Indeed, once we have the proposition for continuous functions, we may approximate measurable $f$ and $g$ by continuous functions in $L^2$ to deduce the general case.
	
	We assume from now on that $f$ and $g$ are continuous.
	Expanding the definition of $\tilde{\sigma}^P_b$, we have
	\begin{equation*}
		\int_{A_d \times Z} f \otimes g~d\tilde{\sigma}^P_b = \UClim_{q \in \Q} f \left( T_d^q b \right) \cdot g \left( R^{P(q)+q} \pi^d_1(b) \right).
	\end{equation*}
	In order to compute this limit, we work with the more convenient extension provided by Theorem \ref{thm: Abramov structure}.
	Let $V$ be a compact $\Q$-vector space and $\alpha \in V$ such that $(A_d, m_{A_d}, T_d)$ is a topological factor of $(V^d, m_V^d, S)$, where
	\begin{multline*}
		S^q(v_1, \ldots, v_d) \\
         = \left( v_1 + q\alpha, v_2 + qv_1 + \binom{q}{2} \alpha, \ldots, v_d + q v_{d-1} + \ldots + \binom{q}{d-1} v_1 + \binom{q}{d} \alpha \right).
	\end{multline*}
	Let $\rho : V^d \to A_d$ be the factor map, and let $\bm{v} = (v_1, \ldots, v_d) \in V^d$ such that $\rho(\bm{v}) = b$.
	Expand $P(x) = \sum_{j=0}^d c_j \binom{x}{j}$.
	Then we can write the point $\left( S^q \bm{v}, v_1 + (P(q) + q)\alpha \right)$ as
	\begin{multline*}
		(\bm{v}, v_1 + c_0\alpha) + q \left( \alpha, v_1, \ldots, v_{d-1}, (c_1 + 1)\alpha \right) + \binom{q}{2} \left( 0, \alpha, v_1, \ldots, v_{d-2}, c_2 \alpha \right) + \ldots \\
        \ldots + \binom{q}{d} (0, \ldots, \alpha, c_d\alpha).
	\end{multline*}
	Let
	\begin{equation*}
		H = \overline{\left\{ q(\alpha, v_1, \ldots, v_{d-1}) : q \in \Q \right\}} \subseteq V^d.
	\end{equation*}
	We write points $\bm{x} \in H$ as $\bm{x} = (x_0, x_1, \ldots, x_{d-1})$ in order to match indices from $\bm{v}$.
	In the product space $H^k$, $k \in \N$, we write coordinates with superscripts, reserving the usage of subscripts for the coordinates of points in $H \subseteq V^d$.
	That is, $\bm{x} \in H^k$ takes the form $\bm{x} = (\bm{x}^1, \ldots, \bm{x}^k)$ with $\bm{x}^j = (x^j_0, x^j_1, \ldots, x^j_{d-1}) \in H$.
	By Proposition \ref{prop: polynomial equidistribution}, the sequence $\left( S^q \bm{v}, v_1 + (P(q) + q)\alpha \right)_{q \in \Q}$ is well-distributed in the space
	\begin{multline*}
		\left\{ (\bm{v}, v_1 + c_0\alpha) + \left( \bm{x}^1, (c_1+1)x^1_0 \right) + \left( 0, x^2_0, \ldots, x^2_{d-2}, c_2 x^2_0 \right) + \ldots \right. \\
        \left. \ldots + (0, \ldots, x^d_0, c_d x^d_0) : \bm{x} \in H^d \right\}.
	\end{multline*}
	Define $\varphi : H^d \to [0,1]$ by
	\begin{equation*}
		\varphi(\bm{x}^1, \ldots, \bm{x}^d) = (f \circ \rho) \left( \bm{v} + \bm{x}^1 + (0, x^2_0, \ldots, x^2_{d-2}) + \ldots + (0, \ldots, x^d_0) \right)
	\end{equation*}
	and $\psi : H^{d+1} \to [0,1]$ by
	\begin{equation*}
		\psi(\bm{x}^1, \ldots, \bm{x}^d, \bm{y}) = (g \circ \rho') \left( v_1 + c_0 \alpha + y^1_0 + \sum_{j=1}^d c_j x^j_0 \right),
	\end{equation*}
	where $\rho'$ is the factor map $\rho' : V \to Z$.
	Let $\xi(\bm{x}) = \varphi(\bm{x}) \psi(\bm{x}, \bm{x}^1)$.
	Then by construction,
	\begin{equation*}
		\int_{H^d} \xi~dm_{H^d} = \int_{A_d \times Z} f \otimes g~d\tilde{\sigma}^P_b \ge \eps.
	\end{equation*}
	
	Expanding
	\begin{equation*}
		I(f,g) = \int_{A_d \times Z} f(u) g(z) \left( \int_{A_d \times Z} f \otimes g~d\tilde{\lambda}^P_{(b, \pi^d_1(u))} \right)~d\tilde{\sigma}^P_b(u,z)
	\end{equation*}
	as an average
	\begin{equation*}
		\UClim_{(q,r) \in \Q^2} f \left( T_d^q b \right) \cdot g \left( R^{P(q)+q} \pi^d_1(b) \right) \cdot f \left( T_d^r b \right) \cdot g \left( R^{P(r)+q} \pi^d_1(b) \right)
	\end{equation*}
	and repeating the previous calculations, we have
	\begin{equation*}
		I(f,g) = \int_{H^d \times H^d} \varphi(\bm{x}) \psi(\bm{x}, \bm{x}^1) \varphi(\bm{y}) \psi(\bm{y}, \bm{x}^1)~d\bm{x}~d\bm{y}.
	\end{equation*}
	Making a change of variables $\bm{y} = (\bm{x}^1 + (1 + c_1^{-1})\bm{t}^1, \bm{x}^2 + \bm{t}^2, \ldots, \bm{x}^d + \bm{t}^d)$ and using the definition of $\psi$, we have
	\begin{align*}
		I(f,g) & = \int_{H^d \times H^d} \varphi(\bm{x}) \psi(\bm{x}, \bm{x}^1) \varphi(\bm{x}^1 + (1 + c_1^{-1})\bm{t}^1, \bm{x}^2 + \bm{t}^2, \ldots, \bm{x}^d + \bm{t}^d)~\cdot \\
        & \qquad \qquad \qquad \qquad \cdot \psi(\bm{x} + \bm{t}, \bm{x}^1 + \bm{t}^1)~d\bm{x}~d\bm{t} \\
		 & \ge \int_{H^d \times H^d} \xi(\bm{x}) \xi(\bm{x} + \bm{t}) \xi(\bm{x}^1 + (1 + c_1^{-1})\bm{t}^1, \bm{x}^2 + \bm{t}^2, \ldots, \bm{x}^d + \bm{t}^d)~d\bm{x}~d\bm{t}.
	\end{align*}
	Now letting $p : H^d \times H^d \to H^d$ be the map
	\begin{equation*}
		p(\bm{x}, \bm{y}) = \left( \bm{x}^1 + (1+c_1^{-1}) \bm{y}^1, \bm{x}^2 + \bm{y}^2, \ldots, \bm{x}^d + \bm{y}^d \right)
	\end{equation*}
	and defining $F : H^d \times H^d \to [0,1]$ by $F = \xi \circ p$, we have
	\begin{equation*}
		I(f,g) \ge \int_{H^d \times H^d \times H^d} F(\bm{x}, \bm{y}) F(\bm{x} + \bm{t}, \bm{y}) F(\bm{x}, \bm{y} + \bm{t})~d\bm{x}~d\bm{y}~d\bm{t}
	\end{equation*}
	and
	\begin{equation*}
		\int_{H^d \times H^d} F~dm_{H^d}^2 = \int_{H^d} \xi~dm_{H^d} \ge \eps.
	\end{equation*}
	Applying Theorem \ref{thm: uniform Ajtai Szemeredi} for the group $G = H^d$, we are done.
\end{proof}

\begin{proof}[Proof of Theorem \ref{thm: sigma is progressive}]
	We use the criterion given in Proposition \ref{prop: progressive criterion}.
	Suppose $f, g \in C(X)$, $f, g \ge 0$, and
	\begin{equation*}
		\int_{X^2} f \otimes g~d\sigma^P_a > 0.
	\end{equation*}
	Then by Theorem \ref{thm: limit formula},
	\begin{multline*}
		\lim_{N \to \infty} \frac{1}{|\Phi_N|} \sum_{q \in \Phi_N} f(T^q a) \int_{X^2} (f \cdot T^{P(q)} g) \otimes g~d\sigma^P_a \\
        = \int_{X^2} f(x_1) g(x_2) \left( \int_{X^2} f \otimes g~d\lambda^P_{(a,x_1)} \right)~d\sigma^P_a(x_1,x_2).
	\end{multline*}
	Applying Proposition \ref{prop: uniform bound on factor} with the functions $\E{f}{A_d}$ and $\E{g}{Z}$ and the point $b = \pi_d(a) \in A_d$ completes the proof.
\end{proof}


\section{Proof of Theorem \ref{thm: density dynamical}} \label{sec: proof finish density}

We now wish to combine the ingredients from the previous sections to prove Theorem \ref{thm: density dynamical}.
Severals results of the previous sections were contingent on a system having divisible discrete spectrum and topological Abramov factors.
The following technical result will allow us to reduce to systems with this property.
(Recall that $\sigma'(T)$ denotes the $\Q$-vector subspace of $\A/\Q$ generated by the eigenvalues of $T$.)

\begin{proposition} \label{prop: divisible and topological factors}
	Let $(X, \mu, T)$ be an ergodic $\Q$-system, $\Phi$ a F{\o}lner sequence, and $a \in X$ a transitive point such that $a \in \gen(\mu, \Phi)$.
	Then there exists an ergodic extension $\pi : (\tilde{X}, \tilde{\mu}, \tilde{T}) \to (X, \mu, T)$, a F{\o}lner sequence $\Psi$, and a transitive point $\tilde{a} \in \tilde{X}$ such that
	\begin{itemize}
		\item	$\pi$ is a topological factor map,
		\item	$\pi(\tilde{a}) = a$,
		\item	$\tilde{a} \in \gen(\tilde{\mu}, \Psi)$,
		\item	$(\tilde{X}, \tilde{\mu}, \tilde{T})$ has topological Abramov factors, and
		\item	$\sigma(\tilde{T}) = \sigma'(T)$.
	\end{itemize}
\end{proposition}

We prove Proposition \ref{prop: divisible and topological factors} by combining the next two lemmas.

\begin{lemma} \label{lem: topological factors}
	Let $(X, \mu, T)$ be an ergodic $\Q$-system, $\Phi$ a F{\o}lner sequence, and $a \in X$ a transitive point such that $a \in \gen(\mu, \Phi)$.
	Then there exists an extension $\pi : (\tilde{X}, \tilde{\mu}, \tilde{T}) \to (X, \mu, T)$, a F{\o}lner sequence $\Psi$, and a transitive point $\tilde{a} \in \tilde{X}$ such that
	\begin{itemize}
		\item	$\pi$ is a topological factor map and a measurable isomorphism,
		\item	$\pi(\tilde{a}) = a$,
		\item	$\tilde{a} \in \gen(\tilde{\mu}, \Psi)$, and
		\item	$(\tilde{X}, \tilde{\mu}, \tilde{T})$ has topological Abramov factors.
	\end{itemize}
\end{lemma}

\begin{proof}
	A related result was shown in \cite[Lemma 5.8]{kmrr1} for $\Z$-systems with the conclusion that the extension $(\tilde{X}, \tilde{\mu}, \tilde{T})$ has topological \emph{pronilfactors}.
	For more general group actions, the arguments were generalized in \cite[Section 3.3]{cm}.
	The only thing we need to check in order to apply the same argument to obtain topological Abramov factors is that the Abramov factors are topologically distal\footnote{A topological dynamical $\Q$-system $(X, T)$ is \emph{distal} if for every pair of distinct points $x \ne y \in X$, one has $\inf_{q \in \Q} d(T^qx, T^qy) > 0$, where $d$ is the metric on $X$.} (see the remark at the end of \cite[Lemma 3.5]{cm}).
	The structure theorem (Theorem \ref{thm: Abramov structure}) reveals that an Abramov system of order $k$ can be built as a height $k$ tower of isometric extensions of the trivial system, so Abramov systems are indeed distal.
\end{proof}

\begin{lemma} \label{lem: divisible}
	Let $(X, \mu, T)$ be an ergodic $\Q$-system with topological Abramov factors, $\Phi$ a F{\o}lner sequence, and $a \in X$ a transitive point such that $a \in \gen(\mu, \Phi)$.
	Then there exists an ergodic extension $\pi : (\tilde{X}, \tilde{\mu}, \tilde{T}) \to (X, \mu, T)$ and a transitive point $\tilde{a} \in \tilde{X}$ such that
	\begin{itemize}
		\item	$\pi$ is a topological factor map,
		\item	$\pi(\tilde{a}) = a$,
		\item	$\tilde{a} \in \gen(\tilde{\mu}, \Phi)$, and
		\item	$\sigma(\tilde{T}) = \sigma'(T)$.
	\end{itemize}
\end{lemma}

\begin{proof}	
	Let $\sigma(T)$ be the discrete spectrum of $T$ and $Z = \hat{\sigma(T)}$ the dual group.
	By the Halmos--von Neumann theorem (see, e.g., \cite[pp. 46--48]{halmos}), the Kronecker factor of $(X, \mu, T)$ is isomorphic to the system $(Z, m_Z, R)$ given by
	\begin{equation*}
		(R^q z)(\beta) = e_{\Q}(q\beta) z(\beta)
	\end{equation*}
	for $q \in \Q$, $z \in Z$ and $\beta \in \sigma(T)$.
	By assumption, there is a continuous factor map $\pi_Z : X \to Z$.
	Let $\sigma'(T)$ be the divisible group generated by $\sigma(T)$, and let $Z' = \hat{\sigma'(T)}$.
	Define a rotational action on $Z'$ by
	\begin{equation*}
		((R')^q z')(\beta') = e_{\Q}(q\beta') z'(\beta')
	\end{equation*}
	for $q \in \Q$, $z' \in Z'$, and $\beta' \in \sigma'(T)$.
	The map $\rho : Z' \to Z$ defined by $\rho(z') = z'|_{\sigma(T)}$ is a factor map.
	
	We can now define the extension $\tilde{X}$.
	Let $\tilde{X} = \{(x,z') : \pi_Z(x) = \rho(z')\}$, and let $\tilde{a} = (a,b)$, where $b \in Z'$ is any point satisfying $\rho(b) = \pi_Z(a)$.
	For convenience, let us write $c = \pi_Z(a) = \rho(b) \in Z$.
	The action $\tilde{T}$ is given by $\tilde{T}^q(x,z') = (T^qx, (R')^qz')$.
	
	Let us check that $\tilde{a} \in \gen(\tilde{\mu}, \Phi)$, where $\tilde{\mu}$ is the relatively independent joining $\tilde{\mu} = \mu \times_Z m_{Z'}$ defined by
	\begin{equation} \label{eq: rel ind joining}
		\int_{\tilde{X}} f(x) g(z')~d\tilde{\mu}(x,z') = \int_Z \E{f}{Z} \cdot \E{g}{Z}~dm_Z
	\end{equation}
	for $f \in C(X)$ and $g \in C(Z')$.
	It is readily verified that the support of $\tilde{\mu}$ is equal to $\tilde{X}$, so this will also establish transitivity of the point $\tilde{a}$.
	
	By the Stone--Weierstrass theorem, it suffices to prove the identity \eqref{eq: rel ind joining} for $g$ of the form $g(z') = z'(\beta')$ for some $\beta' \in \sigma'(T)$.
	That is, we want to show
	\begin{equation*}
		\lim_{N \to \infty} \frac{1}{|\Phi_N|} \sum_{q \in \Phi_N} f(T^qa) e_{\Q}(q\beta') b(\beta') = \begin{cases}
			\int_Z \E{f}{Z}(z) \cdot z(\beta')~dz, & \text{if}~\beta' \in \sigma(T); \\
			0, & \text{if}~\beta' \notin \sigma(T).
		\end{cases}
	\end{equation*}
	
	Since $L^2(m_Z)$ is spanned by characters, we can write $\E{f}{Z} = \sum_{\beta \in \sigma(T)} d_{\beta} h_{\beta}$, where $h_{\beta}(z) = z(\beta)$.
	By orthogonality of characters, $\int_Z \E{f}{Z} \cdot z(\beta)~dz = d_{-\beta}$ for $\beta \in \sigma(T)$.
	Given $\eps > 0$, we may pick finitely many values $\beta_1, \ldots, \beta_n$ so that for the function
	\begin{equation*}
		f_s = \sum_{j=1}^n d_{\beta_j} h_{\beta_j}
	\end{equation*}
	we have $\|\E{f}{Z} - f_s\|_{L^1(m_Z)} \le \|\E{f}{Z} - f_s\|_{L^2(m_Z)} < \eps$.
	Let $f_u = f - f_s \circ \pi$.
	Then
	\begin{multline*}
		\frac{1}{|\Phi_N|} \sum_{q \in \Phi_N} f(T^qa) e_{\Q}(q\beta') b(\beta') = \sum_{j=1}^n d_{\beta_j} b(\beta') c(\beta_j) \cdot \frac{1}{|\Phi_N|} \sum_{q \in \Phi_N} e_{\Q} \left( q(\beta'+\beta_j) \right) \\
		 + b(\beta') \cdot \frac{1}{|\Phi_N|} \sum_{q \in \Phi_N} f_u(T^qa) e_{\Q}(q\beta').
	\end{multline*}
	
	By Theorem \ref{thm: Abramov partially characteristic},
	\begin{equation*}
		\limsup_{N \to \infty} \left| \frac{1}{|\Phi_N|} \sum_{q \in \Phi_N} f_u(T^qa) e_{\Q}(q\beta') \right| \le \norm{L^1(\mu)}{\E{f_u}{\mZ}} < \eps.
	\end{equation*}
	(Strictly speaking, Theorem \ref{thm: Abramov partially characteristic} requires the system to have divisible discrete spectrum.
	However, in the $d=1$ case we are using here, this condition can be removed by using Remark \ref{rem: ww kronecker}.)
	For $j \in \{1, \ldots, n\}$,
	\begin{equation*}
		\lim_{N \to \infty} \frac{1}{|\Phi_N|} \sum_{q \in \Phi_N} e_{\Q} \left( q(\beta'+\beta_j) \right) = \begin{cases}
			1, & \text{if}~\beta' + \beta_j = 0; \\
			0, & \text{otherwise}.
		\end{cases}
	\end{equation*}
	In particular, if $\beta' \notin \sigma(T)$, then
	\begin{equation*}
		\limsup_{N \to \infty} \left| \frac{1}{|\Phi_N|} \sum_{q \in \Phi_N} f(T^qa) e_{\Q}(q\beta') b(\beta') \right| < \eps.
	\end{equation*}
	But $\eps$ was arbitrary, so
	\begin{equation*}
		\lim_{N \to \infty} \frac{1}{|\Phi_N|} \sum_{q \in \Phi_N} f(T^qa) e_{\Q}(q\beta') b(\beta') = 0
	\end{equation*}
	as desired.
	On the other hand, if $\beta' \in \sigma(T)$, then taking $j$ such that $\beta_j = -\beta'$ (if such a $j$ exists), we have
	\begin{equation*}
		\limsup_{N \to \infty} \left| \frac{1}{|\Phi_N|} \sum_{q \in \Phi_N} f(T^qa) e_{\Q}(q\beta') b(\beta') - d_{-\beta} \right| < \eps.
	\end{equation*}
	Thus,
	\begin{equation*}
		\lim_{N \to \infty} \frac{1}{|\Phi_N|} \sum_{q \in \Phi_N} f(T^qa) e_{\Q}(q\beta') b(\beta') = d_{-\beta} = \int_Z \E{f}{Z}(z) \cdot z(\beta)~dz
	\end{equation*}
	as desired.
\end{proof}

Now we may combine the previous two lemmas to prove Proposition \ref{prop: divisible and topological factors}.

\begin{proof}[Proof of Proposition \ref{prop: divisible and topological factors}]
    We construct the desired extension $\pi$ through a tower of three extensions.
	First apply Lemma \ref{lem: topological factors} to obtain an extension $\pi_1 : (X_1, \mu_1, T_1) \to (X, \mu, T)$, a F{\o}lner sequence $\Psi_1$, and a transitive point $a_1 \in X_1$ such that
	\begin{itemize}
		\item	$\pi_1$ is a topological factor map and a measurable isomorphism,
		\item	$\pi_1(a_1) = a$,
		\item	$a_1 \in \gen(\mu_1, \Psi_1)$, and
		\item	$(X_1, \mu_1, T_1)$ has topological Abramov factors.
	\end{itemize}
	Since $(X_1, \mu_1, T_1)$ is measurably isomorphic to $(X, \mu, T)$, we have $\sigma(T_1) = \sigma(T)$.
	Now apply Lemma \ref{lem: divisible} to obtain an ergodic extension $\pi_2 : (X_2, \mu_2, T_2) \to (X_1, \mu_1, T_1)$ and a transitive point $a_2 \in X_2$ such that
	\begin{itemize}
		\item	$\pi_2$ is a topological factor map
		\item	$\pi_2(a_2) = a_1$
		\item	$a_2 \in \gen(\mu_2, \Psi_1)$, and
		\item	$\sigma(T_2) = \sigma'(T_1) = \sigma'(T)$.
	\end{itemize}
	Finally, apply Lemma \ref{lem: topological factors} once again to obtain an extension $\pi_3 : (X_3, \mu_3, T_3) \to (X_2, \mu_2, T_2)$, a F{\o}lner sequence $\Psi_3$, and a transitive point $a_3 \in X_3$ such that
	\begin{itemize}
		\item	$\pi_3$ is a topological factor map and a measurable isomorphism,
		\item	$\pi_3(a_3) = a_2$,
		\item	$a_3 \in \gen(\mu_3, \Psi_3)$, and
		\item	$(X_3, \mu_3, T_3)$ has topological Abramov factors.
	\end{itemize}
	Since $\pi_3$ is a measurable isomorphism, the discrete spectrum of $(X_3, \mu_3, T_3)$ satisfies $\sigma(T_3) = \sigma(T_2) = \sigma'(T)$.
	
	Let $(\tilde{X}, \tilde{\mu}, \tilde{T}) = (X_3, \mu_3, T_3)$, $\tilde{a} = a_3$, $\Psi = \Psi_3$, and $\pi = \pi_1 \circ \pi_2 \circ \pi_3 : \tilde{X} \to X$.
	Then we have all of the desired properties:
	\begin{itemize}
		\item	$\pi$ is a topological factor map, since it is the composition of topological factor maps,
		\item	$\pi(\tilde{a}) = \pi_1(\pi_2(\pi_3(a_3))) = \pi_1(\pi_2(a_2)) = \pi_1(a_1) = a$,
		\item	$\tilde{a} \in \gen(\tilde{\mu}, \Psi)$,
		\item	$(\tilde{X}, \tilde{\mu}, \tilde{T})$ has topological Abramov factors, and
		\item	$\sigma(\tilde{T}) = \sigma'(T)$.
	\end{itemize}
\end{proof}

For the proof of Theorem \ref{thm: density dynamical}, we will use a variant of the Furstenberg--S\'{a}rk\"{o}zy theorem over $\Q$:

\begin{theorem} \label{thm: FS rationals}
	Let $(X, \mu, T)$ be an ergodic $\Q$-system.
	Then for any nonconstant polynomial $P(x) \in \Q[x]$ and any $f \in L^2(\mu)$,
	\begin{equation*}
		\UClim_{q \in \Q} T^{P(q)}f = \int_X f~d\mu,
	\end{equation*}
	where the limit is taken in $L^2(\mu)$.
\end{theorem}

\begin{proof}
	This result can be proved using spectral methods and the Weyl-type equidistribution theorem for polynomials over $\Q$ provided by \cite[Theorem 5.2]{bm_vdC}.
	We give a different proof here based on a general polynomial ergodic theorem for field actions.
	
	Writing
    \begin{equation*}
        \UClim_{q \in \Q} T^{P(q)}f = \UClim_{q \in \Q} T^{P(q) - P(0)}(T^{P(0)}f)    
    \end{equation*}
    and noting that
    \begin{equation*}
        \int_X T^{P(0)}f~d\mu = \int_X f~d\mu,    
    \end{equation*}
    we may assume without loss of generality that $P(0) = 0$.
	Then the polynomial ergodic theorem for actions of fields due to Larick \cite[Theorem 1.1.1]{larick} shows that $\UClim_{q \in \Q} T^{P(q)}f$ is equal to the orthogonal projection of $f$ onto the space of $P$-invariant functions $I_P = \{g \in L^2(\mu) : T^{P(q)}g = g~\text{for every}~q \in \Q\}$.
	We want to show that $I_P$ is the space of constant functions.
	
	For a given $g \in L^2(\mu)$, the set of $r \in \Q$ for which $T^r g = g$ defines a group.
	Hence, if $g \in I_P$, then $T^r g = g$ for every $r$ belonging to the subgroup $G_P = \left\langle P(q) : q \in \Q \right\rangle \le (\Q,+)$ generated by the values of the polynomial $P$.
	Since $T$ is ergodic by assumption, it suffices to show that $G_P = \Q$.
	
	Suppose $d = \deg{P}$, and write $P(x) = a_d x^d + \ldots + a_1 x$ with $a_d \ne 0$.
	Let $P_0 = P$, and define $P_j(x) = P_{j-1}(x+1) - P_{j-1}(x) - P_{j-1}(1)$ for $j \in \N$.
	By construction, we have $G_{P_j} \subseteq G_P$ for every $j \in \N$.
	One can check that $\deg{P_j} = d-j$ for $j \le d$, and in particular, $P_{d-1}(x) = d! a_d x$.
	Therefore, for each $q \in \Q$, we have $q = P_{d-1} \left( \frac{q}{d! a_d} \right) \in G_P$, so $G_P = \Q$ as desired.
\end{proof}

We can now prove Theorem \ref{thm: density dynamical}.

\begin{proof}[Proof of Theorem \ref{thm: density dynamical}]
	Since $a \in \gen(\mu, \Phi)$, the measure $\mu$ is supported on the orbit-closure $\mathcal{O}(a) = \overline{\{T^q a : q \in \Q\}}$.
	Hence, replacing $X$ by $\mathcal{O}(a)$ if needed, we may assume that $a$ is a transitive point.
	Then after passing to an extension if necessary, we may assume by Proposition \ref{prop: divisible and topological factors} that $(X, \mu, T)$ has divisible discrete spectrum and topological Abramov factors.
	
	For each $t \in \Q$, the measure $\sigma^P_{T^ta}$ is $P$-progressive from $T^ta$ by Theorem \ref{thm: sigma is progressive}.
	Our goal is thus to show $\sigma^P_{T^ta}(E \times E) > 0$ for some $t \in \Q$.
	Averaging over $t$ and applying the Fubini property of uniform Ces\`{a}ro limits (see \cite[Lemma 1.1]{bl_cubic}), we have
	\begin{align*}
		\UClim_{t \in \Q} \tilde{\sigma}^P_{T^ta} & = \UClim_{t \in \Q} \UClim_{q \in \Q} \delta_{T_d^{q+t} \pi_d(a)} \times \delta_{R^{P(q)+q+t}\pi_1(a)} \\
		 & = \UClim_{q \in \Q} \UClim_{t \in \Q} \delta_{T_d^{q+t} \pi_d(a)} \times \delta_{R^{P(q)+q+t}\pi_1(a)} \\
		 & = \UClim_{q \in \Q} (T_d^q \times R^{P(q)+q})_* m_{\Delta},
	\end{align*}
	where $m_{\Delta}$ is the Haar measure on the ``diagonal''
	\begin{equation*}
		\Delta = \{(u,z) \in A_d \times Z : \pi^d_1(u) = z\}.
	\end{equation*}
	The set $\Delta$ is $(T_d \times R)$-invariant, so
	\begin{equation*}
		\UClim_{t \in \Q} \tilde{\sigma}^P_{T^ta} = \UClim_{q \in \Q} (\id_{A_d} \times R^{P(q)})_* m_{\Delta}.
	\end{equation*}
	Given continuous functions $f \in C(A_d)$ and $g \in C(Z)$, we apply Theorem \ref{thm: FS rationals} to compute
	\begin{align*}
		\UClim_{q \in \Q} \int_{A_d \times Z} f \otimes R^{P(q)}g~dm_{\Delta} & = \UClim_{q \in \Q} \int_{A_d} f \cdot T_d^{P(q)} (g \circ \pi^d_1)~dm_{A_d} \\
		 & = \left( \int_{A_d} f~dm_{A_d} \right) \left( \int_{A_d} g \circ \pi^d_1~dm_{A_d} \right) \\
		 & = \left( \int_{A_d} f~dm_{A_d} \right) \left( \int_Z g~dm_Z \right).
	\end{align*}
	Thus,
	\begin{equation*}
		\UClim_{t \in \Q} \tilde{\sigma}^P_{T^ta} = m_{A_d} \times m_Z.
	\end{equation*}
	It follows that
	\begin{equation*}
		\UClim_{t \in \Q} \sigma^P_{T^ta} = \mu \times \mu.
	\end{equation*}
	Hence, by the portmanteau lemma (see, e.g., \cite[Theorem 2.1]{billingsley}),
	\begin{equation*}
		\liminf_{N \to \infty} \frac{1}{|\Psi_N|} \sum_{t \in \Psi_N} \sigma^P_{T^ta}(E \times E) \ge \mu(E)^2
	\end{equation*}
	for every F{\o}lner sequence $\Psi$ in $\Q$.
	In particular, there exists $t \in \Q$ such that $\sigma^P_{T^ta}(E \times E) > 0$.
\end{proof}

\begin{remark} \label{rem: shift}
	If one is interested in finding the pattern $\{P(b_i) + b_j : i < j\}$ for an infinite set $B = \{b_n : n \in \N\}$ without worrying about the location of the set $B$, then the shift $t$ in Theorem \ref{thm: density dynamical} is no longer needed.
	This follows easily from the observation (using Theorem \ref{thm: FS rationals}) that the second marginal of $\sigma^P_a$ is equal to $\mu$, so $\sigma^P_a(X \times E) = \mu(E) > 0$.
	
	However, for the full configuration $\{b_i, P(b_i) + b_j : i < j\}$, the shift $t$ cannot be omitted, which we show via the following example.
	Consider the space $X = (\A/\Q)^2$ with Haar measure $\mu = m_{\A/\Q}^2$ and the skew-product action
	\begin{equation*}
		T^q(x,y) = (x + q\alpha, y + 2qx + q^2\alpha)
	\end{equation*}
	for some $\alpha \in \A \setminus \Q$.
	Then $(X, \mu, T)$ is a uniquely ergodic system.
	Let $E \subseteq X$ be the set
	\begin{equation*}
		E = U \times V,
	\end{equation*}
	where $U$ is a small open neighborhood of $0 \in \A/\Q$, and $V$ is an open set away from 0 so that $(V + U) \cap U = \es$.
	We take $a = (0,0) \in X$, which is generic for $\mu$ along every F{\o}lner sequence by unique ergodicity of $(X, T)$.
	Put $A = \{q \in \Q : T^q a \in E\}$.
	We can compute $T^q a = (q\alpha, q^2\alpha)$, so $A = \{q \in \Q : q\alpha \in U, q^2 \alpha \in V\}$.
	Thus, if $b_1, b_2 \in A$, then $(b_1^2 + b_2)\alpha = b_1^2\alpha + b_2\alpha \in V + U$.
	Therefore, $b_1^2 + b_2 \notin A$.
	It follows that $\EFS^{q^2}_X(a) \cap (E \times E) = \es$.
	Given a polynomial $P(x) \in \Q[x]$ of degree $d$, it is possible to generate similar examples with skew-product actions on $(\A/\Q)^d$.
\end{remark}


\section{Proof of Theorem \ref{thm: coloring dynamical}} \label{sec: proof finish coloring}

We now turn to proving the main coloring result, which we formulated in topological dynamical terms in Theorem \ref{thm: coloring dynamical}, reproduced here to aid in the subsequent discussion.

\ColoringDynamical*

Given a topological dynamical $\Q$-system $(X, T)$ and a transitive point $a \in X$, we can find an ergodic $T$-invariant measure $\mu$ and a F{\o}lner sequence $\Phi$ such that $a \in \gen(\mu, \Phi)$; see \cite[Proposition 3.9]{furstenberg_book} and \cite[Proposition 2.7]{bf_d*}.
Using the measure-theoretic tools from the previous sections, we can further reduce Theorem \ref{thm: coloring dynamical} to a statement about the progressive measure $\sigma^P_a$.

\begin{theorem} \label{thm: sigma measure of diagonal covering}
	Let $P(x) \in \Q[x]$ with $\deg{P} \ge 2$.
	Let $(X, \mu, T)$ be an ergodic $\Q$-system with divisible discrete spectrum and topological Abramov factors.
	Let $a \in \gen(\mu, \Phi)$.
	Let $E_1, \ldots, E_r$ be a finite open cover of $X$.
	Then $\sigma^P_a \left( \bigcup_{k=1}^r (E_k \times E_k) \right) > 0$.
\end{theorem}

\begin{proof}[Proof that Theorem \ref{thm: sigma measure of diagonal covering}$\implies$Theorem \ref{thm: coloring dynamical}]
	Suppose $(X, T)$ is a topological dynamical $\Q$-system, $a \in X$ is a transitive point, and $E_1, \ldots, E_r$ is a finite open cover of $X$.
	Let $\mu$ be an ergodic $T$-invariant measure and $\Phi$ a F{\o}lner sequence such that $a \in \gen(\mu, \Phi)$.
	After passing to an extension if necessary, we may assume by Proposition \ref{prop: divisible and topological factors} that $(X, \mu, T)$ has divisible discrete spectrum and topological Abramov factors.
	Hence, by Theorem \ref{thm: sigma measure of diagonal covering}, $\sigma^P_a \left( \bigcup_{k=1}^r (E_k \times E_k) \right) > 0$.
	But by Theorem \ref{thm: sigma is progressive}, $\sigma^P_a$ is $P$-progressive from $a$, so $\EFS^P_X(a) \cap \left( \bigcup_{k=1}^r (E_k \times E_k) \right) \ne \es$.
\end{proof}


\subsection{Reducing to averages in skew products}

The measure $\sigma^P_a$ is defined in terms of a polynomial average on the Abramov factor of order $d = \deg{P}$.
In order to prove Theorem \ref{thm: sigma measure of diagonal covering}, it is therefore enough to establish the following variant in Abramov systems with the functions $f_k$ equal to the projection of $\ind_{E_k}$ to the Abramov factor.

\begin{theorem} \label{thm: uniform bound in Abramov systems}
	Let $d, r \in \N$ with $d \ge 2$.
	There exists a constant $c = c(d,r) > 0$ such that the following holds:
	for any polynomial $P(x) \in \Q[x]$ of degree $d$, any order $d$ ergodic Abramov $\Q$-system $(X, \mu, T)$, any $a \in X$, and any measurable functions $f_1, \ldots, f_r : X \to [0,1]$ such that $\sum_{k=1}^r f_k \ge 1$, one has
	\begin{equation*}
		\sum_{k=1}^r \int_{X \times X} f_k \otimes f_k~d\sigma^P_a \ge c(d,r).
	\end{equation*}
\end{theorem}

Imposing a lower bound $c(d,r)$ depending only on the degree $d$ of $P$ and the number of functions $r$ appears to make the problem more challenging.
However, it affords us significantly more flexibility to make additional simplifications.
Indeed, using standard approximation arguments, we can reduce Theorem \ref{thm: uniform bound in Abramov systems} to the following statement involving averages of functions along polynomial sequences in the group $(\A/\Q)^l$ for some $l$.

\begin{theorem} \label{thm: coloring skew product average}
	Let $d, r \in \N$ with $d \ge 2$.
	There exists a constant $c = c(d,r) > 0$ such that the following holds:
	for any polynomial $P(x) \in \Q[x]$ of degree $d$, any $l \in \N$, any $\alpha \in (\A/\Q)^l$ with $\overline{\{q\alpha : q \in \Q\}} = (\A/\Q)^l$, any $v \in (\A/\Q)^{ld}$, and any continuous functions $f_1, \ldots, f_r : (\A/\Q)^{ld} \to [0,1]$ with $\sum_{k=1}^r f_k = 1$, one has
	\begin{equation*}
		\sum_{k=1}^r \UClim_{q \in \Q} f_k \left( \left( v_j + \sum_{i=1}^{j-1} \binom{q}{j-i} v_i + \binom{q}{j} \alpha \right)_{j=1}^d \right) \tilde{f}_k \left( v_1 + (P(q)+q)\alpha \right) \ge c(d,r),
	\end{equation*}
	where $\tilde{f}_k : (\A/\Q)^l \to [0,1]$ is defined by $\tilde{f}_k(u) = \int_{(\A/\Q)^{l(d-1)}} f_k(u, w)~dw$.
\end{theorem}

\begin{proof}[Proof that Theorem \ref{thm: coloring skew product average}$\implies$Theorem \ref{thm: uniform bound in Abramov systems}.]
	Assume that Theorem \ref{thm: coloring skew product average} holds for the given $d, r$.
	Let $P(x) \in \Q[x]$ be a polynomial of degree $d$, let $(X, \mu, T)$ be an ergodic Abramov $\Q$-system of order $d$, let $a \in X$, and let $f_1, \ldots, f_r : X \to [0,1]$ such that $\sum_{k=1}^r f_k \ge 1$.
	Let $c = c(d,r)$ be as given by Theorem \ref{thm: coloring skew product average}.
	
	Our first step is replace $f_1, \ldots, f_r$ by continuous functions.
	We may find functions $f'_k : X \to [0,1]$ such that $f'_k \le f_k$ and $\sum_{k=1}^r f'_k = 1$.
	Let $\eps > 0$.
	By Lusin's theorem, let $K \subseteq X$ be a closed set such that $\mu(K) > 1 - \eps$ and $f'_k|_K$ is continuous for each $k \in \{1, \ldots, r\}$.
	Then by the Tietze extension theorem, we may find continuous functions $g'_k : X \to [0,1]$ such that $g'_k|_K = f'_k|_K$.
	On the set $K$, we have $\sum_{k=1}^r g'_k = \sum_{k=1}^r f'_k = 1$.
	Put $g_1 = g'_1$.
	Then define $g_k$ inductively by $g_k = \min\{g'_k, 1 - (g_1 + \ldots + g_{k-1})\}$ for $k \in \{2, \ldots, r-1\}$, and let $g_r = 1 - (g_1 + \ldots + g_{r-1})$.
	Then $g_1, \ldots, g_r : X \to [0,1]$ are continuous functions satisfying $\sum_{k=1}^r g_k = 1$ and $g_k|_K = g'_k|_K = f'_k|_K$.
	Observe that
	\begin{align*}
		\sum_{k=1}^r \int_{X \times X} f_k \otimes f_k~d\sigma_a^P & \ge \sum_{k=1}^r \int_{K \times K} g_k \otimes g_k~d\sigma_a^P \\
		 & \ge \left( \sum_{k=1}^r \int_{X \times X} g_k \otimes g_k~d\sigma_a^P \right) - r\sigma_a^P \left( (X \times X) \setminus (K \times K) \right).
	\end{align*}
	Moreover,
	\begin{equation*}
		\sigma_a^P \left( (X \times X) \setminus (K \times K) \right) \le \sigma_a^P \left( X \times (X \setminus K) \right) + \sigma_a^P \left( (X \setminus K) \times X \right) = 2 \mu(X \setminus K) = 2 \eps.
	\end{equation*}
	Hence,
	\begin{equation*}
		\sum_{k=1}^r \int_{X \times X} f_k \otimes f_k~d\sigma_a^P \ge \left( \sum_{k=1}^r \int_{X \times X} g_k \otimes g_k~d\sigma_a^P \right) - 2r\eps.
	\end{equation*}
	
	Our next step is to pass to an extension of $(X, \mu, T)$.
	By Theorem \ref{thm: Abramov structure}, there exists a compact $\Q$-vector space $V$ and an element $\alpha \in V$ such that $(X, \mu, T)$ is a factor of the system $(V^d, m_V^d, S)$ with
	\begin{multline*}
		S^q(v_1, \ldots, v_d) \\
         = \left( v_1 + q\alpha, v_2 + qv_1 + \binom{q}{2} \alpha, \ldots, v_d + q v_{d-1} + \ldots + \binom{q}{d-1} v_1 + \binom{q}{d} \alpha \right),
	\end{multline*}
	and we may assume that the factor map $\pi : V^d \to X$ is continuous.
	Let $h_k = g_k \circ \pi : V^d \to [0,1]$, and choose $v \in V^d$ such that $\pi(v) = a$.
	
	Now we move to a finite-dimensional space.
	The space $V$ arises as the dual of a countable $\Q$-subspace $U \subseteq \A/\Q$.
	If $U$ is finite-dimensional of dimension $l$ (as a $\Q$-vector space), then $V$ is isomorphic to $(\A/\Q)^l$, and we do not need to carry out an additional approximation argument.
	Suppose $U$ is infinite-dimensional.
	We may write $U$ as an increasing union of subspaces $U = \bigcup_{l \in \N} U_l$ such that $\dim(U_l) = l$ for each $l \in \N$.
	Then $V = \hat{U}$ is the inverse limit of the spaces $V_l = \hat{U_l}$, and $V_l \cong (\A/\Q)^l$ for each $l \in \N$.
	For each $l \in \N$, let $\rho_l : V \to V_l$ be the projection map $\rho_l(w) = w|_{U_l}$, viewing points $w \in V$ as characters on $U$.
	We can then write $h_k(w_1, \ldots, w_d) = \lim_{l \to \infty} h_{k,l}(\rho_l(w_1), \ldots, \rho_l(w_d))$ for some sequence of continuous functions $h_{k,l} : V_l^d \to [0,1]$.
	Furthermore, we may assume that $\sum_{k=1}^r h_{k,l} = 1$ for each $l \in \N$.
	By the bounded convergence theorem,
	\begin{align*}
		\sum_{k=1}^r \int_{X \times X} g_k \otimes g_k~d\sigma_a^P & = \sum_{k=1}^r \int_{V^d \times V^d} h_k \otimes h_k~d\sigma_v^P \\
        & = \lim_{l \to \infty} \sum_{k=1}^r \int_{V_l^d \times V_l^d} h_{k,l} \otimes h_{k,l}~d\sigma_{(\rho_l(v_1), \ldots, \rho_l(v_d))}^P,
	\end{align*}
	so we now seek to control
	\begin{equation*}
		\sum_{k=1}^r \int_{V_l^d \times V_l^d} h_{k,l} \otimes h_{k,l}~d\sigma_{\rho_l(v)}^P
	\end{equation*}
	for $l \in \N$.
	
	Since the functions $h_{k,l}$ are continuous, we may use the definition of $\sigma_{\rho_l(v)}^P$ as a weak* limit of averages to compute the integral
	\begin{equation*}
		\sum_{k=1}^r \int_{V_l^d \times V_l^d} h_{k,l} \otimes h_{k,l}~d\sigma_{(\rho_l(v_1), \ldots, \rho_l(v_d))}^P
	\end{equation*}
	as an average
	\begin{multline} \label{eq: integral as average}
		\sum_{k=1}^r \UClim_{q \in \Q} h_{k,l} \left( \left( \rho_l(v_j) + \sum_{i=1}^{j-1} \binom{q}{j-i} \rho_l(v_i) + \binom{q}{j} \rho_l(\alpha) \right)_{j=1}^d \right) \cdot \\
        \cdot \tilde{h}_{k,l} \left( \rho_l(v_1) + (P(q)+q) \rho_l(\alpha) \right),
	\end{multline}
	where $\tilde{h}_{k,l} : V_l \to [0,1]$ is the function $\tilde{h}_{k,l}(u) = \int_{V_l^{d-1}} h_{k,l}(u, w)~dw$.
	By Theorem \ref{thm: coloring skew product average}, the quantity in \eqref{eq: integral as average} is bounded below by $c$, since $V_l \cong (\A/\Q)^l$.
	Thus, tracing through the inequalities above,
	\begin{equation*}
		\sum_{k=1}^r \int_{X \times X} f_k \otimes f_k~d\sigma_a^P \ge c - 2r\eps.
	\end{equation*}
	But $\eps$ was arbitrary, so
	\begin{equation*}
		\sum_{k=1}^r \int_{X \times X} f_k \otimes f_k~d\sigma_a^P \ge c
	\end{equation*}
	as desired.
\end{proof}


\subsection{Linearizing the average}

Let $d, r \in \N$ with $d \ge 2$.
Fix a degree $d$ polynomial $P(x) \in \Q[x]$, a number $l \in \N$, an element $\alpha \in (\A/\Q)^l$ satisfying $\overline{\{q\alpha : q \in \Q\}} = (\A/\Q)^l$, and an element $v \in (\A/\Q)^{ld}$.
Let $f_1, \ldots, f_r : (\A/\Q)^{ld} \to [0,1]$ with $\sum_{k=1}^r f_k = 1$.
We define a quantity
\begin{equation*}
	M = \sum_{k=1}^r \UClim_{q \in \Q} f_k \left( \left( v_j + \sum_{i=1}^{j-1} \binom{q}{j-i} v_i + \binom{q}{j} \alpha \right)_{j=1}^d \right) \tilde{f}_k \left( v_1 + (P(q)+q)\alpha \right).
\end{equation*}
Our goal is to show $M \ge c(d,r)$.

Grouping terms by degree and rearranging, we can write
\begin{multline*}
	\left( v_j + \sum_{i=1}^{j-1} \binom{q}{j-i} v_i + \binom{q}{j} \alpha \right)_{j=1}^d \\
     = (v_1, \ldots, v_d) + q(\alpha, v_1, \ldots, v_{d-1}) + \ldots + \binom{q}{d} (0, \ldots, 0, \alpha).
\end{multline*}
For notational convenience, let $v_0 = \alpha$.
Expand $v_i = (v_{i,1}, \ldots, v_{i,l})$ in coordinates for each $i \in \{0, 1, \ldots, d\}$.
Put $F_0 = \{1, \ldots, l\}$.
The assumption $\overline{\{q\alpha : q \in \Q\}} = (\A/\Q)^l$ means that $v_{0,1}, \ldots, v_{0,l}$ are linearly independent over $\Q$.
By an inductive construction, we can choose sets $F_i \subseteq \{1, \ldots, l\}$ such that for each $k \in \{0, 1, \ldots, d-1\}$,
\begin{equation*}
	\{v_{i,j} : 0 \le i \le k, j \in F_i\}
\end{equation*}
is a basis for the subspace of $(\A/\Q)$ spanned by $\{v_{i,j} : 0 \le i \le k, 1 \le j \le l\}$.
By construction, for each $i \in \{0, 1, \ldots, d-1\}$, there is a linear map $L_i : (\A/\Q)^{F_0 \times \ldots \times F_i} \to (\A/\Q)^{\{1, \ldots, l\} \setminus F_i}$ such that $(v_{i,j})_{j \notin F_i} = L_i \left( (v_{i',j'})_{i' \le i, j' \in F_{i'}} \right)$.
We can thus rewrite
\begin{multline*}
	(\alpha, v_1, \ldots, v_{d-1}) = \left( v_0, \left( (v_{1,j})_{j \in F_1}, L_1(v_0, (v_{1,j})_{j \in F_1}) \right), \ldots \right. \\
	 \left. \ldots, \left( (v_{d-1,j})_{j \in F_{d-1}}, L_{d-1}(v_0, (v_{1,j})_{j \in F_1}, \ldots, (v_{d-1,j})_{j \in F_{d-1}}) \right) \right).
\end{multline*}

Write $P(x) + x = \sum_{j=0}^d a_j \binom{x}{j}$.
Then applying Proposition \ref{prop: polynomial equidistribution},
\begin{multline*}
	M = \sum_{k=1}^r \int_{(\A/\Q)^{F_0 \times \ldots \times F_{d-1}} \times \ldots \times (\A/\Q)^{F_0}} g_k\left( x^1_0, x^2_0 + \left( x^1_1, L_1(x^1_0, x^1_1) \right), \ldots \right. \\
	 \left. \ldots, x^d_0 + \left( x^{d-1}_1, L_1(x^{d-1}_0, x^{d-1}_1) \right) + \ldots + \left( x^1_{d-1}, L_{d-1}(x^1_0, x^1_1, \ldots, x^1_{d-1}) \right) \right) \cdot \\
	 \cdot \int_{(\A/\Q)^{l(d-1)}} g_k \left( a_0 \alpha + \sum_{j=1}^d a_j x^j_0, z \right)~dz~d\bm{x},
\end{multline*}
where $g_k(x_1, \ldots, x_d) = f_k(v_1+x_1, \ldots, v_d + x_d)$.
Performing a change of variables to replace $a_0 \alpha + \sum_{j=1}^d a_j x^j_0$ by $x^d_0$, we have
\begin{align*}
	M = \sum_{k=1}^r & \int_{(\A/\Q)^{F_0 \times \ldots \times F_{d-1}} \times \ldots \times (\A/\Q)^{F_0}} g_k \Bigg( x^1_0, x^2_0 + \left( x^1_1, L_1(x^1_0, x^1_1) \right), \ldots \\
	 & \ldots, \frac{1}{a_d} x^d_0 + \left( x^{d-1}_1, L_1(x^{d-1}_0, x^{d-1}_1) \right) + \ldots \\
     & \qquad \qquad \ldots + \left( x^1_{d-1}, L_{d-1}(x^1_0, x^1_1, \ldots, x^1_{d-1}) \right)  - \frac{a_0}{a_d} \alpha - \sum_{j=1}^{d-1} \frac{a_j}{a_d} x^j_0 \Bigg) \cdot \\
	 & \quad \cdot \int_{(\A/\Q)^{l(d-1)}} g_k \left( x^d_0, z \right)~dz~d\bm{x}.
\end{align*}

Define functions $h_k : \left( (\A/\Q)^{F_0 \times \ldots \times F_{d-1}} \right)^d \to [0,1]$ by
\begin{multline*}
	h_k(\bm{x}^1, \bm{x}^2, \ldots, \bm{x}^d) \\
     = g_k \Bigg( x^1_0, x^2_0 + \left( x^1_1, L_1(x^1_0, x^1_1) \right), \ldots, \frac{1}{a_d} x^d_0 + \left( x^{d-1}_1, L_1(x^{d-1}_0, x^{d-1}_1) \right) + \ldots \\
      \qquad \qquad \ldots + \left( x^1_{d-1}, L_{d-1}(x^1_0, x^1_1, \ldots, x^1_{d-1}) \right)  - \frac{a_0}{a_d} \alpha - \sum_{j=1}^{d-1} \frac{a_j}{a_d} x^j_0 \Bigg).
\end{multline*}
Using translation invariance in the $z$-coordinate and writing $Y = (\A/\Q)^{F_0 \times \ldots \times F_{d-1}}$, the expression for $M$ simplifies to
\begin{equation} \label{eq: M Ramsey count}
	M = \sum_{k=1}^r \int_{Y^{2d-1}} h_k(y_1, \ldots, y_d) h_k(y_d, \ldots, y_{2d-1})~d\bm{y}.
\end{equation}


\subsection{Applying Ramsey's theorem}

We will apply a form of Ramsey's theorem for ordered hypergraphs to bound the right hand side of \eqref{eq: M Ramsey count} from below (in terms of $r$ and $d$).
Given a set $S$ and a number $l \in \N$, we denote by $\binom{S}{l}$ the family of all $l$-element subsets of $S$.
An \emph{$l$-uniform hypergraph} is a pair $H = (V, E)$, where $V$ is a finite set of \emph{vertices} and $E \subseteq \binom{V}{l}$ is the set of \emph{edges}.
If the vertex set is totally ordered (we will take $V = \{1, \ldots, m\}$ for some $m \in \N$), we say that $H$ is an \emph{ordered $l$-uniform hypergraph}.
The \emph{complete ordered $l$-uniform hypergraph on $n$ vertices}, which we denote by $K^{(l)}_n$, is the hypergraph with vertex set $\{1, \ldots, n\}$ and edge set $\binom{\{1, \ldots, n\}}{l}$.

Suppose $H$ is an ordered $l$-uniform hypergraph on the vertex set $\{1, \ldots, m\}$ and $G$ is an ordered $l$-uniform hypergraph on the vertex set $\{1, \ldots, n\}$ for some $m < n \in \N$.
We say that $G$ \emph{contains an ordered copy} of $H$ if there is a strictly increasing map $f : \{1, \ldots, m\} \to \{1, \ldots, n\}$ such that every edge $\{i,j\}$ in $H$ maps to an edge $\{f(i), f(j)\}$ in $G$.
The $r$-color \emph{ordered Ramsey number} of an ordered $l$-uniform hypergraph $H$, denoted by $R_{<}(H,r)$, is the smallest number $n$ such that every $r$-coloring of the edges of $K^{(l)}_n$ contains a monochromatic ordered copy of $H$.
It is an easy exercise to check that the ordered Ramsey number of a complete ordered hypergraph coincides with the usual Ramsey number of the unordered version of the complete hypergraph.
In particular, $R_{<}(H,r)$ exists and is finite by Ramsey's theorem.

Ordered Ramsey numbers were formally introduced only in the last decade (see \cite{cfls, bckk}), but several early results in Ramsey theory such as the Erd\H{o}s--Szekeres theorem on monotone paths and the ``happy ending problem,'' also of Erd\H{o}s and Szekeres, have natural formulations in the language of ordered Ramsey numbers of graphs or hypergraphs.
We refer the reader to the introduction in \cite{cfls} and \cite[Section 1.2]{bckk} for more discussion.

The following measure-theoretic version of Ramsey's theorem for ordered hypergraphs (applied to the hypergraph on $2d-1$ vertices with edges $\{1, \ldots, d\}$ and $\{d, \ldots, 2d-1\}$) gives a lower bound on $M$ via \eqref{eq: M Ramsey count} and thus establishes Theorem \ref{thm: coloring skew product average}.

\begin{theorem} \label{thm: measure theory Ramsey}
	Let $H$ be an ordered $l$-uniform hypergraph on vertex set $\{1, \ldots, m\}$, and let $r \in \N$.
	There exists a constant $c = c(H,r) > 0$ satisfying the following:
	for any probability space $(X, \mu)$ and any measurable functions $\phi_1, \ldots, \phi_r : X^l \to [0,1]$ satisfying $\sum_{k=1}^r \phi_k = 1$, one has
	\begin{equation*}
		\sum_{k=1}^r \int_{X^m} \prod_{e \in E(H)} \phi_k((x_j)_{j \in e})~d\mu^m(x_1, \ldots, x_m) \ge c.
	\end{equation*}
\end{theorem}

\begin{proof}
	The idea of the proof of Theorem \ref{thm: measure theory Ramsey} is to view the functions $\phi_1, \ldots, \phi_r$ as a model for a random $r$-coloring of the edges of a complete $l$-uniform hypergraph in such a way that
	\begin{equation*}
		\sum_{k=1}^r \int_{X^m} \prod_{e \in E(H)} \phi_k((x_j)_{j \in e})~d\mu^m(x_1, \ldots, x_m)
	\end{equation*}
	computes the expected proportion of ordered copies of $H$ that are monochromatic.
	
	Let $n = R_<(H,r)$.
	Define a random $r$-coloring of the edges of $K^{(l)}_n$ as follows.
	First, we let $v_1, \ldots, v_n \in X$ be i.i.d. random variables sampled according to the probability measure $\mu$.
	Then, for each $1 \le i_1 < i_2 < \dots < i_l \le n$, we assign the color of the edge $c(\{i_1, \ldots, i_l\}) \in \{1, \ldots, r\}$ as the random variable taking value $k$ with probability $\phi_k(v_{i_1}, \ldots, v_{i_l})$, with the color assignment done independently for the different subsets $\{i_1, \ldots, i_l\}$.
	
	For each $k \in \{1, \ldots, r\}$, let
	\begin{equation*}
		M_k = \{(i_1, \ldots, i_m) : 1 \le i_1 < \ldots < i_m \le n, \forall e \in E(H), c(\{i_j : j \in e\}) = k\}
	\end{equation*}
	be the (random) set of monochromatic ordered copies of $H$ of color $k$, and let $M = \bigcup_{k=1}^r M_k$.
	By the choice of $n$ via Ramsey's theorem,
	\begin{equation*}
		|M| \ge 1.
	\end{equation*}
	On the other hand, we can compute the expected number of monochromatic ordered copies of $H$ by
	\begin{multline*}
		\mathbb{E}[|M|] = \sum_{k=1}^r \mathbb{E}[|M_k|] = \sum_{k=1}^r \sum_{1 \le i_1 < \ldots < i_m \le n} \mathbb{P}[\forall e \in E(H), c(\{i_j : j \in e\}) = k] \\
		 = \sum_{k=1}^r \sum_{1 \le i_1 < \ldots < i_m \le n} \int_{X^m} \prod_{e \in E(H)} \mathbb{P}[c(\{i_j : j \in e\}) = k \mid \forall \ell, v_{i_\ell} = x_\ell]~d\mu^m(x_1, \ldots, x_m) \\
		 = \binom{n}{m} \sum_{k=1}^r \int_{X^m} \prod_{e \in E(H)} \phi_k((x_j)_{j \in e})~d\mu^m(x_1, \ldots, x_m)
	\end{multline*}
	Thus,
	\begin{equation*}
		\sum_{k=1}^r \int_{X^m} \prod_{e \in E(H)} \phi_k((x_j)_{j \in e})~d\mu^m(x_1, \ldots, x_m) \ge \binom{R_{<}(H,r)}{m}^{-1}.
	\end{equation*}
\end{proof}


\section{A polynomial Wiener--Wintner theorem for $\Q$-actions} \label{sec: ww}

The goal of this section is to prove Theorem \ref{thm: polynomial ww adelic}, restated below for convenience.

{\renewcommand\footnote[1]{}\WW*}

Our method of proof is based on a proof of a corresponding result for totally ergodic $\Z$-systems due to Lesigne \cite{lesigne}.


\subsection{Spectral disjointness}

We first prove that polynomially twisted averages converge to 0 if the leading coefficient of the polynomial is not an eigenvalue of the system:

\begin{lemma} \label{lem: spectral disjointness}
	Let $(X, \mu, T)$ be an ergodic $\Q$-system with divisible discrete spectrum.
	Let $\Phi$ be a tempered F{\o}lner sequence in $\Q$.
	Let $f \in L^1(\mu)$.
	Then there exists a set $X_0 \subseteq X$ with $\mu(X_0) = 1$ satisfying the following property: if $d \in \N$, $\alpha \notin \sigma(T)$, and $Q(x) \in \A[x]$ with $\deg{Q} < d$, then
	\begin{equation*}
		\lim_{N \to \infty} \frac{1}{|\Phi_N|} \sum_{q \in \Phi_N} f(T^qx) \cdot e_{\Q} \left( q^d \alpha + Q(q) \right) = 0.
	\end{equation*}
	for every $x \in X_0$.
\end{lemma}

\begin{remark}
	Taking $f = 1$, Lemma \ref{lem: spectral disjointness} says that polynomial sequences are equidistributed in $\A/\Q$.
	This was previously observed (along with a multidimensional generalization) in \cite[Theorem 5.2]{bm_vdC}.
\end{remark}

\begin{proof}
	Note that
	\begin{equation*}
		\sup_{N \in \N} \sup_{P(x) \in \A[x]} \left| \frac{1}{|\Phi_N|} \sum_{q \in \Phi_N} f(T^qx) \cdot e_{\Q}(P(q)) \right| \le \sup_{N \in \N} \frac{1}{|\Phi_N|} \sum_{q \in \Phi_N} |f(T^qx)|.
	\end{equation*}
	Therefore, by the maximal ergodic theorem (see \cite[Theorem 3.2]{lindenstrauss}), it suffices to prove the lemma under the assumption $f \in C(X)$.
	
	Let $X_0$ be the set of generic points for $\mu$ along $\Phi$.
	That is, $x \in X_0$ if and only if
	\begin{equation*}
		\lim_{N \to \infty} \frac{1}{|\Phi_N|} \sum_{q \in \Phi_N} g(T^qx) = \int_X g~d\mu
	\end{equation*}
	for every $g \in C(X)$.
	By Lindenstrauss's pointwise ergodic theorem (see \cite[Theorem 1.2]{lindenstrauss}), $\mu(X_0) = 1$.
	We will prove that this set $X_0$ has the desired property.
	
	We induct on $d$.
	Suppose $d=1$.
	We want to show: if $\alpha \notin \sigma(T)$ and $x \in X_0$, then
	\begin{equation*}
		\lim_{N \to \infty} \frac{1}{|\Phi_N|} \sum_{q \in \Phi_N} f(T^qx) \cdot e_{\Q}(q\alpha) = 0.
	\end{equation*}
	Let $(\Phi_{N_k})_{k \in \N}$ be an arbitrary subsequence such that
	\begin{equation*}
		\lim_{k \to \infty} \frac{1}{|\Phi_{N_k}|} \sum_{q \in \Phi_{N_k}} f(T^qx) \cdot e_{\Q}(q\alpha)
	\end{equation*}
	exists.
	Passing to a further subsequence if necessary, we may assume that the sequence
	\begin{equation*}
		\frac{1}{|\Phi_{N_k}|} \sum_{q \in \Phi_{N_k}} \delta_{T^q x} \times \delta_{q\alpha}
	\end{equation*}
	converges in the weak* topology to a probability measure $\lambda$ on $X \times \A/\Q$.
	The measure $\lambda$ is a joining of $(X, \mu, T)$ and $(\A/\Q, m_{\A/\Q}, R_{\alpha})$, where $R_{\alpha}^q z = z + q\alpha$ for $z \in \A/\Q$ and $q \in \Q$.
	The system $(\A/\Q, m_{\A/\Q}, R_{\alpha})$ is a system with discrete spectrum, and $\sigma(R_{\alpha}) = \{q\alpha : q \in \Q\}$ is disjoint from $\sigma(T)$.
	Therefore, the systems $(X, \mu, T)$ and $(\A/\Q, m_{\A/\Q}, R_{\alpha})$ are spectrally disjoint, hence disjoint in the sense of Furstenberg (this implication was shown in \cite[Theorem 2.1]{hp} for $\Z$-systems and generalizes to actions of abelian groups; see, e.g., \cite[p. 114]{kl}).
	Thus, $\lambda = \mu \times m_{\A/\Q}$, so
	\begin{align*}
		\lim_{k \to \infty} \frac{1}{|\Phi_{N_k}|} \sum_{q \in \Phi_{N_k}} f(T^qx) \cdot e_{\Q}(q\alpha) & = \int_{X \times \A/\Q} f \otimes e_{\Q}~d\lambda \\
        & = \left( \int_X f~d\mu \right) \left( \int_{\A/\Q} e_{\Q}~dm_{\A/\Q} \right) = 0.
	\end{align*}
	
	Suppose now that the lemma holds for some $d \in \N$.
	Let $\alpha \in \A/\Q$ with $\alpha \notin \sigma(T)$, and suppose $Q(x) \in \A[x]$ is a polynomial of degree at most $d$.
	Let $u(q) = f(T^qx) \cdot e_{\Q} (q^{d+1} \alpha + Q(q)) \in \C$ for $q \in \Q$.
	Fix $r \in \Q \setminus \{0\}$, and note that
	\begin{equation*}
		u(q+r) \overline{u(q)} = (\Delta_r f)(T^qx) \cdot e_{\Q} \left( (d+1) q^d r \alpha + Q_r(q) \right)
	\end{equation*}
	for some polynomial $Q_r$ of degree $\deg{Q_r} \le d-1$.
	Since $\sigma(T)$ is divisible, we have $(d+1)r\alpha \notin \sigma(T)$ and $\Delta_r f \in C(X)$, so by the induction hypothesis,
	\begin{equation*}
		\lim_{N \to \infty} \frac{1}{|\Phi_N|} \sum_{q \in \Phi_N} u(q+r) \overline{u(q)} = 0.
	\end{equation*}
	Hence, by Lemma \ref{lem: vdC},
	\begin{equation*}
		\lim_{N \to \infty} \frac{1}{|\Phi_N|} \sum_{q \in \Phi_N} u(q) = 0
	\end{equation*}
	as desired.
\end{proof}


\subsection{A skew-product construction for systems with divisible discrete spectrum}

A feature of systems with divisible discrete spectrum is that the eigenfunctions, which take values in the unit circle $\bbS^1$, can be lifted to functions taking values in the divisible group $\A/\Q$:

\begin{lemma} \label{lem: adelic eigenfunction}
	Let $(X, \mu, T)$ be an ergodic $\Q$-system with divisible discrete spectrum.
	Suppose $\alpha \in \sigma(T)$.
	Then there exists a measurable function $g : X \to \A/\Q$ such that
	\begin{equation*}
		T^q g = g + q\alpha
	\end{equation*}
	for every $q \in \Q$.
\end{lemma}

\begin{proof}
	The lemma states that there is a factor map $g : (X, \mu, T) \to (\A/\Q, m_{\A/\Q}, R_{\alpha})$, where $R_{\alpha}^q$ is the rotation by $q\alpha$ for $q \in \Q$.
	This can be seen as a consequence of the Halmos--von Neumann theorem together with the assumption that $\sigma(T)$ is divisible.
	We give a sketch below.
	
	If $\alpha = 0$, we can take $g$ to be a constant function.
	Suppose $\alpha \ne 0$.
	By the Halmos--von Neumann theorem (see, e.g., \cite[pp. 46--48]{halmos}), the Kronecker factor of $(X, \mu, T)$ is isomorphic to the rotational system $(Z, m_Z, R)$, where $Z$ is the dual group of $\sigma(T)$ and $R$ is the $\Q$-action given by
	\begin{equation*}
		(R^qz)(\beta) = e_{\Q}(q\beta) z(\beta)
	\end{equation*}
	for $q \in \Q$, $z \in Z$ and $\beta \in \sigma(T)$.
	Since $\sigma(T)$ is divisible, we have $H = \{q \alpha : q \in \Q\} \subseteq \sigma(T)$.
	The map $\pi_H : z \mapsto z|_H$ is a surjective group homomorphism from $Z$ to $\hat{H}$.
	The rotational action $R$ induces a rotation $S$ on $\hat{H}$ given by
	\begin{equation*}
		(S^q w)(r\alpha) = e_{\Q}(qr\alpha) w(r\alpha)
	\end{equation*}
	for $q, r \in \Q$ and $w \in \hat{H}$.
	Now, the map $q \mapsto q\alpha$ is an isomorphism from $\Q$ to $H$, which induces an isomorphism $\phi : \hat{H} \to \A/\Q$ determined by
	\begin{equation*}
		e_{\Q}(q\phi(w)) = w(q\alpha)
	\end{equation*}
	for $w \in \hat{H}$ and $q \in \Q$.
	
	Let us now consider the map $\phi \circ \pi_H : Z \to \A/\Q$.
	First, for any $z \in Z$, the element $(\phi \circ \pi_H)(z) \in \A/\Q$ is determined by the identity
	\begin{equation} \label{eq: map without rotation}
		e_{\Q} \left( r \cdot (\phi \circ \pi_H)(z) \right) = z(r\alpha) \qquad (\forall r \in \Q).
	\end{equation}
	Next, for any $q \in \Q$ and $z \in Z$, we have $(\phi \circ \pi_H)(R^q z) = \phi(S^q \pi_H(z))$, so for $r \in \Q$,
	\begin{multline} \label{eq: map with rotation}
		e_{\Q} \left( r \cdot (\phi \circ \pi_H)(R^q z) \right) = e_{\Q} \left( r \cdot \phi(S^q \pi_H(z)) \right) \\
         = (S^q \pi_H(z))(r\alpha) = (R^q z)(r\alpha) = e_{\Q}(qr\alpha) z(r\alpha).
	\end{multline}
	Combining \eqref{eq: map without rotation} and \eqref{eq: map with rotation}, we have
	\begin{equation*}
		(\phi \circ \pi_H)(R^qz) = (\phi \circ \pi_H)(z) + q\alpha
	\end{equation*}
	for $z \in Z$ and $q \in \Q$.
	We can then take $g$ to be the composition of $\phi \circ \pi_H$ with the factor map from $(X, \mu, T)$ to the Kronecker factor $(Z, m_Z, R)$.
\end{proof}

\begin{lemma} \label{lem: ergodic skew-product}
	Let $(X, \mu, T)$ be an ergodic $\Q$-system.
	Suppose $g : X \to \A/\Q$ satisfies $T^q g = g + q\alpha$ for every $q \in \Q$.
	Then there exists $c \in \A/\Q$ such that for every $k \in \N$, the action
	\begin{multline*}
		S^q(x, y_1, \ldots, y_k) \\
         = \Bigg( T^qx, y_1 + 2q (g(x)+c) + q^2\alpha, y_2 + 3q y_1 + 3q^2 (g(x)+c) + q^3\alpha, \ldots \\
		 \ldots, y_k + \sum_{j=1}^{k-1} \binom{k+1}{j} q^j y_{k-j} + (k+1) q^k (g(x)+c) + q^{k+1} \alpha \Bigg)
	\end{multline*}
	is ergodic on $\left( X \times (\A/\Q)^k, \mu \times m_{\A/\Q}^k \right)$.
\end{lemma}

\begin{proof}
	A similar statement for $\Z$-actions appears in \cite[Lemma 3]{lesigne}.
	We follow essentially the same strategy of proof.
	
	Suppose $f \in L^2(X \times (\A/\Q)^k)$ is $S$-invariant.
	We do a Fourier expansion in the adelic coordinates.
	Given $\bm{r} = (r_1, \ldots, r_k) \in \Q^k$, let $f_{\bm{r}} : X \to \C$ be the function
	\begin{equation*}
		f_{\bm{r}}(x) = \int_{(\A/\Q)^k} f(x,\bm{y}) e_{\Q}(-\bm{r} \cdot \bm{y})~d\bm{y},
	\end{equation*}
	where $\bm{r} \cdot \bm{y} = \sum_{j=1}^k r_j y_j \in \A/\Q$.
	Then $f(x,\bm{y}) = \sum_{\bm{r} \in \Q^k} f_{\bm{r}}(x) e_{\Q}(\bm{r} \cdot \bm{y})$.
	
	Applying $S$, we have
	\begin{multline*}
		\sum_{\bm{r} \in \Q^k} f_{\bm{r}}(x) e_{\Q}(\bm{r} \cdot \bm{y}) \\
         = \sum_{\bm{r} \in \Q^k} f_{\bm{r}}(T^qx) e_{\Q} \left( \sum_{j=1}^k (j+1)q^j r_j (g(x)+c) + \sum_{j=1}^k q^{j+1}r_j \alpha \right) e_{\Q} (\bm{r}'_q \cdot \bm{y}),
	\end{multline*}
	where
	\begin{equation*}
		\bm{r}'_{q,j} = \sum_{i=j}^k \binom{i+1}{j+1} q^{i-j} r_i.
	\end{equation*}
	Then matching Fourier coefficients,
	\begin{equation*}
		f_{\bm{r}'_q}(x) = f_{\bm{r}}(T^qx) e_{\Q} \left( \sum_{j=1}^k (j+1)q^j r_j (g(x)+c) + \sum_{j=1}^k q^{j+1}r_j \alpha \right).
	\end{equation*}
	In particular, $\norm{L^2(\mu)}{f_{\bm{r}'_q}} = \norm{L^2(\mu)}{f_{\bm{r}}}$.
	By Parseval's identity,
	\begin{equation*}
		\sum_{\bm{r}} \norm{L^2(\mu)}{f_{\bm{r}}}^2 = \norm{L^2\left( \mu \times m_{\A/\Q}^k \right)}{f}^2 < \infty,
	\end{equation*}
	so $\norm{L^2(\mu)}{f_{\bm{r}}} \to 0$ as $|\bm{r}| \to \infty$.
	It follows that $f_{\bm{r}} = 0$ unless $r_2 = \ldots = r_k = 0$.
	
	Put $f_r = f_{(r, \bm{0})}$ for $r \in \Q$.
	Then $f(x,\bm{y}) = \sum_{r \in \Q} f_r(x) e_{\Q}(ry_1)$ and using $S$-invariance,
	\begin{equation*}
		\sum_{r \in \Q} f_r(x) e_{\Q}(ry_1) = \sum_{r \in \Q} f_r(T^qx) e_{\Q}(2 q r (g(x)+c) + q^2 r \alpha) e_{\Q}(ry_1).
	\end{equation*}
	Thus, if $f_r \ne 0$, we have
	\begin{equation*}
		f_r(x) = f_r(T^qx) e_{\Q} (2qr (g(x)+c) + q^2 r \alpha).
	\end{equation*}
	Our goal is thus to choose $c \in \A/\Q$ such that the equation
	\begin{equation} \label{eq: functional equation for ergodicity}
		\forall q \in \Q, \Delta_q h = e_{\Q} (2qr (g(x)+c) + q^2 r \alpha)
	\end{equation}
	has no solutions with $h : X \to \bbS^1$ measurable and $r \ne 0$.
	Fix $r \ne 0$.
	We observe that if \eqref{eq: functional equation for ergodicity} has solutions $h_1$ and $h_2$ corresponding to distinct values $c_1, c_2 \in \A/\Q$, then
	\begin{equation*}
		\innprod{h_1}{h_2} = \innprod{T^qh_1}{T^qh_2} = e_{\Q}(2qr(c_1-c_2)) \innprod{h_1}{h_2}
	\end{equation*}
	for every $q \in \Q$, from which we deduce that $h_1$ and $h_2$ are orthogonal.
	Since $L^2(\mu)$ is separable, it follows that for each $r \in \Q \setminus \{0\}$, the set of values $C_r \subseteq \A/\Q$ for which \eqref{eq: functional equation for ergodicity} has solutions is countable.
	Therefore, $\bigcup_{r \in \Q \setminus \{0\}} C_r$ is countable, and for any choice of $c \in \A/\Q \setminus \bigcup_{r \in \Q \setminus \{0\}} C_r$, the action $S$ is ergodic.
\end{proof}


\subsection{Key induction step}

\begin{lemma} \label{lem: phase polynomial induction step}
	Let $(X, \mu, T)$ be an ergodic $\Q$-system, and let $k \in \N$.
	Suppose $g : X \to \A/\Q$ satisfies $T^q g = g + q\alpha$ for every $q \in \Q$.
	Assume that
	\begin{multline*}
		S^q(x, y_1, \ldots, y_k) \\
        = \Bigg( T^qx, y_1 + 2q g(x) + q^2\alpha, y_2 + 3q y_1 + 3q^2 g(x) + q^3\alpha, \ldots \\
		 \ldots, y_k + \sum_{j=1}^{k-1} \binom{k+1}{j} q^j y_{k-j} + (k+1) q^k g(x) + q^{k+1} \alpha \Bigg)
	\end{multline*}
	is ergodic on $\left( X \times (\A/\Q)^k, \mu \times m_{\A/\Q}^k \right)$.
	If $h \in \mE_j(S)$ for some $j \in \{1, \ldots, k-1\}$, then there exist $(t_1, \ldots, t_j) \in \Q^j$ and $h_0 \in \mE_{j+1}(T)$ such that
    \begin{equation*}
        h(x, \bm{y}) = h_0(x) e_{\Q} \left( \sum_{i=1}^j t_i y_i \right).
    \end{equation*}
\end{lemma}

\begin{proof}
	We prove the lemma by induction on $j$, adapting the argument from \cite[pp. 782--783]{lesigne}.
	Suppose $h \in \mE_1(S)$, say with $T^qh = e_{\Q}(q\beta) h$.
	Expanding $h(x,\bm{y}) = \sum_{\bm{r} \in \Q^k} h_{\bm{r}}(x) e_{\Q}(\bm{r} \cdot \bm{y})$ as a Fourier series in the $\bm{y}$ coordinates as in the proof of Lemma \ref{lem: ergodic skew-product}, the Fourier coefficients satisfy
	\begin{equation*}
		h_{\bm{r}'_q}(x) e_{\Q}(q\beta) = h_{\bm{r}}(T^qx) e_{\Q} \left( \sum_{j=1}^k (j+1)q^j r_j g(x) + \sum_{j=1}^k q^{j+1}r_j \alpha \right)
	\end{equation*}
	for
	\begin{equation*}
		\bm{r}'_{q,j} = \sum_{i=j}^k \binom{i+1}{j+1} q^{i-j} r_i.
	\end{equation*}
	In particular, $\norm{L^2(\mu)}{h_{\bm{r}'_q}} = \norm{L^2(\mu)}{h_{\bm{r}}}$, so arguing as in Lemma \ref{lem: ergodic skew-product}, $h_{\bm{r}} = 0$ whenever $(r_2, \ldots, r_k) \ne \bm{0}$.
	
	Now let $h_r = h_{(r,\bm{0})}$ for $r \in \Q$.
	If $r, s \in \Q$ and $h_r, h_s \ne 0$, then
	\begin{equation} \label{eq: h_r/h_s behavior}
		\frac{h_r}{h_s} = T^q \left( \frac{h_r}{h_s} \right) e_{\Q} \left( 2q (r-s)g + q^2 (r-s)\alpha \right).
	\end{equation}
	Let $\psi(x,\bm{y}) = e_{\Q}((r-s)y_1)$.
	Then
	\begin{equation*}
		\psi \left( S^q(x, \bm{y}) \right) = e_{\Q} \left( 2q(r-s)g(x) + q^2(r-s)\alpha \right) \psi(x,\bm{y}).
	\end{equation*}
	Hence, the function $\left( \frac{h_r}{h_s} \circ \pi_X \right) \cdot \psi$ is $S$-invariant.
	By ergodicity of $S$, it follows that $r = s$.
	Thus, there is a unique $t_1 \in \Q$ such that $h_{t_1} \ne 0$, whence $h(x,\bm{y}) = h_{t_1}(x) e_{\Q}(t_1y_1)$ as desired.
	
	Suppose now that $j \in \{2, \ldots, k-1\}$ and the lemma holds for $j-1$.
	Let $h \in \mE_j(S)$.
	By definition, this means that for every $q \in \Q$, the multiplicative derivative $\Delta_q h$ belongs to $\mE_{j-1}(S)$, so by the induction hypothesis, there exist functions $\psi_q \in \mE_j(T)$ and elements $(t_{q,1}, \ldots, t_{q,j-1}) \in \Q^{j-1}$ such that
	\begin{equation} \label{eq: derivative is a quasi-eigenfunction}
		h \left( S^q(x, \bm{y}) \right) = \psi_q(x) e_{\Q} \left( t_{q,1}y_1 + \ldots + t_{q,j-1}y_{j-1} \right) h(x,\bm{y}).
	\end{equation}
	We once again express $h$ as a Fourier series $h(x,\bm{y}) = \sum_{\bm{r} \in \Q^k} h_{\bm{r}}(x) e_{\Q}(\bm{r} \cdot \bm{y})$.
	The identity \eqref{eq: derivative is a quasi-eigenfunction} then yields
	\begin{multline} \label{eq: derivative is a quasi-eigenfunction Fourier}
		\sum_{\bm{r} \in \Q^k} h_{\bm{r}}(T^qx) e_{\Q} \left( \sum_{l=1}^k (l+1)q^l r_l g(x) + \sum_{l=1}^k q^{l+1}r_l \alpha \right) e_{\Q}(\bm{r}'_q \cdot \bm{y}) \\
		 = \psi_q(x) e_{\Q} \left( t_{q,1}y_1 + \ldots + t_{q,j-1}y_{j-1} \right) \sum_{\bm{r} \in \Q^k} h_{\bm{r}}(x) e_{\Q}(\bm{r} \cdot \bm{y}),
	\end{multline}
	where
	\begin{equation*}
		\bm{r}'_{q,l} = \sum_{i=l}^k \binom{i+1}{l+1} q^{i-l} r_i.
	\end{equation*}
	Let $\bm{r}''_q = (r'_{q,1} - t_{q,1}, \ldots r'_{q,j-1} - t_{q,j-1}, r'_{q,j}, \ldots r'_{q,k})$.
	Then multiplying both sides of \eqref{eq: derivative is a quasi-eigenfunction Fourier} by $e_{\Q}(-t_{q,1}y_1 - \ldots - t_{q,j-1}y_{j-1})$, we have
	\begin{multline*}
		\sum_{\bm{r} \in \Q^k} h_{\bm{r}}(T^qx) e_{\Q} \left( \sum_{l=1}^k (l+1)q^l r_l g(x) + \sum_{l=1}^k q^{l+1}r_l \alpha \right) e_{\Q}(\bm{r}''_q \cdot \bm{y}) \\
		 = \psi_q(x) \sum_{\bm{r} \in \Q^k} h_{\bm{r}}(x) e_{\Q}(\bm{r} \cdot \bm{y}),
	\end{multline*}
	so
	\begin{equation} \label{eq: Fourier coefficient relation}
		h_{\bm{r}}(T^qx) e_{\Q} \left( \sum_{l=1}^k (l+1)q^l r_l g(x) + \sum_{l=1}^k q^{l+1}r_l \alpha \right) = \psi_q(x) h_{\bm{r}''_q}(x).
	\end{equation}
	In particular, $\norm{L^2(\mu)}{h_{\bm{r}}} = \norm{L^2(\mu)}{h_{\bm{r}''_q}}$, so as in Lemma \ref{lem: ergodic skew-product}, $h_{\bm{r}} = 0$ unless $\bm{r}''_q = \bm{r}$ for every $q \in \Q$.
	The equation $\bm{r}''_q = \bm{r}$ is equivalent to the system of equations
	\begin{equation*}
		\sum_{i=l+1}^k \binom{i+1}{l+1} q^{i-l} r_i = \begin{cases}
			t_{q,l}, & \text{if}~l \le j-1 \\
			0, & \text{if}~l \ge j.
		\end{cases}
	\end{equation*}
	Therefore, if $\bm{r}''_q = \bm{r}$ for every $q \in \Q$,then $r_{j+1} = \ldots = r_k = 0$, and then taking $q = 1$, we have
	\begin{equation} \label{eq: binomial relationship}
		t_{1,l} = \sum_{i=l+1}^k \binom{i+1}{l+1} r_i
	\end{equation}
	for $l \le j-1$, which uniquely determines the values of $r_2, \ldots, r_j$ as certain linear combinations of $t_{1,1}, \ldots, t_{1,j-1}$.
	Let $s_2, \ldots, s_j$ be the unique solution $(r_2, \ldots, r_j) = (s_2, \ldots, s_j)$ to the system of equations \eqref{eq: binomial relationship}.
	For $r \in \Q$, let $h_r = h_{(r, \bm{s})}$.
	Then \eqref{eq: Fourier coefficient relation} becomes
	\begin{equation} \label{eq: derivative is good}
		h_r(T^qx) e_{\Q} \left( 2qrg(x) + q^2 r \alpha + \sum_{l=2}^j (l+1)q^l s_l g(x) + \sum_{l=2}^j q^{l+1} s_l \alpha \right) = \psi_q(x) h_r(x).
	\end{equation}
	Let
	\begin{equation*}
		\phi_q(x) = e_{\Q} \left( g(x) \left( 2qr + \sum_{l=2}^j (l+1)q^l s_l \right) \right).
	\end{equation*}
	Then for $t \in \Q$,
	\begin{align*}
		\phi_q(T^tx) & = e_{\Q} \left( (g(x) + t\alpha) \left( 2qr + \sum_{l=2}^j (l+1)q^l s_l \right) \right) \\
        & = e_{\Q} \left( t\alpha \left( 2qr + \sum_{l=2}^j (l+1)q^l s_l \right) \right) \phi_q(x),
	\end{align*}
	so $\phi_q$ is an eigenfunction of $T$ with eigenvalue
	\begin{equation*}
		\left( 2qr + \sum_{l=2}^j (l+1)q^l s_l \right) \alpha.
	\end{equation*}
	Thus, \eqref{eq: derivative is good} expresses $\Delta_q h_r$ as a product of the form
	\begin{equation*}
		(\text{constant}) \cdot (\text{eigenfunction}~\phi_q) \cdot (\text{phase polynomial}~\psi_q~\text{of degree at most}~j),
	\end{equation*}
	so $h_r \in \mE_{j+1}(T)$.
	
	If $r, s \in \Q$ and $h_r, h_s \ne 0$, then dividing \eqref{eq: derivative is good} for $r$ by \eqref{eq: derivative is good} for $s$, we have
	\begin{equation*}
		\left( \frac{h_r}{h_s} \right) (T^qx) e_{\Q} \left( 2q(r-s)g(x) + q^2(r-s)\alpha \right) = \left( \frac{h_r}{h_s} \right)(x).
	\end{equation*}
	This is the same equation \eqref{eq: h_r/h_s behavior} encountered in the $j=1$ case above, so we again conclude $r = s$.
	Hence, letting $r$ be the unique value for which $h_r \ne 0$ and taking $(t_1, t_2, \ldots, t_j) = (r, s_2, \ldots, s_j)$, we have $h(x, \bm{y}) = h_r(x) e_{\Q}(t_1y_1 + \ldots + t_jy_j)$, which expresses $h$ in the desired form.
\end{proof}


\subsection{Phases with spectral frequency}

\begin{proposition} \label{prop: convergence in joining}
	Let $(X, \mu, T)$ be an ergodic $\Q$-system with divisible discrete spectrum, and let $d \in \N$.
	Let $\Phi$ be a tempered F{\o}lner sequence in $\Q$.
	Let $f \in L^1(\mu)$, and suppose $\E{f}{\mA_d} = 0$.
	For each $\alpha \in \sigma(T)$, there exists a set $X_{\alpha} \subseteq X$ with $\mu(X_{\alpha}) = 1$ satisfying the following property: if $Q(x) \in \A[x]$ with $\deg{Q} < d$, then
	\begin{equation*}
		\lim_{N \to \infty} \frac{1}{|\Phi_N|} \sum_{q \in \Phi_N} f(T^qx) \cdot e_{\Q} \left( q^d \alpha + Q(q) \right) = 0.
	\end{equation*}
	for every $x \in X_{\alpha}$.
\end{proposition}

\begin{proof}
	We induct on $d$.
	Suppose $d = 1$.
	We want to show that if $\E{f}{\mZ} = 0$, then there exists a full measure subset $X_{\alpha} \subseteq X$ such that
	\begin{equation*}
		\lim_{N \to \infty} \frac{1}{|\Phi_N|} \sum_{q \in \Phi_N} f(T^qx) \cdot e_{\Q}(q\alpha) = 0.
	\end{equation*}
	for every $x \in X_{\alpha}$.
	We will take $X_{\alpha}$ to be the set of points $x \in X$ such that
	\begin{equation*}
		\lim_{N \to \infty} \frac{1}{|\Phi_N|} \sum_{q \in \Phi_N} (\Delta_rf)(T^qx) = \int_X \Delta_r f~d\mu
	\end{equation*}
	for every $r \in \Q$.
	By Lindenstrauss's pointwise ergodic theorem (see \cite[Theorem 1.2]{lindenstrauss}), the set $X_{\alpha}$ has $\mu(X_{\alpha}) = 1$.
	Now fix $x \in X_{\alpha}$, and let $u(q) = f(T^qx) \cdot e_{\Q}(q\alpha) \in \C$ for $q \in \Q$.
	Note that
	\begin{align*}
		z(r) & = \lim_{N \to \infty} \frac{1}{|\Phi_N|} \sum_{q \in \Phi_N} u(q+r) \overline{u(q)} \\
        & = e_{\Q}(r\alpha) \cdot \lim_{N \to \infty} \frac{1}{|\Phi_N|} \sum_{q \in \Phi_N} (\Delta_r f)(T^qx) \\
        & = e_{\Q}(r\alpha) \int_X \Delta_r f~d\mu.
	\end{align*}
	Therefore, for any F{\o}lner sequence $\Psi$,
	\begin{equation*}
		\lim_{M \to \infty} \frac{1}{|\Psi_M|} \sum_{r \in \Psi_M} |z(r)| = \lim_{M \to \infty} \frac{1}{|\Psi_M|} \sum_{r \in \Psi_M} \left| \innprod{T^rf}{f} \right| = 0
	\end{equation*}
	by Proposition \ref{prop: Kronecker}(vii).
	Applying Lemma \ref{lem: vdC}, we conclude that
	\begin{equation*}
		\lim_{N \to \infty} \frac{1}{|\Phi_N|} \sum_{q \in \Phi_N} f(T^qx) \cdot e_{\Q}(q\alpha) = \lim_{N \to \infty} \frac{1}{|\Phi_N|} \sum_{q \in \Phi_N} u(q) = 0.
	\end{equation*}

	Suppose $d \ge 2$ and the proposition holds up to degree $d-1$.
	Fix $\alpha \in \sigma(T)$.
	By Lemma \ref{lem: adelic eigenfunction}, let $g : X \to \A/\Q$ be measurable with $T^q g = g + q\alpha$.
	Up to modification by a constant, we may assume that the skew-product action
	\begin{multline*}
		S^q(x, y_1, \ldots, y_{d-1}) = \Bigg( T^qx, y_1 + 2q g(x) + q^2\alpha, y_2 + 3q y_1 + 3q^2 g(x) + q^3\alpha, \ldots \\
		  \ldots, y_{d-1} + \sum_{j=1}^{d-2} \binom{d}{j} q^j y_{d-1-j} + d q^{d-1} g(x) + q^d \alpha \Bigg)
	\end{multline*}
	is ergodic on $\left( X \times (\A/\Q)^{d-1}, \mu \times m_{\A/\Q}^{d-1} \right)$ by Lemma \ref{lem: ergodic skew-product}.
	
	Define a function $\tilde{f} : X \times (\A/\Q)^{d-1} \to \bbS^1$ by $\tilde{f}(x, \bm{y}) = f(x) e_{\Q}(y_{d-1})$.
	We claim the $\tilde{f}$ is orthogonal to $\mE_{d-1}(S)$.
	Indeed, given any function $h \in \mE_{d-1}(S)$, we may write $h(x, \bm{y}) = h_0(x) e_{\Q}(t_1y_1 + \ldots + t_{d-1}y_{d-1})$ for some $h_0 \in \mE_d(T)$ and $(t_1, \ldots, t_{d-1}) \in \Q^{d-1}$ by Lemma \ref{lem: phase polynomial induction step}, so
	\begin{equation*}
		\innprod{\tilde{f}}{h}_{L^2 \left( \mu \times m_{\A/\Q}^{d-1} \right)} = \innprod{f}{h_0}_{L^2(\mu)} \cdot \ind\{t_1 = \ldots = t_{d-2} = 0, t_{d-1} = 1\} = 0,
	\end{equation*}
	since $f$ is orthogonal to $\mE_d(T)$.
	Thus, by the induction hypothesis, there exists, for each $\beta \in \sigma(S)$, a full measure subset $Y_{\beta} \subseteq X \times (\A/\Q)^{d-1}$ such that
	\begin{equation*}
		\lim_{N \to \infty} \frac{1}{|\Phi_N|} \sum_{q \in \Phi_N} \tilde{f}(S^q(x, \bm{y})) \cdot e_{\Q} \left( q^{d-1} \beta + Q(q) \right) = 0.
	\end{equation*}
	for every $(x, \bm{y}) \in Y_{\beta}$ and every polynomial $Q$ of degree at most $d-2$.
	Also, by Lemma \ref{lem: spectral disjointness}, there exists a set $Y' \subseteq X \times (\A/\Q)^{d-1}$ of full measure such that if $\beta \notin \sigma(S)$, then
	\begin{equation*}
		\lim_{N \to \infty} \frac{1}{|\Phi_N|} \sum_{q \in \Phi_N} \tilde{f}(S^q(x, \bm{y})) \cdot e_{\Q} \left( q^{d-1} \beta + Q(q) \right) = 0
	\end{equation*}
	for every $(x, \bm{y}) \in Y'$ and every polynomial $Q$ of degree at most $d-2$.
	Let
	\begin{equation*}
		Y = Y' \cap \bigcap_{\beta \in \sigma(S)} Y_{\beta},
	\end{equation*}
	Then
	\begin{equation} \label{eq: induction hypothesis in skew-product}
		\lim_{N \to \infty} \frac{1}{|\Phi_N|} \sum_{q \in \Phi_N} \tilde{f}(S^q(x, \bm{y})) \cdot e_{\Q} \left( Q(q) \right) = 0
	\end{equation}
	for every $(x, \bm{y}) \in Y$ and every polynomial $Q$ of degree at most $d-1$.
	
	Note that
	\begin{equation} \begin{split} \label{eq: orbit in skew-product twists by polynomial}
		\tilde{f}(S^q(x, \bm{y})) & = f(T^qx) e_{\Q} \left( y_{d-1} + \sum_{j=1}^{d-2} \binom{d}{j} q^j y_{d-1-j} + d q^{d-1} g(x) + q^d \alpha \right) \\
		 & = f(T^qx) e_{\Q}(q^d \alpha + Q_{(x,\bm{y})}(q)),
	\end{split} \end{equation}
	where $Q_{(x,\bm{y})}$ is a polynomial of degree at most $d-1$ depending on $(x, \bm{y}) \in X \times (\A/\Q)^{d-1}$.
	
	By Fubini's theorem, there exists a set $X_{\alpha} \subseteq X$ with $\mu(X_{\alpha}) = 1$ such that $\left( \{x\} \times (\A/\Q)^{d-1} \right) \cap Y \ne \es$ for each $x \in X_{\alpha}$.
	Suppose $x \in X_{\alpha}$, and let $\bm{y} \in (\A/\Q)^{d-1}$ such that $(x, \bm{y}) \in Y$.
	Let $Q$ be an arbitrary polynomial of degree at most $d-1$.
	Then
	\begin{multline*}
		\lim_{N \to \infty} \frac{1}{|\Phi_N|} \sum_{q \in \Phi_N} f(T^qx) \cdot e_{\Q} \left( q^d\alpha + Q(q) \right) \\
		 = \lim_{N \to \infty} \frac{1}{|\Phi_N|} \sum_{q \in \Phi_N} \tilde{f}(S^q(x,\bm{y})) \cdot e_{\Q} \left( \underbrace{Q(q) - Q_{(x, \bm{y})}(q)}_{\text{degree at most}~d-1}  \right) = 0
	\end{multline*}
	by \eqref{eq: orbit in skew-product twists by polynomial} and \eqref{eq: induction hypothesis in skew-product}.
\end{proof}


\subsection{Finishing the proof}

\begin{proof}[Proof of Theorem \ref{thm: polynomial ww adelic}]
	By Lemma \ref{lem: spectral disjointness}, let $X' \subseteq X$ with $\mu(X') = 1$ such that
	\begin{equation*}
		\lim_{N \to \infty} \frac{1}{|\Phi_N|} \sum_{q \in \Phi_N} f(T^qx) \cdot e_{\Q} \left( q^d\alpha + Q(q) \right) = 0
	\end{equation*}
	for every $x \in X'$, every $\alpha \notin \sigma(T)$, and every polynomial $Q$ of degree at most $d-1$.
	Then by Proposition \ref{prop: convergence in joining}, for each $\alpha \in \sigma(T)$, let $X_{\alpha} \subseteq X$ with $\mu(X_{\alpha}) = 1$ such that
	\begin{equation*}
		\lim_{N \to \infty} \frac{1}{|\Phi_N|} \sum_{q \in \Phi_N} f(T^qx) \cdot e_{\Q} \left( q^d \alpha + Q(q) \right) = 0
	\end{equation*}
	for $x \in X_{\alpha}$ and every polynomial $Q$ of degree at most $d-1$.
	Taking an intersection, we have that if $x \in X' \cap \bigcap_{\alpha \in \sigma(T)} X_{\alpha}$, then
	\begin{equation*}
		\lim_{N \to \infty} \frac{1}{|\Phi_N|} \sum_{q \in \Phi_N} f(T^qx) \cdot e_{\Q} \left( P(q) \right) = 0
	\end{equation*}
	for every polynomial $P$ of degree at most $d$.
	The discrete spectrum $\sigma(T)$ is countable, so the intersection $X' \cap \bigcap_{\alpha \in \sigma(T)} X_{\alpha}$ is of full measure, and we are done.
\end{proof}


\section{Abramov $\Q$-systems} \label{sec: Abramov}

The goal of this section to prove Theorem \ref{thm: Abramov structure} about the structure of Abramov systems.


\subsection{Skew-products}

We begin by analyzing the skew-product systems of the form appearing in the statement of Theorem \ref{thm: Abramov structure}.
Consider a $\Q$-vector space $V$, an element $\alpha \in V$ generating a dense subspace, and let $S$ be the $\Q$-action on $V^k$ given by
\begin{multline*}
	S^q(v_1, \ldots, v_k) \\
    = \left( v_1 + q\alpha, v_2 + qv_1 + \binom{q}{2} \alpha, \ldots, v_k + q v_{k-1} + \ldots + \binom{q}{k-1} v_1 + \binom{q}{k} \alpha \right).
\end{multline*}
Let $W = \hat{V}$.
Then $W$ is a countable discrete abelian group (in fact, $\Q$-vector space), and by the Stone--Weierstrass theorem, $L^2(m_V)$ has an orthonormal basis $(g_{\bm{w}})_{\bm{w} \in W^k}$ given by $g_{\bm{w}} = w_1 \otimes \ldots \otimes w_k$.
We will write $W$ with additive notation.

Let us compute the multiplicative derivatives of the functions $g_{\bm{w}}$, $\bm{w} \in W^k$.
For $\bm{v} \in V^k$ and $q \in \Q$, we have
\begin{align*}
	g_{\bm{w}}(S^q \bm{v}) & = \prod_{j=1}^k w_j \left( v_j + \sum_{i=1}^{j-1} \binom{q}{j-i} v_i + \binom{q}{j} \alpha \right) \\
	 & = \left( qw_1 + \ldots + \binom{q}{k}w_k \right)(\alpha) \cdot \prod_{i=1}^k \left( \sum_{j=i}^k \binom{q}{j-i} w_j \right)(v_i).
\end{align*}
Hence,
\begin{equation*}
	\Delta_q g_{\bm{w}} = \left( qw_1 + \ldots + \binom{q}{k}w_k \right)(\alpha) \cdot g_{\bm{w}'(q)},
\end{equation*}
where
\begin{equation} \label{eq: w'}
	w'_i(q) = \sum_{j=i+1}^k \binom{q}{j-i} w_j.
\end{equation}

The upshot of this calculation is that there is a natural association between the functions $g_{\bm{w}}$ and polynomial maps from $\Q$ to $W$.
Namely, if we let $P_{\bm{w}} : \Q \to W$ be the map
\begin{equation*}
	P_{\bm{w}}(t) = \sum_{j=1}^k \binom{t}{j} w_j,
\end{equation*}
then using the Vandermonde identity for binomial coefficients, we have
\begin{equation*}
	P_{\bm{w}}(t+q) = \sum_{j=1}^k \sum_{i=0}^j \binom{t}{i} \binom{q}{j-i} w_j = \sum_{i=0}^k \binom{t}{i} \sum_{j=i}^k \binom{q}{j-i} w_j.
\end{equation*}
Therefore, $P_{\bm{w}}(t + q) - P_{\bm{w}}(t) = P_{\bm{w}'(q)}(t) + \sum_{j=0}^k \binom{q}{j}w_j$ with $\bm{w}'$ as in \eqref{eq: w'}.
Write $\partial_q$ for the discrete differencing operator $\partial_q P(t) = P(t+q) - P(t)$.
Then the map $\psi$ from the space of polynomial functions $\Q \to W$ to $L^2(m_V)$ sending $P_{\bm{w}} + w_0$ to $w_0(\alpha) g_{\bm{w}}$ is a group homomorphism satisfying
\begin{equation*}
	\psi(\partial_q P) = \Delta_q \psi(P).
\end{equation*}

In order to realize an Abramov system as a factor of a system of the form above, we therefore want to find a space of functions with a similar encoding in terms of polynomial functions on $\Q$.
For technical reasons, it will be easier to work with polynomials in terms of their coefficients.
Given $k \in \N$ we may associate to a tuple $(c, t_1, \ldots, t_k) \in \bbS^1 \times (\A/\Q)^k$ the polynomial $q \mapsto c \cdot e_{\Q} \left( \sum_{j=1}^k t_j \binom{q}{j} \right)$.
(We will show in Corollary \ref{cor: polynomial representation} below that every polynomial function from $\Q$ to $\bbS^1$ is uniquely representable in this form.)
The precise result we will prove is the following.

\begin{restatable}{theorem}{Encoding} \label{thm: labeling of polynomial phases}
	Let $(X, \mu, T)$ be an ergodic $\Q$-system.
	There is a family of injective group homomorphisms $\phi_k : \mE_k(T) \to \bbS^1 \times (\A/\Q)^k$ for $k \ge 0$ such that:
	\begin{enumerate}[(1)]
		\item	$\phi_0(c) = c$ for every $c \in \bbS^1$;
		\item	for every $j \le k$ and $f \in \mE_j(T)$,
			\begin{equation*}
				\phi_k(f) = (\phi_j(f), \bm{0});
			\end{equation*}
			and
		\item	for every $k \in \N$ and $f \in \mE_k(T)$, if $\phi_k(f) = (c, t_1, \ldots, t_k)$, then for every $q \in \Q$,
			\begin{equation*}
				\phi_{k-1}(\Delta_q f) = (e_{\Q}(t'_0(q)), t'_1(q), \ldots, t'_{k-1}(q)),
			\end{equation*}
			where
			\begin{equation*}
				t'_i(q) = \sum_{j=i+1}^k \binom{q}{j-i} t_j.
			\end{equation*}
	\end{enumerate}
\end{restatable}

The main effort of this section will be devoted to proving Theorem \ref{thm: labeling of polynomial phases}.
Let us first show that Theorem \ref{thm: Abramov structure} follows from Theorem \ref{thm: labeling of polynomial phases}.
We will need two small lemmas for this deduction.
The first lemma, which will be proved in Subsection \ref{sec: phase polynomials} below, shows that phase polynomials form natural bases for the Abramov factors.

\begin{restatable}{lemma}{Orthogonal} \label{lem: phase polynomials orthogonal}
	Let $(X, \mu, T)$ be an ergodic $\Q$-system, and let $k \in \N$.
	If $f, g \in \mE_k(T)$, then either $f$ and $g$ are constant multiples of one another or $f$ and $g$ are orthogonal.
\end{restatable}

The second lemma allows us to reduce from the difficult task of finding a factor map to the easier task of constructing a map between Hilbert spaces.

\begin{lemma} \label{lem: unitary factor map}
	Let $(X, \mu)$ and $(Y, \nu)$ be standard Borel probability spaces, and let $\Phi : L^2(\mu) \to L^2(\nu)$ be a linear map.
	Suppose $\Phi$ satisfies:
	\begin{enumerate}[(a)]
		\item	for any $f, f' \in L^2(\mu)$,
			\begin{equation*}
				\innprod{f}{f'}_{L^2(\mu)} = \innprod{\Phi(f)}{\Phi(f')}_{L^2(\nu)};
			\end{equation*}
		\item	for any $f \in L^2(\mu)$,
			\begin{equation*}
				f \in L^{\infty}(\mu) \iff \Phi(f) \in L^{\infty}(\nu);
			\end{equation*}
			and
		\item	for any $f, f' \in L^{\infty}(\mu)$,
			\begin{equation*}
				\Phi(ff') = \Phi(f) \cdot \Phi(f').
			\end{equation*}
	\end{enumerate}
	Then there exists a measure-preserving map $\tilde{\Phi} : Y \to X$ such that $\Phi(f) = f \circ \tilde{\Phi}$.
\end{lemma}

\begin{proof}
	A map $\Phi$ satisfying (a), (b), and (c) is always induced by a homomorphism of measure algebras (see \cite[Theorem 2.4]{walters}).
	The assumption that our probability spaces are standard Borel ensures that the measure algebra homomorphism is realized by a measure-preserving map (see \cite[Theorem 2.2]{walters}).
\end{proof}

\begin{proof}[Proof that Theorem \ref{thm: labeling of polynomial phases}$\implies$Theorem \ref{thm: Abramov structure}]
	Let $(X, \mu, T)$ be an Abramov system of order $k$.
	Let $\phi_k : \mE_k(T) \to \bbS^1 \times (\A/\Q)^k$ be the map provided by Theorem \ref{thm: labeling of polynomial phases}.
	Note that $\phi_k(cf) = \phi_k(c) \phi_k(f) = (c, \bm{0}) \phi_k(f)$, so the image of the map $\phi_k$ is of the form $\bbS^1 \times H$ for some subgroup $H \subseteq (\A/\Q)^k$.
	By Lemma \ref{lem: phase polynomials orthogonal}, if we let $f_{\bm{t}} \in \mE_k(T)$ be the function with $\phi_k(f_{\bm{t}}) = (1,\bm{t})$ for $\bm{t} \in H$, then $(f_{\bm{t}})_{\bm{t} \in H}$ is an orthonormal basis in $L^2(\mu)$.
	In particular, $H$ is countable.
	Let $W \subseteq \A/\Q$ be the countable $\Q$-vector space generated by the coordinate projections of $H$.
	That is, $W$ is the smallest $\Q$-vector space such that $H \subseteq W^k$.
	
	Let $V = \hat{W}$, where we treat $W$ as a discrete group.
	We will write the $\Q$-vector space $V$ with additive notation.
	Let $\alpha \in V$ be the element defined by $\alpha(w) = e_{\Q}(w)$ for $w \in W$.
	We then define an action $S$ on $V^k$ by
	\begin{multline*}
		S^q(v_1, \ldots, v_k) \\
         = \left( v_1 + q\alpha, v_2 + qv_1 + \binom{q}{2} \alpha, \ldots, v_k + q v_{k-1} + \ldots + \binom{q}{k-1} v_1 + \binom{q}{k} \alpha \right).
	\end{multline*}
	For $\bm{w} \in W^k$, we let $g_{\bm{w}}$ be the function $g_{\bm{w}}(\bm{v}) = \prod_{j=1}^k v_j(w_j)$.
	As described at the beginning of this subsection, the functions $(g_{\bm{w}})_{\bm{w} \in W^k}$ form an orthonormal basis in $L^2(m_V^k)$ and satisfy
	\begin{equation} \label{eq: derivative of character}
		\Delta_q g_{\bm{w}} = e_{\Q} \left( \sum_{j=1}^k \binom{q}{j} w_j \right) \cdot g_{\bm{w}'(q)}
	\end{equation}
	with $\bm{w}'$ given by \eqref{eq: w'}.
	
	Define $\Phi : L^2(\mu) \to L^2(m_V^k)$ by
	\begin{equation*}
		\Phi \left( \sum_{\bm{t} \in H} c(\bm{t}) f_{\bm{t}} \right) = \sum_{\bm{t} \in H} c(\bm{t}) g_{\bm{t}}
	\end{equation*}
	for $c \in \ell^2(H)$.
	Since $(f_{\bm{t}})_{\bm{t} \in H}$ and $(g_{\bm{t}})_{\bm{t} \in H}$ are orthonormal systems in $L^2(\mu)$ and $L^2(m_V^k)$ respectively, $\Phi$ is a linear isometry.
	Let us check that $\Phi$ also satisfies conditions (b) and (c) from Lemma \ref{lem: unitary factor map}.
	
	Observe that for $\bm{t}, \bm{t}' \in H$, $f_{\bm{t}} \cdot f_{\bm{t}'} = f_{\bm{t} + \bm{t}'}$ and $g_{\bm{t}} \cdot g_{\bm{t}'} = g_{\bm{t} + \bm{t}'}$.
	It follows that
	\begin{equation*}
		\Phi(f f') = \Phi(f) \cdot \Phi(f')
	\end{equation*}
	when $f$ and $f'$ are finite linear combinations of the basis functions $(f_{\bm{t}})_{\bm{t} \in H}$.
	But finite linear combinations are dense in $L^{\infty}(\mu)$ with respect to the $L^2$ norm and $\Phi$ is continuous, so (c) holds.
	
	It remains to check (b).
	This property is more subtle than the others, and we proceed in stages.
	We will first prove the property: if $f \in L^2(\mu)$ and $f \ge 0$, then $\Phi(f) \ge 0$.
	To see this, note that every nonnegative function can be approximated by a nonnegative simple function, so by linearity of $\Phi$, it suffices to show that $\Phi(\ind_E) \ge 0$ for every measurable set $E \subseteq X$.
	But using property (c), the function $g = \Phi(\ind_E)$ satisfies $g = g^2$, so $g$ is almost everywhere $\{0,1\}$-valued.
	
	As a consequence, if $f \in L^{\infty}(\mu)$ is real-valued, then by linearity of $\Phi$, we have
	\begin{equation*}
		0 \le \Phi(\norm{\infty}{f} - f) = \norm{\infty}{f} - \Phi(f),
	\end{equation*}
	so $\norm{\infty}{\Phi(f)} \le \norm{\infty}{f}$.
	The map $\Phi^{-1} : \Phi(L^2(\mu)) \to L^2(\mu)$ is a linear isometry satisfying the multiplication property (c), so the same argument applied to $\Phi^{-1}$ shows that $\norm{\infty}{f} \le \norm{\infty}{\Phi(f)}$.
	Thus, (b) holds.
	
	By Lemma \ref{lem: unitary factor map}, there exists a measure-preserving map $\pi : Y \to X$ such that $\Phi(f) = f \circ \pi$.
	To show that $\pi$ is a factor map, we need to check the additional property: for every $q \in \Q$ and $f \in L^2(\mu)$,
	\begin{equation} \label{eq: intertwine}
		\Phi(T^qf) = S^q \Phi(f).
	\end{equation}
	Applying property (3) from Theorem \ref{thm: labeling of polynomial phases},
	\begin{equation*}
		T^q f_{\bm{t}} = f_{\bm{t}} \cdot e_{\Q} \left( \sum_{j=1}^k \binom{q}{j} t_j \right) \cdot f_{\bm{t}'(q)},
	\end{equation*}
	so by property (c),
	\begin{equation*}
		\Phi(T^q f_{\bm{t}}) = g_{\bm{t}} \cdot e_{\Q} \left( \sum_{j=1}^k \binom{q}{j} t_j \right) \cdot g_{\bm{t}'(q)}.
	\end{equation*}
	By \eqref{eq: derivative of character}, this is equal to $S^q g_{\bm{t}} = S^q \Phi(f_{\bm{t}})$.
	Therefore, \eqref{eq: intertwine} holds for $f = f_{\bm{t}}$ and hence for general $f$ by applying linearity and continuity of $\Phi$.
	
	\bigskip
	
	To see that we have the topological statement claimed at the end of Theorem \ref{thm: Abramov structure}, we make the following observation.
	The space $\Phi(L^2(\mu)) \subseteq L^2(m_V^k)$ is spanned by the functions $\{g_{\bm{t}} : \bm{t} \in H\}$.
	If we let $\tilde{X} = \hat{H}$, then there is a continuous surjective group homomorphism $\pi : V \to \tilde{X}$.
	Using the map $\Phi$, it is not hard to check that $(X, \mu, T)$ is measurably isomorphic to the system $(\tilde{X}, \tilde{\mu}, \tilde{T})$, where $\tilde{\mu} = m_{\tilde{X}}$ is the Haar measure on $\tilde{X}$ and $\tilde{T}$ is the action induced by $S$.
\end{proof}


\subsection{Polynomials over $\Q$}

The goal for the remainder of the section is to prove Theorem \ref{thm: labeling of polynomial phases}.
An important intermediate step is to classify polynomial functions from $\Q$ to $\bbS^1$ using a similar labeling by tuples in $\bbS^1 \times (\A/\Q)^k$.

We define a differencing operator on functions from $\Q$ to $\bbS^1$ by $\partial_q \varphi : t \mapsto \varphi(t+q) \overline{\varphi(t)}$.
We then recursively define $\partial^{k+1}_{q_1, \ldots, q_k, q_{k+1}}$ by $\partial^{k+1}_{q_1, \ldots, q_k, q_{k+1}} = \partial_{q_{k+1}} \circ \partial^k_{q_1, \ldots, q_k}$ for $k \in \N$ and $(q_1, \ldots, q_k, q_{k+1}) \in \Q^{k+1}$.
For $k \ge 0$, we say $\varphi : \Q \to \bbS^1$ is a \emph{polynomial of degree at most $k$} if $\partial^{k+1}_{q_1, \ldots, q_k, q_{k+1}} \varphi(t) = 1$ for every $q_1, \ldots, q_k, q_{k+1}, t \in \Q$.
Let $\mathcal{P}_k$ denote the space of polynomials of degree at most $k$ for each $k \ge 0$.

We want to produce a canonical form for polynomials over $\Q$.
We start by producing a canonical form for multilinear functions, which we will then lift to polynomials.
For $k \in \N$, let $\mathcal{M}_k$ be the space of $\Z$-multilinear functions from $\Q^k$ to $\bbS^1$.
That is $\eta \in \mathcal{M}_k$ if and only if for each $j \in \{1, \ldots, k\}$ and $(q_1, \ldots, q_{j-1}, q_{j+1}, \ldots, q_k)$, the map $q \mapsto \eta(q_1, \ldots, q_{j-1}, q, q_{j+1}, \ldots, q_k)$ is a group homomorphism.

For $k \ge 0$ and $\varphi : \Q \to \bbS^1$, we define $D^k \varphi : \Q^k \to \bbS^1$ by
\begin{equation*}
	D^k \varphi(q_1, \ldots, q_k) = \partial^k_{q_1, \ldots, q_k} \varphi(0) = \sum_{J \subseteq \{1, \ldots, k\}} C^{k-|J|} \varphi \left( \sum_{j \in J} q_j \right),
\end{equation*}
where $C$ is the complex conjugation map.

\begin{lemma} \label{lem: multilinearization}
	Let $\varphi : \Q \to \bbS^1$, and let $k \in \N$.
	Then
	\begin{enumerate}[(1)]
		\item	$\varphi \in \mathcal{P}_k$ if and only if $D^k \varphi \in \mathcal{M}_k$, and
		\item	the map $D^k : \mathcal{P}_k \to \mathcal{M}_k$ is a group homomorphism with $\ker(D^k) = \mathcal{P}_{k-1}$.
	\end{enumerate}
\end{lemma}

\begin{proof}
	Let us first prove (1).
	The function $D^k\varphi$ is symmetric, so it suffices to prove that for fixed $(q_1, \ldots, q_{k-1})$, the function
	\begin{equation*}
		\psi(q) = D^k\varphi(q_1, \ldots, q_{k-1}, q)
	\end{equation*}
	is a homomorphism.
	Letting $\xi = \partial^{k-1}_{q_1, \ldots, q_{k-1}} \varphi$, we can write $\psi(q) = \partial_q \xi(0)$.
	Note that $\xi \in \mathcal{P}_1$ for every choice of $(q_1, \ldots, q_{k-1}) \in \Q^{k-1}$ if and only if $\varphi \in \mathcal{P}_k$.
	Thus, we are reduced to showing that the function $\psi(q) = \partial_q \xi(0)$ is a homomorphism if and only if $\xi \in \mathcal{P}_1$.
	In other words, it suffices to prove the $k=1$ case of the lemma.
	
	Observe that for $q, r \in \Q$,
	\begin{align*}
		D\varphi(q+r) & = \varphi(q+r) \overline{\varphi(0)} \\
        & = \varphi(q) \overline{\varphi(0)} \cdot \varphi(r) \overline{\varphi(0)} \cdot \varphi(q+r) \overline{\varphi(q)} \overline{\varphi(r)} \varphi(0) \\
        & = D\varphi(q) \cdot D\varphi(r) \cdot \partial^2_{q,r}\varphi(0).
	\end{align*}
	Hence, $D\varphi$ is a homomorphism if and only if $\partial^2_{q,r}\varphi(0) = 1$ for every $q, r \in \Q$.
	Recalling that $\varphi \in \mathcal{P}_1$ if and only if $\partial^2_{q,r}\varphi(t) = 1$ for every $q, r, t \in \Q$, we are now reduced to showing that if $\partial^2_{q,r}\varphi(0) = 1$ for every $q, r \in \Q$, then $\partial^2_{q,r}\varphi$ is the constant 1 function for $q, r \in \Q$.
	This follows immediately from the identity
	\begin{equation} \label{eq: determined by values at 0}
		\partial^2_{q,r} \varphi(t) = \partial^2_{t+q,r} \varphi(0) \cdot \overline{\partial^2_{t,r}\varphi(0)}.
	\end{equation}
	
	Now we prove (2).
	By definition, it is clear that $D^k(\varphi \psi) = D^k \varphi \cdot D^k \psi$, so $D^k$ is a group homomorphism.
	Observe that $\varphi \in \ker(D^k)$ if and only if for every $(q_1, \ldots, q_k) \in \Q^k$,
	\begin{equation*}
		D^k \varphi(q_1, \ldots, q_k) = \partial^k_{q_1, \ldots, q_k} \varphi(0) = 1.
	\end{equation*}
	Using the identity \eqref{eq: determined by values at 0}, this is equivalent to having $\partial^k_{q_1, \ldots, q_k} \varphi(t) = 1$ for every $(q_1, \ldots, q_k) \in \Q^k$ and $t \in \Q$, but this is exactly what it means for $\varphi$ to be a polynomial of degree at most $k-1$.
\end{proof}

We now want to better understand the multilinear map $D^k \varphi$.
The following proposition characterizes multilinear maps.

\begin{proposition} \label{prop: multilinear}
	Let $k \in \N$, and let $\eta \in \mathcal{M}_k$.
	Then there exists $\alpha \in \A/\Q$ such that for every $(q_1, \ldots, q_k) \in \Q$,
	\begin{equation*}
		\eta(q_1, \ldots, q_k) = e_{\Q} \left( q_1 \cdot \ldots \cdot q_k \alpha \right).
	\end{equation*}
\end{proposition}

\begin{proof}
	When $k = 1$, this is a statement that the Pontryagin dual of $\Q$ (as a discrete group) is isomorphic to $\A/\Q$, which we discussed above in Subsection \ref{sec: Abramov characteristic}.
	
	Suppose the proposition holds for some $k \in \N$, and let $\eta \in \mathcal{M}_{k+1}$.
	By the induction hypothesis, there exists a map $\alpha : \Q \to \A/\Q$ such that
	\begin{equation} \label{eq: multilinear IH}
		\eta(q_1, \ldots, q_k, q_{k+1}) = e_{\Q} \left( q_1 \cdot \ldots \cdot q_k \alpha(q_{k+1}) \right)
	\end{equation}
	for every $(q_1, \ldots, q_k, q_{k+1}) \in \Q^{k+1}$.
	For fixed $q_{k+1}$, since \eqref{eq: multilinear IH} holds for $q_1 = \ldots = q_{k-1} = 1$ and arbitrary $q_k$, the element $\alpha(q_{k+1})$ is uniquely determined.
	Moreover, since $\eta$ is multilinear, $\alpha$ is in fact a homomorphism from $\Q$ to $\A/\Q$.
	Thus, $\alpha(q_{k+1}) = q_{k+1} \alpha(1)$ and we are done.
\end{proof}

As a consequence, we obtain an analogue of Theorem \ref{thm: labeling of polynomial phases} for polynomial functions defined on $\Q$.

\begin{corollary} \label{cor: polynomial representation}
	Let $k \ge 0$, and let $\varphi \in \mathcal{P}_k$.
	Then there exists a unique tuple $(c, a_1, \ldots, a_k) \in \bbS^1 \times (\A/\Q)^k$ such that for every $q \in \Q$,
	\begin{equation*}
		\varphi(q) = c \cdot e_{\Q} \left( \sum_{j=1}^k a_j \binom{q}{j} \right).
	\end{equation*}
	Moreover, the map that sends $\varphi$ to $(c, a_1, \ldots, a_k)$ is a group isomorphism, and $\partial_q \varphi$ corresponds to the tuple $(\varphi(q) \cdot \overline{c}, a'_1(q), \ldots, a'_{k-1}(q), 0)$ with
	\begin{equation*}
		a'_i(q) = \sum_{j=i+1}^k a_j \binom{q}{j-i}.
	\end{equation*}
\end{corollary}

\begin{proof}
	The ``moreover'' statement follows from the first statement by a straightforward computation, so we prove only the first part of the corollary.
	
	We induct on $k$.
	The space $\mathcal{P}_0$ is the space of constant functions, so there is nothing to prove in the base case $k=0$.
	
	Suppose the corollary holds for some $k \ge 0$, and let $\varphi \in \mathcal{P}_{k+1}$.
	By Lemma \ref{lem: multilinearization}(1) and Proposition \ref{prop: multilinear}, there exists a unique $a_{k+1} \in \A/\Q$ such that for every $(q_1, \ldots, q_k, q_{k+1}) \in \Q^{k+1}$,
	\begin{equation*}
		D^{k+1} \varphi(q_1, \ldots, q_k, q_{k+1}) = e_{\Q} \left( q_1 \cdot \ldots \cdot q_k \cdot q_{k+1} a_{k+1} \right)
	\end{equation*}
	Define $\psi : \Q \to \bbS^1$ by $\psi(q) = e_{\Q} \left( a_{k+1} \binom{q}{k+1} \right)$.
	Then $D^{k+1} \psi = D^{k+1} \varphi$, so $\varphi \cdot \overline{\psi} \in \mathcal{P}_k$ by Lemma \ref{lem: multilinearization}(2).
	Thus, by the induction hypothesis, there is a unique choice $(c, a_1, \ldots, a_k) \in \bbS^1 \times (\A/\Q)^k$ such that
	\begin{equation*}
		(\varphi \cdot \overline{\psi})(q) = c \cdot e_{\Q} \left( \sum_{j=1}^k a_j \binom{q}{j} \right).
	\end{equation*}
	Multiplying both sides by $\psi$, we obtain the desired expression for $\varphi$.
\end{proof}


\subsection{Phase polynomials} \label{sec: phase polynomials}

We now turn our attention to the spaces of phase polynomials $\mE_k(T)$ as preparation for proving Theorem \ref{thm: labeling of polynomial phases}.

\begin{lemma}
	Let $(X, \mu, T)$ be an ergodic $\Q$-system, and let $k \ge 0$.
	Then the set $\mE_k(T)$ of phase polynomials of degree at most $k$ is a group under pointwise multiplication.
\end{lemma}

\begin{proof}
	We induct on $k$.
	For $k = 0$, $\mE_0(T)$ is the set of constant functions taking values in the unit circle, so $\mE_0(T)$ is isomorphic to the group $\bbS^1$.
	
	Suppose that $\mE_k(T)$ is a group for some $k \ge 0$.
	Let $f, g \in \mE_{k+1}(T)$.
	Then there are functions $\zeta_q, \xi_q \in \mE_k(T)$ such that $T^q f = \zeta_q \cdot f$ and $T^q g = \xi_q \cdot g$ for $q \in \Q$.
	Hence, $T^q (fg) = (\zeta_q \xi_q) \cdot fg$, and $\zeta_q \xi_q \in \mE_k(T)$ by the induction hypothesis, so $fg \in \mE_{k+1}$.
\end{proof}

For a function $f : X \to \bbS^1$ and $k \in \N$, let $\Delta^k f : \Q^k \times X \to \bbS^1$ be the map defined by $(\Delta^k f)(q_1, \ldots, q_k; x) = \Delta^k_{q_1, \ldots, q_k} f(x)$.
If $(X, \mu, T)$ is ergodic and $f \in \mE_k(T)$, then $\Delta^k f$ is equal almost everywhere to a constant function in $x$, so we may view $\Delta^k f$ as a function from $\Q^k$ to $\bbS^1$.

\begin{lemma} \label{lem: derivative is multilinear}
	Let $(X, \mu, T)$ be an ergodic $\Q$-system, and let $k \in \N$.
	If $f \in \mE_k(T)$, then $\Delta^k f \in \mathcal{M}_k$.
\end{lemma}

\begin{proof}
	For $k = 1$, we note that if $f \in \mE_1(T)$, then $T^q f = \lambda(q) f$ for some $\lambda : \Q \to \bbS^1$, and $\lambda = \Delta f$ must be a homomorphism, since $T$ is a group action.
	
	Suppose the lemma holds for some $k \in \N$, and let $f \in \mE_{k+1}(T)$.
	Then for each $q_{k+1} \in \Q$, $\Delta_{q_{k+1}} f \in \mE_k(T)$, so $\Delta^k \left( \Delta_{q_{k+1}} f \right) \in \mathcal{M}_k$ by the induction hypothesis.
	Hence, the map
	\begin{equation*}
		\Delta^{k+1} f : (q_1, \ldots, q_k, q_{k+1}) \mapsto \Delta^{k+1}_{q_1, \ldots, q_k, q_{k+1}} f = \left( \Delta^k (\Delta_{q_{k+1}} f) \right)(q_1, \ldots, q_k)
	\end{equation*}
	is multilinear in the first $k$ coordinates.
	But $\Delta^{k+1} f$ is a symmetric function, so $\Delta^{k+1}f \in \mathcal{M}_{k+1}$.
\end{proof}

By Proposition \ref{prop: multilinear} and Lemma \ref{lem: derivative is multilinear}, we may define a map $\delta_k : \mE_k(T) \to \A/\Q$ by
\begin{equation*}
	(\Delta^k f)(q_1, \ldots, q_k) = e_{\Q} \left( q_1 \cdot \ldots \cdot q_k \cdot \delta_k(f) \right).
\end{equation*}
The element $\delta_k(f)$ will serve as the ``leading coefficient'' for the encoding of $f$ as a polynomial in Theorem \ref{thm: labeling of polynomial phases}.

\begin{proposition} \label{prop: polynomials as adeles}
	Let $(X, \mu, T)$ be an ergodic $\Q$-system, and let $k \in \N$.
	Then $\delta_k : \mE_k(T) \to \A/\Q$ is a group homomorphism and $\ker(\delta_k) = \mE_{k-1}(T)$.
\end{proposition}

\begin{proof}
	Given $f, g \in \mE_k(T)$, we have $\Delta^k (fg) = (\Delta^k f) (\Delta^k g)$, so $\delta_k(f \cdot g) = \delta_k(f) + \delta_k(g)$.
	That is, $\delta_k$ is a group homomorphism.
	
	Note that $\delta_k(f) = 0$ if and only if $\Delta^k f = 1$.
	Thus, by the definition of $\mE_{k-1}(T)$, we have $\ker(\delta_k) = \mE_{k-1}(T)$.
\end{proof}

We can now give a short proof of Lemma \ref{lem: phase polynomials orthogonal}, restated below for convenience.

\Orthogonal*

\begin{proof}
	By considering the function $h = f \overline{g}$, it suffices to prove that if $h \in \mE_k(T)$ for some $k \in \N$ and $h$ is nonconstant, then $\int_X h~d\mu = 0$.
	
	Assume that $k \in \N$ is minimal.
	That is, $h \in \mE_k(T) \setminus \mE_{k-1}(T)$.
	Then $\delta_k(h) \ne 0$, so we may compute the Host--Kra uniformity seminorm
	\begin{equation*}
		\seminorm{U^k}{h}^{2^k} = \UClim_{\bm{q} \in \Q^k} \int_X \Delta^k_{\bm{q}} h~d\mu = \UClim_{\bm{q} \in \Q^k} e_{\Q} \left( q_1 \cdot \ldots \cdot q_k \cdot \delta_k(h) \right) = 0.
	\end{equation*}
	By monotonicity of the Host--Kra seminorms (see \cite[Lemma 3.9]{hk} and \cite[Lemma A.20]{btz}), we conclude that $\int_X h~d\mu = 0$.
\end{proof}


\subsection{Encoding of phases polynomials as $\Q$-polynomials}

We have gathered all of the ingredients to prove Theorem \ref{thm: labeling of polynomial phases}, which we reproduce below for the convenience of the reader.

\Encoding*

\begin{proof}
	We construct the isomorphisms $\phi_k$ by induction on $k$.
	The isomorphism $\phi_0$ is defined by (1).
	
	Suppose we have constructed $\phi_0, \ldots, \phi_k$ for some $k \ge 0$.
	Let $f \in \mE_{k+1}(T)$.
	Put $t_{k+1} = \delta_{k+1}(f)$.
	For each $q \in \Q$, the function $\Delta_q f$ is a phase polynomial of degree at most $k$, so let $(c(q), t_1(q), \ldots, t_k(q)) = \phi_k(\Delta_q f) \in \bbS^1 \times (\A/\Q)^k$.
	
	\bigskip
	
	\noindent \textbf{Claim.} There exists a unique tuple $\bm{t} = (t_1, \ldots, t_k) \in (\A/\Q)^k$ such that
	\begin{equation} \label{eq: c polynomial}
		c(q) = e_{\Q} \left( \sum_{j=1}^{k+1} \binom{q}{j} t_j \right)
	\end{equation}
	and
	\begin{equation} \label{eq: t_i polynomial}
		t_i(q) = \sum_{j=i+1}^{k+1} \binom{q}{j-i} t_j
	\end{equation}
	for $q \in \Q$ and $i \in \{1, \ldots, k\}$.
	
	\bigskip
	
	Before proving the claim, let us show how it can be used to finish the proof of the theorem.
	For $\bm{t} \in (\A/\Q)^{k+1}$ and $q \in \Q$, let $\bm{s}(\bm{t}, q) \in (\A/\Q)^{k+1}$ be the tuple $(s_0(\bm{t},q), \ldots, s_k(\bm{t},q))$ with coordinates
	\begin{equation*}
		s_i(\bm{t},q) = \sum_{j=i+1}^{k+1} \binom{q}{j-i} t_j.
	\end{equation*}
	Then for each $\bm{t} \in (\A/\Q)^{k+1}$, we consider the space
	\begin{equation*}
		\mathcal{F}_{\bm{t}} = \left\{ f \in \mE_{k+1}(T) : \phi_k(\Delta_q f) = (e_{\Q}(s_0(\bm{t},q)), s_1(\bm{t},q), \ldots, s_k(\bm{t},q))~\text{for every}~q \in \Q \right\}.
	\end{equation*}
	The space $\mathcal{F}_{\bm{t}}$ may be empty.
	If $\mathcal{F}_{\bm{t}} \ne \es$ and $f, g \in \mathcal{F}_{\bm{t}}$, then $f$ and $g$ are constant multiples of each other: indeed, since $\phi_k$ is injective, we have $\Delta_q (f\overline{g}) = 1$, so $f \overline{g}$ is constant by ergodicity of $T$.

	By the claim, $\mE_{k+1}(T) = \bigcup_{\bm{t} \in (\A/\Q)^{k+1}} \mathcal{F}_{\bm{t}}$.
	Moreover, for $\bm{t}, \bm{t}' \in (\A/\Q)^{k+1}$, since $\bm{s}(\bm{t} + \bm{t}',\cdot) = \bm{s}(\bm{t},\cdot) + \bm{s}(\bm{t}',\cdot)$ and $\phi_k$ is a homomorphism, we have $\mathcal{F}_{\bm{t}} \cdot \mathcal{F}_{\bm{t}'} = \mathcal{F}_{\bm{t}+\bm{t}'}$.
	Let $H = \{\bm{t} \in (\A/\Q)^{k+1} : \mathcal{F}_{\bm{t}} \ne \es\}$, and let $K = \{\bm{s} \in (\A/\Q)^k : (\bm{s},0) \in H\}$.
	By Proposition \ref{prop: polynomials as adeles}, $\mE_k(T) = \bigcup_{\bm{s} \in K} \mathcal{F}_{(\bm{s},0)}$, so $\phi_k(\mE_k(T)) = \bbS^1 \times K$.
	Let $\psi_k : \bbS^1 \times K \to \mE_k(T)$ be the inverse map.
	We can extend the map $\psi_k(1, \cdot) : K \to \mE_k(T)$ to a homomorphism $\vartheta : H \to \mE_{k+1}(T)$ such that $\vartheta(\bm{s},0) = \psi_k(1,\bm{s})$ for $\bm{s} \in K$ and $\vartheta(\bm{t}) \in \mathcal{F}_{\bm{t}}$ for each $\bm{t} \in H$.
	For $f \in \mathcal{F}_{\bm{t}}$, we define $\phi_{k+1}(f) = (f \cdot \overline{\vartheta(\bm{t})}, \bm{t})$.
	By the observation that elements of $\mathcal{F}_{\bm{t}}$ differ up to constant multiples, the function $f \cdot \overline{\vartheta(\bm{t})}$ is constant, so $\phi_{k+1}(f) \in \bbS^1 \times H$.
	
	Let us check that $\phi_{k+1}$ is an injective homomorphism.
	Suppose $f, g \in \mE_{k+1}(T)$ are arbitrary.
	Let $\bm{t}, \bm{t}' \in H$ such that $f \in \mathcal{F}_{\bm{t}}$ and $g\in \mathcal{F}_{\bm{t}'}$.
	Then $fg \in \mathcal{F}_{\bm{t} + \bm{t}'}$, so
	\begin{equation*}
		\phi_{k+1}(fg) = (fg \cdot \overline{\vartheta(\bm{t} + \bm{t}')}, \bm{t} + \bm{t}') = ((f \cdot \overline{\vartheta(\bm{t})})(g \cdot \overline{\vartheta(\bm{t}')}), \bm{t} + \bm{t}') = \phi_{k+1}(f) \phi_{k+1}(g).
	\end{equation*}
	Thus, $\phi_{k+1}$ is a homomorphism.
	Now, if $\phi_{k+1}(f) = (1, \bm{0})$, then by (2) and then (1),
	\begin{equation*}
		f = \vartheta(\bm{0}) = \psi_k(1,\bm{0}) = \phi_k^{-1}(1, \bm{0}) = \phi_0^{-1}(1) = 1.
	\end{equation*}
	This proves that $\phi_{k+1}$ is injective.
	
	We must check that $\phi_{k+1}$ satisfies properties (2) and (3).
	Property (3) holds by the claim.
	By the induction hypothesis, it suffices to check that $j = k$ case of (2).
	Suppose $f \in \mE_k(T)$, and let $\phi_k(f) = (c, \bm{s}) \in \bbS^1 \times K$.
	Then by property (1) and property (2) for $\phi_k$, we have
	\begin{equation*}
		f = \psi_k(c, \bm{s}) = \psi_k(c, \bm{0}) \cdot \psi_k(1, \bm{s}) = c \cdot \psi_k(1, \bm{s}).
	\end{equation*}
	Therefore,
	\begin{equation*}
		\phi_{k+1}(f) = (f \cdot \overline{\vartheta(\bm{s},0)}, \bm{s}, 0) = (f \cdot \overline{\psi_k(1, \bm{s})}, \bm{s},0) = (c, \bm{s}, 0) = (\phi_k(f), 0).
	\end{equation*}
	
	\bigskip
	
	Now we prove the claim.
	If $k = 0$, then using the definition of $\phi_0$, we can write $\Delta_q f = c(q)$.
	On the other hand, $\Delta_q f = e_{\Q}(q \cdot \delta_1(f)) = e_{\Q}(qt_1)$, so $c(q) = e_{\Q}(qt_1)$.
	This verifies the claim in the case $k = 0$.
	
	To give a sense for the strategy, we first work out the case $k = 1$.
	Taking a second multiplicative derivative and applying property (3), we have
	\begin{equation*}
		e_{\Q}(q_1q_2t_2) = \Delta^2_{q_1, q_2} f = \Delta_{q_2} (\Delta_{q_1} f) = \phi_0(\Delta_{q_2} (\Delta_{q_1} f)) = e_{\Q}(q_2 \cdot t_1(q_1)).
	\end{equation*}
	Therefore, $t_1(q) = q t_2$ as desired.
	Now we analyze $c(q)$.
	On the one hand
	\begin{equation*}
		\phi_1(\Delta_{q+r} f) = (c(q+r), (q+r)t_2).
	\end{equation*}
	On the other hand,
	\begin{equation*}
		\phi_1(\Delta_{q+r} f) = \phi_1(\Delta_q f) \cdot \phi_1(\Delta_r f) \cdot \phi_1(\Delta^2_{q,r} f) = (c(q)c(r)e_{\Q}(qrt_2), (q+r)t_2).
	\end{equation*}
	Hence, $c(q+r) = c(q)c(r) e_{\Q}(qrt_2)$.
	It follows that $c : \Q \to \bbS^1$ is a quadratic polynomial so takes the form $c(q) = e_{\Q}(qt_1 + \binom{q}{2}t_2)$ for some $t_1 \in \A/\Q$ by Corollary \ref{cor: polynomial representation}.
	
	Now we consider general $k \ge 2$.
	Applying property (3) for the lower degree phase polynomial $\Delta_q f$ for some $q \in \Q$ by induction,
	\begin{equation*}
		\phi_{k-1}(\Delta^2_{q,r}f) = \left( e_{\Q} \left( \sum_{j=1}^k \binom{r}{j} t_j(q) \right), \left( \sum_{j=i+1}^k \binom{r}{j-i} t_j(q) \right)_{i=1}^{k-1} \right).
	\end{equation*}
	Since the function $\Delta^2 f$ is symmetric, we thus have
	\begin{equation*}
		\sum_{j=i+1}^k \binom{r}{j-i} t_j(q) = \sum_{j=i+1}^k \binom{q}{j-i} t_j(r)
	\end{equation*}
	for every $i \in \{0, \dots, k-1\}$ and $q, r \in \Q$.
	Taking $r = 1$, we have
	\begin{equation*}
		t_{i+1}(q) = \sum_{j=i+1}^k \binom{q}{j-i} t_j(1).
	\end{equation*}
	Thus, putting $t_j = t_{j-1}(1)$ for $j \in \{2, \ldots, k-1\}$ yields \eqref{eq: t_i polynomial} so long as we can ensure that $t_k(1) = t_{k+1} = \delta_{k+1}(f)$.
	But differentiating and applying (3) repeatedly, we obtain
	\begin{equation*}
		\Delta^{k+1}_{q_1, \ldots, q_k, q} f = \phi_0(\Delta^k_{q_1, \ldots, q_k} \Delta_q f) = e_{\Q}(q_1 \cdot \ldots \cdot q_k \cdot t_k(q)),
	\end{equation*}
	so $t_k(q) = q \cdot \delta_{k+1}(f)$ as desired.
	It remains to check that \eqref{eq: c polynomial} is satisfied for some choice of $t_1 \in \A/\Q$.
	Using the identity $\phi_k(\Delta_{q+r}f) = \phi_k(\Delta_q f) \cdot \phi_k(\Delta_r f) \cdot \phi_k(\Delta^2_{q,r} f)$ as in the $k=1$ case above and focusing on the first coordinate, we have
	\begin{multline} \label{eq: derivative of c}
		c(q+r) = c(q) c(r) e_{\Q} \left( \sum_{j=1}^k \binom{r}{j} t_j(q) \right) \\
		 = c(q) c(r) e_{\Q} \left( \sum_{j=1}^k \binom{r}{j} \sum_{l=j+1}^{k+1} \binom{q}{l-j} t_l \right) \\
		 = c(q) c(r) e_{\Q} \left( \sum_{\substack{1 \le i, j \le k, \\ i+j \le k+1}} \binom{q}{i} \binom{r}{j} t_{i+j} \right).
	\end{multline}
	This show that $c(0) = 1$.
	Moreover, the derivative $\partial_r c(q) = c(q+r) \overline{c(q)}$ is a polynomial of degree at most $k$, so $c \in \mathcal{P}_{k+1}$.
	We can therefore write $c(q) = e_{\Q} \left( \sum_{j=1}^{k+1} a_j \binom{q}{j} \right)$ for a unique choice of coefficients $a_j \in \A/\Q$ by Corollary \ref{cor: polynomial representation}.
	Plugging this expression back into \eqref{eq: derivative of c}, we see that $a_j = t_j$ for $j \ge 2$.
	Letting $t_1$ be the coefficient $a_1$ completes the proof of the claim.
\end{proof}


\section{A set of positive density in the rationals not containing any shifted $\Delta$-set} \label{sec: straus example rationals}

The goal of this short section is to elaborate on the discussion after Theorem \ref{thm: density} and show that polynomials of the form $P(x) = -x + c$ are bad for sumsets.
Given an infinite set $B = \{b_n : n \in \N\} \subseteq \Q$, we call the set of differences $\{b_j - b_i : i < j\}$ a \emph{$\Delta$-set}.
In comparison to sets of the form $\{P(b_i) + b_j : i < j\}$ for polynomials $P(x) \ne -x + c$, $\Delta$-sets possess additional arithmetic structure, which allows for constructions of large subsets of $\Q$ not containing any shifted $\Delta$-set.
Recall that a set $R \subseteq \Q$ is a \emph{set of (measurable) recurrence} if for every $\Q$-system $(X, \mu, T)$ and every measurable set $E \subseteq X$ with $\mu(E) > 0$, there exists $r \in R$ such that $\mu(E \cap T^{-r}E) > 0$.

\begin{proposition}
	Suppose $B = \{b_n : n \in \N\} \subseteq \Q$ is infinite.
	Then the set of differences $\{b_j - b_i : i < j\}$ is a set of recurrence.
\end{proposition}

\begin{proof}
	This can be seen by carefully examining the standard proof of the Poincar\'{e} recurrence theorem, as noted by Furstenberg in the context of the integers; see \cite[p. 74]{furstenberg_book}, where $\Delta$-sets are given as one of the first examples of \emph{Poincar\'{e} sequences} (Furstenberg's term for sets of recurrence).
	For completeness, we include the argument here.
	Suppose $(X, \mu, T)$ is a measure-preserving $\Q$-system and $A \subseteq X$ is a measurable set with $\mu(A) > 0$.
	Consider the sequence of sets $A_n = T^{-b_n}A$.
	Since $\mu$ is a probability measure and $T$ is measure-preserving, if $N > \frac{1}{\mu(A)}$, then we must have $\mu(A_i \cap A_j) > 0$ for some $1 \le i < j \le N$.
	Using the measure preserving property again, we conclude that $\mu(A \cap T^{-(b_j-b_i)}A) > 0$.
\end{proof}

\begin{theorem} \label{thm: straus example rationals}
	Let $\Phi$ be a F{\o}lner sequence in $\Q$, and let $\eps > 0$.
	There exists a set $A \subseteq \Q$ such that $\underline{d}_{\Phi}(A) > 1 - \eps$ and $A - t$ is a set of nonrecurrence for every $t \in \Q$.
	In particular, $A - t$ does not contain a $\Delta$-set for any $t \in \Q$.
\end{theorem}

We will use the following general result from \cite{counterexample}:

\begin{theorem}[{\cite[Theorem 2.2]{counterexample}}] \label{thm: counterexample criterion}
	Let $\Gamma$ be an abelian group, and let $\Phi$ be a F{\o}lner sequence in $\Gamma$.
	Let $\mathcal{C}$ be a family of infinite subsets of $\Gamma$, and suppose there is a family $\mathcal{D}$ of subsets of $\Gamma$ with the following properties:
	\begin{enumerate}[(1)]
		\item	For any $C \in \mathcal{C}$ and any $D \in \mathcal{D}$, one has $C \cap D \in \mathcal{C}$, and
		\item	$\inf_{D \in \mathcal{D}} \overline{d}_{\Phi}(D) = 0$.
	\end{enumerate}
	Then for any $\eps > 0$, there exists $A \subseteq \Gamma$ with $\underline{d}_{\Phi}(A) > 1 - \eps$ such that for any $t \in \Gamma$, $A - t \notin \mathcal{C}$.
\end{theorem}

\begin{proof}[Proof of Theorem \ref{thm: straus example rationals}]
	Let $\mathcal{C}$ be the family of sets of recurrence and $\mathcal{D} = \mathcal{C}^*$ the family of sets of return times.
	It suffices to check that $\mathcal{C}$ and $\mathcal{D}$ satisfy the conditions in Theorem \ref{thm: counterexample criterion} for an arbitrary F{\o}lner sequence $\Phi$.
	
	Property (1) is a standard result in ergodic theory.
	We reproduce a proof here.
	Let $C$ be a set of recurrence, and let $D = \{q \in \Q : \mu(E \cap T^{-q}E) > 0\}$ for some measure-preserving $\Q$-system $(X, \mu, T)$ and measurable set $E \subseteq X$ with $\mu(E) > 0$.
	We want to show $C \cap D$ is a set of recurrence.
	Let $(Y, \nu, S)$ be another measure-preserving $\Q$-system and $F \subseteq Y$ a measurable set with $\nu(F) > 0$.
	Then since $C$ is a set of recurrence, there exists $c \in C$ such that
	\begin{equation*}
		(\mu \times \nu) \left( (E \times F) \cap (T \times S)^{-c} (E \times F) \right) > 0.
	\end{equation*}
	Therefore,
	\begin{equation*}
		\mu(E \cap T^{-c}E) > 0 \qquad \text{and} \qquad \nu(F \cap S^{-c}F) > 0.
	\end{equation*}
	The first inequality shows that $c \in D$ and the second then establishes that $C \cap D$ is a set of recurrence.
	
	To prove property (2), consider the action of $\Q$ on $\T$ by rotations $T^qx = x + q \bmod{1}$, and let $m_{\T}$ be the Haar measure on $\T$.
	Let $E = [0,\delta] \subseteq \T$.
	Then the set of return times $D = \{q \in \Q : m_{\T}(E \cap T^{-q}E) > 0\}$ can be written as $D = \{q : \|q\| < \delta\}$, where $\|\cdot\|$ is the distance to the nearest integer.
	Given any F{\o}lner sequence $\Phi$, we have
	\begin{equation*}
		d_{\Phi}(D) = 2\delta
	\end{equation*}
	by the portmanteau lemma (see, e.g., \cite[Theorem 2.1]{billingsley}).
	Since $\delta$ was arbitrary, this establishes property (2).
\end{proof}


\section{Counterexamples for polynomial sumsets with ordering reversed} \label{sec: i>j examples}

In this section, we produce examples to illustrate the necessity of the asymmetric ordered sumset condition $i < j$ appearing in Theorems \ref{thm: density P} and \ref{thm: coloring P}.


\subsection{Sets of positive density without shifted polynomial sumsets}

We first prove Theorem \ref{thm: i>j bad for density} concerning sets of positive density, restated below.
In the integer setting, an analogue of Theorem \ref{thm: i>j bad for density} in the case $d = 2$ was demonstrated by \cite[Example 3.14]{kmrr_survey} using the property that the squares $\{n^2 : n \in \N\}$ form a zero density subset of $\N$.
We use a similar approach for the rational setting.

\ReversedOrderDensity*

\begin{proof}
	Let $D = \{x^d : x \in \Q\}$.
	We claim that $d^*(D) = 0$.
	Suppose for contradiction that $d^*(D) > 0$.
	Then by Szemer\'{e}di's theorem, $D$ contains arbitrarily long arithmetic progressions.
	But $D$ has no arithmetic progression of length 4 if $d = 2$ by a theorem traditionally attributed to Fermat (see \cite[p. 440]{dickson}) and no arithmetic progression of length 3 if $d \ge 3$ by \cite{dm}.
	
	Now by \cite[Lemma 3.1]{counterexample}, there exists a set $N \subseteq \Q$ such that $d_{\Phi}(N) = 0$ and $(D+t) \setminus N$ is finite for each $t \in \Q$.
	Let $A = \Q \setminus N$.
	Then $d_{\Phi}(A) = 1$ and for each $t \in \Q$,
	\begin{equation*}
		|\{q \in \Q : q^d + t \in A\}| \le 2 \cdot |A \cap (D+t)| < \infty.
	\end{equation*}
\end{proof}


\subsection{Finite colorings without monochromatic polynomial sumsets}

Given a finite coloring of $\Q$ and an arbitrary function $f : \Q \to \Q$, Ramsey's theorem implies that there exists a monochromatic configuration of the form $\{f(b_i) + b_j : i > j\}$ for some infinite set $B = \{b_n : n \in \N\}$, so reversing the ordering $i < j$ to $i > j$ does not have any impact on ordered sumset configurations on their own from the point of view of finite colorings.
However, if we also wish for the set $B$ to be of the same color, the ordering present in Theorem \ref{thm: coloring P} is essential, as shown by Theorem \ref{thm: i>j bad for coloring}.

\ReversedOrderColoring*

\begin{proof}
	To motivate the coloring, let us first consider the corresponding problem over $\N$.
	If we take a logarithm, then for large $n$ (with $t$ fixed), $\log(n^d+t) \approx d\log{n}$.
	We then want a coloring of the positive real numbers that assigns different colors to $x$ and a small interval around $dx$.
	Taking another logarithm and noting that $\log(dx) = \log(x) + \log(d)$, we then reduce to finding a coloring of the real line that assigns different colors to $y$ and small intervals around $y + \log(d)$.
	This can be done by partitioning $\R$ into intervals of length $\frac{2}{3} \log(d)$ and coloring cyclically.
	Thus, in $\N$, we can take the coloring $c : \N \to \{0,1,2\}$ defined by
	\begin{equation*}
		c(n) = \left\lfloor \frac{3\log\log(n)}{2\log(d)} \right\rfloor \bmod{3}
	\end{equation*}
	for $n \geq 2$ and $c(1) = 0$.
	
	In order to adapt the above coloring to the rational numbers, we compose with the height function $H : \Q \to \N$.
	For a rational number $x = \frac{a}{b}$ with $\gcd(a,b) = 1$, define the \emph{height} by $H(x) = \max\{|a|,|b|\}$.
	Let $\chi : \Q \to \{0,1,2\}$ be given by $\chi = c \circ H$ so that
	\begin{equation*}
		\chi(q) = \left\lfloor \frac{3\log\log(H(q))}{2\log(d)} \right\rfloor \bmod{3}
	\end{equation*}
	for $q \notin \{-1,0,1\}$ and $\chi(-1) = \chi(0) = \chi(1) = 0$.
	Let us check that this coloring has the desired property.
	
	Let $t \in \Q$.
	The key property of the height map to control the number of solutions to $\chi(q) = \chi(q^d+t)$ is that $H(q^d + t)$ has size comparable to $H(q)^d$.
	More explicitly, we have the following inequalities:
	
	\bigskip
	
	\noindent \textbf{Claim.} For every $q \in \Q$,
	\begin{equation*}
		\frac{H(q)^d}{2H(t)^2} \le H(q^d+t) \le 2H(t)H(q)^d.
	\end{equation*}
	
	\begin{proof}[Proof of Claim] \renewcommand\qedsymbol{$\blacksquare$}
		Write $q = \frac{a}{b}$ and $t = \frac{r}{s}$ with $\gcd(a,b) = \gcd(r,s) = 1$.
		Then
		\begin{equation*}
			q^d + t = \frac{sa^d + rb^d}{sb^d}.
		\end{equation*}
		Putting this fraction into reduced form can only shrink the numerator and denominator, so we immediately have the upper bound
		\begin{align*}
			H(q^d+t) & \leq \max\{|s| |a|^d + |r| |b|^d, |s| |b|^d\} \\
            & \leq 2 \max\{|r|,|s|\} \cdot \max\{|a|^d, |b|^d\} \\
            & = 2 H(t) H(q)^d.
		\end{align*}
		
		Now let us prove the lower bound.
		Suppose $m = \gcd(sa^d + rb^d, sb^d)$ so that $H(q^d+t) = \max \left\{ \frac{|sa^d+rb^d|}{m}, \frac{|s| |b|^d}{m} \right\}$.
		Then $m \mid s(sa^d+rb^d) - r(sb^d) = s^2 a^d$ and $m \mid s(sb^d) = s^2 b^d$.
		But $a^d$ and $b^d$ are coprime, so $m \mid s^2$.
		Therefore,
		\begin{equation*}
			H(q^d+t) \geq \max \left\{ \frac{|sa^d+rb^d|}{|s|^2}, \frac{|b|^d}{|s|} \right\}.
		\end{equation*}
		Let $\delta = \min \left\{ \left( \frac{|s| - \frac{1}{2}}{|r|} \right)^{1/d}, 1 \right\}$.
		We consider separately the cases $|b| \ge \delta H(q)$ and $b < \delta H(q)$.
		If $|b| \ge \delta H(q)$, then
		\begin{equation*}
			H(q^d+t) \ge \frac{|b|^d}{|s|} \ge \frac{\delta^d}{|s|} H(q) \ge \min \left\{ \frac{|s| - \frac{1}{2}}{|r| |s|}, \frac{1}{|s|} \right\} H(q) \ge \frac{H(q)}{2H(t)^2}.
		\end{equation*}
		Now suppose $|b| < \delta H(q)$.
		In particular, $H(q) = |a|$.
		Then
		\begin{multline*}
			H(q^d+t) \ge \frac{|s| |a|^d - |r| |b|^d}{|s|^2} > \frac{(|s| - |r| \delta^d)|a|^d}{|s|^2} \\
            \ge \frac{\left( |s| - |r| \frac{|s|-\frac{1}{2}}{|r|} \right)|a|^d}{|s|^2} = \frac{|a|^d}{2|s|^2} \ge \frac{H(q)^d}{2H(t)^2}.
		\end{multline*}
	\end{proof}
	
	By the claim, we have
	\begin{equation*}
		d \log(H(q)) - C_t \le \log(H(q^d+t)) \le d \log(H(q)) + C_t
	\end{equation*}
	for $C_t = 2 \log(H(t)) + \log(2)$.
	We now apply the estimates
	\begin{equation*}
		\log(x+y) \le \log(x) + \frac{y}{x} \qquad \text{and} \qquad \log(x-y) \ge \log(x) -\frac{y}{x-y}
	\end{equation*}
	for $0 \le y < x$ to obtain the bounds (taking $x = \log(H(q))$ and $y = C_t$)
	\begin{multline*}
		\log\log(H(q)) + \log(d) - \frac{C_t}{d\log(H(q)) - C_t} \\ \le \log\log(H(q^d+t)) \\
        \le \log\log(H(q)) + \log(d) + \frac{C_t}{d\log(H(q))}
	\end{multline*}
	for $q \in \Q$ such that $H(q) > e^{C_t/d}$.
	For convenience, let $f(q) = \frac{3\log\log(H(q))}{2\log(d)}$.
	The previous inequalities then read
	\begin{equation*}
		\frac{3C_t}{2 d \log(d) \log(H(q)) - 2C_t\log(d)} \le f(q^d+t) - f(q) - \frac{3}{2} \le \frac{3C_t}{2 d \log(d) \log(H(q))}.
	\end{equation*}
	By the definition of $\chi$, if $1 \le |f(q_1) - f(q_2)| \le 2$, then $\chi(q_1) \ne \chi(q_2)$.
	Hence, if $\chi(q) = \chi(q^d+t)$, then
	\begin{equation*}
		\frac{3C_t}{2 d \log(d) \log(H(q)) - 2C_t\log(d)} > \frac{1}{2}.
	\end{equation*}
	Equivalently,
	\begin{equation*}
		H(q) < \exp \left( \frac{3+\log(d)}{d\log(d)} C_t \right).
	\end{equation*}
	There are only finitely many rational numbers $q \in \Q$ satisfying this height bound for a given value of $t \in \Q$, so $\{q \in \Q : \chi(q) = \chi(q^d+t)\}$ is a finite set.
\end{proof}


\section{Polynomial sumsets in the integers} \label{sec: integers}

Theorems \ref{thm: counterexample} and \ref{thm: not ramsey pr} show that the direct analogues of Theorems \ref{thm: density P} and \ref{thm: coloring P} fail over the integers.
There are, nevertheless, nontrivial results that one can establish related to infinite polynomial sumsets in the integers.
After exhibiting an explicit coloring behind Theorem \ref{thm: not ramsey pr} in Subsection \ref{sec: 5-coloring}, we discuss in the remaining subsection various additional conditions on sets of positive density in the integers that guarantee the presence of infinite polynomial sumset configurations.


\subsection{Failure of partition regularity for infinite polynomial sumsets over the integers} \label{sec: 5-coloring}

As promised in the introduction, we give a short proof of Theorem \ref{thm: not ramsey pr} by producing an explicit 5-coloring of the positive integers for which every infinite polynomial sumset of the form $\{b_i, b_i^2 + b_j : i < j\}$ meets at least two color classes.

We build a 5-coloring $\chi : \N \to \{\textbf{odd}, \textbf{even}_{00}, \textbf{even}_{01}, \textbf{even}_{10}, \textbf{even}_{11}\}$ as follows.
If $n \in \N$ is odd, then $\chi(n) = \textbf{odd}$.
Note that the \textbf{odd} color class cannot contain any triple $\{b_1, b_2, b_1^2 + b_2\}$, since the square of an odd number is odd and the sum of two odd numbers is even.
We split the even numbers into four color classes $\textbf{even}_{ij}$ for $(i,j) \in \{0,1\} \times \{0,1\}$ using a pair of 2-colorings.
Let $c(n) \in \{0,1\}$ to be the coefficient of $2^{2v_2(n)}$ in the binary expansion of $n$, where $v_2$ is the 2-adic valuation.
Define a ``dyadic'' coloring $d(n) = \lfloor \log_2(n) \rfloor \bmod{2} \in \{0,1\}$.
The important property of the coloring $d$ is that for every $n \in \N$, $d(2n) \ne d(n)$.
We define $\chi(n) = \textbf{even}_{ij}$ if $n$ is even, $c(n) = i$, and $d(v_2(n)) = j$.

Let $B = \{b_1 < b_2 < \ldots\}$ be an infinite set.
As noted above, if $0 \in v_2(B)$ (i.e., if $B$ contains an odd number), then $\{b_i, b_i^2 + b_j : i < j\}$ contains an even number, so it is not monochromatic.
Assume now that all elements of $B$ are even.
We consider separately the cases that $v_2(B)$ is finite and infinite.

Suppose $v_2(B)$ is finite.
Then by the pigeonhole principle, there exists a value $v \in \N$ and elements $x, y \in B$ with $x < y$ and $v_2(x) = v_2(y) = v$.
Write $x = 2^vs$ and $y = 2^vt$.
Then $s$ and $t$ are both odd numbers, which we expand in binary as $s = 1 + s_1 \cdot 2 + \ldots + s_v \cdot 2^v + \ldots$ and $t = 1 + t_1 \cdot 2 + \ldots + t_v \cdot 2^v + \ldots$.
Note that $c(x) = s_v$ and $c(y) = t_v$.
We compute
\begin{equation*}
	x^2 + y = 2^v(2^v s^2 + t) = 2^v + t_1 \cdot 2^{v+1} + \ldots + t_{v-1} \cdot 2^{2v-1} + (t_v + 1) \cdot 2^{2v} + \ldots.
\end{equation*}
Hence, $c(x^2 + y) \ne t_v = c(y)$, so $y$ and $x^2 + y$ are of different colors.

Now suppose $v_2(B)$ is infinite.
Then we may choose $x, y \in B$ with $x < y$ and $2v_2(x) < v_2(y)$.
The valuation of $x^2 + y$ is then given by $v_2(x^2 + y) = 2v_2(x)$, so $d(v_2(x^2+y)) \ne d(v_2(x))$.
Hence, $x$ and $x^2+y$ are assigned different colors by $\chi$.

\bigskip

A key feature of the argument given above is that when the set $v_2(B)$ is infinite, it cannot be bounded from above.
This very simple observation may fail if $B$ is a set of rational numbers instead of integers, so there is no contradiction with Theorem \ref{thm: coloring P}.


\subsection{Totally ergodic sets}

We now begin our discussion of conditions enabling the production of infinite polynomial sumsets in sets of integers.
The first condition relates to the dynamical property of \emph{total ergodicity}.

Fix a set $A \subseteq \N$ with $d^*(A) > 0$.
By the Furstenberg correspondence principle (Theorem \ref{thm: correspondence}), there exists an ergodic $\Z$-system $(X, \mu, T)$, a transitive point $a \in X$, a F{\o}lner sequence $\Phi$ such that $a \in \gen(\mu, \Phi)$, and a clopen set $E \subseteq X$ with $\mu(E) > 0$ such that $A = \{n \in \N : T^n a \in E\}$.
A configuration $\{b_i, P(b_i) + b_j : i < j\}$ can be produced in a shift of $A$ by finding an Erd\H{o}s--Furstenberg--S\'{a}rk\"{o}zy $P$-progression $(T^ta, x_1, x_2)$ with $x_1, x_2 \in E$ for some $t \in \Z$.
We know by Theorem \ref{thm: counterexample} that it is not always possible to find such progressions.
However, we can avoid such counterexamples by imposing an additional condition on the system $(X, \mu, T)$.

Before stating the result, we introduce the requisite terminology.
A $\Z$-system $(X, \mu, T)$ is called \emph{totally ergodic} if $T^n$ is ergodic for every $n \in \N$.
We say that a polynomial $P(x) \in \Q[x]$ is \emph{integer-valued} if $P(\Z) \subseteq \Z$.
Integer-valued polynomials can be written in the form $P(x) = \sum_{j=0}^d c_j \binom{x}{j}$ for some coefficients $c_j \in \Z$.

The next theorem implies that if $A$ has a totally ergodic Furstenberg system\footnote{Such sets are referred to as \emph{totally ergodic sets} in \cite{fish} and \cite{bf}, where it is shown that totally ergodic sets contain many interesting (finite) polynomial patterns that do not exist in arbitrary sets of positive density.}, then a shift of $A$ contains an infinite configuration of the form $\{b_i, b_i^2 + b_j : i < j\}$.

\begin{theorem} \label{thm: TE}
	Let $(X, \mu, T)$ be a totally ergodic $\Z$-system, $\Phi$ a F{\o}lner sequence, and $a \in \gen(\mu, \Phi)$.
	Then for any open set $E \subseteq X$ with $\mu(E) > 0$ and any integer-valued polynomial $P(x) \in \Q[x]$ such that both $P(x)$ and $P(x) + x$ are nonconstant, there exists $t \in \Z$ such that $\EFS^P_X(T^ta) \cap (E \times E) \ne \es$.
\end{theorem}

\begin{proof}
	The proof of Theorem \ref{thm: TE} is essentially the same as the proof of Theorem \ref{thm: density dynamical}, so we give only a sketch.
	By \cite[Lemma 5.8]{kmrr1}, we may assume that $(X, \mu, T)$ has topological Abramov factors.
	Let $P(x) \in \Q[x]$ such that both $P(x)$ and $P(x) + x$ are nonconstant, and let $d = \deg{P}$.
	Using Lesigne's polynomial Wiener--Wintner theorem \cite{lesigne} in place of Theorem \ref{thm: polynomial ww adelic}, one can establish the following variant of Theorem \ref{thm: characteristic factor}:
	\begin{quote}
		\emph{for any $f \in C(X)$ and $g \in L^2(\mu)$ and any $\eps > 0$, there exists a decomposition $f = f_s \circ \pi_d + f_u$ with $f_s \in C(A_d)$ and $f_u \in C(X)$ such that $\norm{L^1(\mu)}{\E{f_u}{\mA_d}} < \eps$ and}
		\begin{equation} \label{eq: decomposition}
			\limsup_{N \to \infty} \norm{L^2(\mu)}{\frac{1}{|\Phi_N|} \sum_{n \in \Phi_N} \left( f(T^na) \cdot T^{P(n)}g - f_s(T_d^n \pi_d(a)) \cdot T^{P(n)} \E{g}{\mZ} \right)} < \eps.
		\end{equation}
	\end{quote}
	By the assumption that $T$ is totally ergodic, one can find a skew-product transformation
	\begin{equation*}
		S(x_1, \ldots, x_d) = (x_1 + \alpha, x_2 + x_1, \ldots, x_d + x_{d-1})
	\end{equation*}
	for $(x_1, \ldots, x_d) \in G^d$, where $G$ is a finite or countably infinite product of $\T$, such that the Abramov factor $(A_d, m_{A_d}, T_d)$ is a (topological) factor of $(G^d, m_G^d, S)$; see \cite{abramov} and \cite[Lemma 4.1]{frakra}.
	Then using Weyl's equidistribution theorem \cite{weyl}, we can define measures $\tilde{\lambda}^P_{(\pi_d(a), \pi_1(x))}$ by
	\begin{equation*}
		\tilde{\lambda}^P_{(\pi_d(a), \pi_1(x))} = \UClim_{n \in \Z} \delta_{T_d^n \pi_d(a)} \times \delta_{R^{P(n)} \pi_1(x)}
	\end{equation*}
	and then lift these to measures $\lambda^P_{(a,x)}$ on $X \times X$ by putting
	\begin{equation*}
		\int_{X \times X} f \otimes g~d\lambda^P_{(a,x)} = \int_{A_d \times Z} \E{f}{A_d} \otimes \E{f}{Z}~d\tilde{\lambda}^P_{(\pi_d(a), \pi_1(x))}
	\end{equation*}
	for $f, g \in C(X)$.
	The second marginal of $\tilde{\lambda}^P_{(\pi_d(a), \pi_1(x))}$ is given by the average $\UClim_{n \in \Z} \delta_{R^{P(n)} \pi_1(x)}$.
	In a general group rotational system, such an average may produce a measure that is singular to the Haar measure on $Z$.
	However, since we assume that $T$ is totally ergodic, this average is exactly equal to the Haar measure $m_Z$ as a consequence of Weyl's equidistribution theorem ($T$ being totally ergodic is equivalent to all of the eigenvalues of $T$ being irrational, so the limiting behavior of the polynomial average boils down to equidistribution of sequences $(P(n)t)_{n \in \Z}$ with $t \in \T$ irrational).
	Hence, $\lambda^P_{(a,x)}$ is well-defined.
	The decomposition statement \eqref{eq: decomposition} shows that
	\begin{equation*}
		\lim_{N \to \infty} \frac{1}{|\Phi_N|} \sum_{n \in \Phi_N} f(T^n a) \cdot g(T^{P(n)}x) = \int_{X \times X} f \otimes g~d\lambda^P_{(a,x)}
	\end{equation*}
	in $L^2(\mu)$, where both sides are considered as functions of $x$ (cf. Theorem \ref{thm: limit formula}).
	
	We similarly define a measure $\tilde{\sigma}^P_{\pi_d(a)} = \tilde{\lambda}^{P(x)+x}_{(\pi_d(a), \pi_1(a))}$ and $\sigma^P_a = \lambda^{P(x)+x}_{(a,a)}$.
	Proposition \ref{prop: uniform bound on factor} still holds in the new setting by replacing instances of Proposition \ref{prop: polynomial equidistribution} with Weyl's equidistribution theorem, so $\sigma^P_a$ is $P$-progressive from $a$.
	To finish the proof, we then compute the average $\UClim_{t \in \N} \sigma^P_{T^ta}$.
	We use the fact that in a totally ergodic system,
	\begin{equation*}
		\UClim_{n \in \Z} T^{P(n)}f = \int_X f~d\mu
	\end{equation*}
	in $L^2(\mu)$ for $f \in L^2(\mu)$; see \cite[Lemma 3.14]{furstenberg_book} and \cite[Section 2]{ert_update}.
	Arguing as in the proof of Theorem \ref{thm: density dynamical}, we deduce that $\UClim_{t \in \N} \sigma^P_{T^ta} = \mu \times \mu$.
	Thus, for any F{\o}lner sequence $\Psi$,
	\begin{equation*}
		\liminf_{N \to \infty} \frac{1}{|\Psi_N|} \sum_{t \in \Psi_N} \sigma^P_{T^ta}(E \times E) \ge \mu(E)^2,
	\end{equation*}
	so there exists $t \in \Q$ such that $\sigma^P_{T^ta}(E \times E) > 0$, whence $\EFS^P_X(T^ta) \cap (E \times E) \ne \es$.
\end{proof}

A natural follow-up to Theorem \ref{thm: TE} is to describe sets $A \subseteq \N$ having a totally ergodic Furstenberg system.
Essentially what this says is that for any bounded function $f : \N \to \C$ in the algebra generated by $\ind_A$ and its shifts (i.e., $f$ is a linear combination of indicator functions of sets of the form $(A - t_1) \cap \ldots \cap (A - t_k)$ for some $k \in \N$ and $t_1, \ldots, t_k \in \Z$), then $f$ does not correlate with any periodic function.
A precise statement can be found in \cite[Theorem 1.2]{bf}.

One may ask what happens if we impose only the condition that $\ind_A$ does not correlate with periodic functions, rather than imposing such a condition on the entire algebra generated by $\ind_A$ and its shifts.
Interpreting ``correlation'' along an appropriately chosen F{\o}lner sequence, this corresponds dynamically to the set $E$ from the Furstenberg correspondence principle satisfying $\E{\ind_E - \mu(E)}{\mZ_{rat}} = 0$, where $\mZ_{rat}$ is the factor generated by periodic functions (sometimes called the \emph{rational Kronecker factor} or the \emph{procyclic factor}).
A more precise description of such sets is given in \cite[Section 3.2]{bf} under the moniker \emph{relatively totally ergodic sets}.
Relatively totally ergodic sets contain a rich variety of polynomial patterns (see \cite[Section 3.3]{bf}), so it is reasonable to make the following conjecture, which we state in dynamical language:

\begin{conjecture} \label{conj: trivial projection to rational factor}
	Let $(X, \mu, T)$ be an ergodic $\Z$-system, $\Phi$ a F{\o}lner sequence, and $a \in \gen(\mu, \Phi)$.
	Let $E \subseteq X$ be an open set with $\mu(E) > 0$ satisfying $\E{\ind_E - \mu(E)}{\mZ_{rat}} = 0$.
	Then for any integer-valued polynomial $P(x) \in \Q[x]$ such that both $P(x)$ and $P(x) + x$ are nonconstant, there exists $t \in \Z$ such that $\EFS^P_X(T^ta) \cap (E \times E) \ne \es$.
\end{conjecture}

The methods of this paper do not seem to be sufficient to address Conjecture \ref{conj: trivial projection to rational factor}, since defining the measure $\sigma^P_a$ required absolute continuity of a certain measure on the Kronecker factor, and this property may fail in systems with nontrivial rational Kronecker factor.


\subsection{Relatively weakly mixing sets}

As supporting evidence for Conjecture \ref{conj: trivial projection to rational factor}, we will now show that (a strengthening of) the conclusion of the conjecture holds if one strengthens the assumption $\E{\ind_E - \mu(E)}{\mZ_{rat}} = 0$ to $\E{\ind_E - \mu(E)}{\mZ} = 0$.
Our argument is based on the argument given in \cite[Proposition 3.27]{kmrr_survey} for the case $P(x) = x$.

\begin{proposition} \label{prop: trivial projection to Kronecker factor}
	Let $(X, \mu, T)$ be a $\Z$-system, $\Phi$ a F{\o}lner sequence, and $a \in \gen(\mu, \Phi)$.
	Let $E \subseteq X$ be an open set with $\mu(E) > 0$ satisfying $\E{\ind_E - \mu(E)}{\mZ} = 0$.
	Then for any nonconstant integer-valued polynomial $P(x) \in \Q[x]$, we have $\EFS^P_X(a) \cap (E \times E) \ne \es$.
\end{proposition}

\begin{proof}
	The key observation is the following:
	
	\bigskip
	
	\noindent \textbf{Claim.} For every Borel set $B \subseteq X$ and every $\eps > 0$,
	\begin{equation*}
		d^* \left( \left\{ n \in \Z : \left| \mu \left( B \cap T^{-P(n)}E \right) - \mu(B) \mu(E) \right| > \eps \right\} \right) = 0.
	\end{equation*}
	
	\begin{proof}[Proof of Claim] \renewcommand\qedsymbol{$\blacksquare$}
		Let $f = \ind_E - \mu(E)$ and $g = \ind_B$.
		First we will show
		\begin{equation*}
			\UClim_{n \in \Z} \left| \innprod{f}{T^{Q(n)}f} \right|^2 = 0
		\end{equation*}
		for every nonconstant integer-valued polynomial $Q(x) \in \Q[x]$.
		By Herglotz's theorem (see, e.g., \cite[4.13]{nadkarni}), let $\nu_f$ be a positive measure on $\T$ such that
		\begin{equation*}
			\innprod{f}{T^nf} = \hat{\nu}_f(n) = \int_{\T} e(nt)~d\nu_f(t)
		\end{equation*}
		for $n \in \Z$.
		Since $f$ is a weakly mixing function, the measure $\nu_f$ is continuous (see, e.g., \cite[Corollary 4.17]{nadkarni}).
		Let $\Theta$ be a F{\o}lner sequence in $\Z$.
		Then
		\begin{equation*}
			\frac{1}{|\Theta_N|} \sum_{n \in \Theta_N} \left| \innprod{f}{T^{Q(n)}f} \right|^2 = \frac{1}{|\Theta_N|} \sum_{n \in \Theta_N} \int_{\T^2} e(Q(n)(t-s))~d\nu_f(t)~d\nu_f(s).
		\end{equation*}
		Now, by Weyl's equidistribution theorem, we have
		\begin{equation*}
			\lim_{N \to \infty} \frac{1}{|\Theta_N|} \sum_{n \in \Theta_N} e(Q(n)x) = 0
		\end{equation*}
		for every $x \in \R \setminus \Q$.
		Therefore,
		\begin{equation*}
			\limsup_{N \to \infty} \frac{1}{|\Theta_N|} \sum_{n \in \Theta_N} \left| \innprod{f}{T^{Q(n)}f} \right|^2 \le \left( \nu_f \times \nu_f \right) \left( \left\{ (s, t) \in \T^2 : t - s \in \Q \right\} \right).
		\end{equation*}
		For each $s \in \T$, the set $\{t \in \T : t - s \in \Q\}$ is countable, so $\nu_f \left( \{t \in \T : t - s \in \Q\} \right) = 0$.
		By Fubini's theorem, it follows that
		\begin{equation*}
			\limsup_{N \to \infty} \frac{1}{|\Theta_N|} \sum_{n \in \Theta_N} \left| \innprod{f}{T^{Q(n)}f} \right|^2 = 0.
		\end{equation*}
		
		Now we want to prove
		\begin{equation*}
			\UClim_{n \in \Z} \left| \innprod{g}{T^{P(n)}f} \right|^2 = 0.
		\end{equation*}
		Let $u_n = \innprod{g}{T^{P(n)}f} T^{P(n)}f$.
		Then for any $h \in \Z$,
		\begin{equation*}
			\innprod{u_{n+h}}{u_n} = \innprod{g}{T^{P(n)}f} \innprod{g}{T^{P(n+h)}f} \innprod{T^{P(n+h)}f}{T^{P(n)}f}.
		\end{equation*}
		But $T$ is measure-preserving, so we can rewrite this as
		\begin{equation*}
			\innprod{u_{n+h}}{u_n} = \innprod{g}{T^{P(n)}f} \innprod{g}{T^{P(n+h)}f} \innprod{T^{P(n+h) - P(n)}f}{f},
		\end{equation*}
		and so by Cauchy--Schwarz,
		\begin{equation*}
			\left| \innprod{u_{n+h}}{u_n} \right| \le \norm{2}{f}^2 \norm{2}{g}^2 \left| \innprod{T^{P(n+h) - P(n)}f}{f} \right|.
		\end{equation*}
		If $\deg{P} \ge 2$, then $Q_h(n) = P(n+h) - P(n)$ is a nonconstant polynomial, so
		\begin{equation*}
			\UClim_{n \in \Z} \left| \innprod{u_{n+h}}{u_n} \right| = 0.
		\end{equation*}
		On the other hand, if $P(n) = an + b$, then $P(n+h) - P(n) = ah$, and we have
		\begin{equation*}
			\UClim_{h \in \Z} \UClim_{n \in \Z} \left| \innprod{u_{n+h}}{u_n} \right| \le \norm{2}{f}^2 \norm{2}{g}^2 \UClim_{h \in \Z} \left| \innprod{T^{ah}f}{f} \right| = 0.
		\end{equation*}
		Hence, by van der Corput (Lemma \ref{lem: vdC}), we have
		\begin{equation*}
			\UClim_{n \in \Z} u_n = 0
		\end{equation*}
		in $L^2(\mu)$.
		Taking the inner product with $g$ gives
		\begin{equation*}
			\UClim_{n \in \Z} \left| \innprod{g}{T^{P(n)}f} \right|^2 = 0
		\end{equation*}
		as desired.
	\end{proof}
	
	Let $A = \{n \in \N : T^n a \in E\}$.
	By the portmanteau lemma, $d_{\Phi}(A) \ge \mu(E) > 0$.
	Therefore by the claim, we may choose $b_1 \in A$ such that $\mu(E \cap T^{-P(b_1)}E) > 0$.
	Let $E_1 = E \cap T^{-P(b_1)}E$ and $A_1 = A \cap A - P(b_1)$.
	Note that $A_1 = \left\{ n \in \N : T^na \in E_1 \right\}$, and $d_{\Phi}(A_1) > 0$.
	At each stage, we choose $b_n \in A_{n-1}$ such that $\mu \left( E_{n-1} \cap T^{-P(b_n)} E \right) > 0$.
	In the end, we obtain a sequence $(b_n)_{n \in \N}$ in $A$ with $P(b_i) + b_j \in A$ for $i < j$.
	By passing to subsequences if necessary, we can ensure that $T^{b_n}a$ converges to some point $x_1 \in E$ and $T^{P(b_n)}x_1$ converges to a point $x_2 \in E$ so that $(a, x_1, x_2)$ is the desired Erd\H{o}s--Furstenberg--S\'{a}rk\"{o}zy $P$-progression.
\end{proof}

Given a set $A \subseteq \N$, we can give a reasonable description of when its Furstenberg system satisfies the hypothesis of Proposition \ref{prop: trivial projection to Kronecker factor}.
Let $\Phi$ be a F{\o}lner sequence.
Say that a bounded function $f : \N \to \R$ \emph{admits correlations along $\Phi$} if
\begin{equation*}
	\lim_{N \to \infty} \frac{1}{|\Phi_N|} \sum_{n \in \Phi_N} f(n + t_1) \cdot \ldots \cdot f(n + t_k)
\end{equation*}
exists for every $k \in \N$ and $t_1, \ldots, t_k \in \N \cup \{0\}$.
We define the \emph{local $U^k(\Phi)$ uniformity seminorms} of a function $f$ along $\Phi$ by
\begin{align*}
	\seminorm{U^0(\Phi)}{f} & = \lim_{N \to \infty} \frac{1}{|\Phi_N|} \sum_{n \in \Phi_N} f(n), \\
	\seminorm{U^{k+1}(\Phi)}{f}^{2^{k+1}} & = \lim_{H \to \infty} \frac{1}{H} \sum_{h=1}^H \seminorm{U^k(\Phi)}{\Delta_h f}^{2^k}.
\end{align*}
These seminorms were introduced in \cite{hk_uniformity}, and the authors demonstrated the existence of the relevant limits using known results in ergodic theory.

One can show that if $A \subseteq \N$ and $\ind_A$ admits correlations along $\Phi$, then
\begin{equation*}
    \seminorm{U^2(\Phi)}{\ind_A - d_{\Phi}(A)} = 0    
\end{equation*}
if and only if for the system $(X, \mu, T)$ and set $E \subseteq X$ representing $A$ via the Furstenberg correspondence principle, one has $\E{\ind_E - \mu(E)}{\mZ} = 0$; see \cite[Appendix A]{kmrr_survey}.
Sets $A$ satisfying this property are called \emph{relatively weakly mixing} in \cite{bf}, and several additional descriptions are given in \cite[Section 5.2]{bf}. 
We have the following corollary.

\begin{corollary}
	Let $P(x) \in \Q[x]$ be a nonconstant integer-valued polynomial.
	Suppose $A \subseteq \N$ and there is a F{\o}lner sequence $\Phi$ such that $d_{\Phi}(A) > 0$ and
    \begin{equation*}
        \norm{U^2(\Phi)}{\ind_A - d_{\Phi}(A)} = 0.
    \end{equation*}
	Then there is an infinite set $\{b_n : n \in \N\} \subseteq \N$ such that
    \begin{equation*}
        \{b_j, P(b_i) + b_j : i < j\} \subseteq A.
    \end{equation*}
\end{corollary}


\subsection{Open questions}

We end our discussion by collecting some further questions about polynomial sumsets in the integers.

Our first question concerns a variant of Theorem \ref{thm: coloring P} in the integers.

\begin{question} \label{q: coloring}
	If $\N$ is finitely colored, must there exists an infinite set $B = \{b_n : n \in \N\}$ and a shift $t \in \Z$ such that $\{b_i, b_i^2 + b_j : i < j\} + t$ is monochromatic?
\end{question}

The shift $t$ cannot be eliminated by Theorem \ref{thm: not ramsey pr}.
However, the arguments in Subsection \ref{sec: 5-coloring} (as well as the proof using nonstandard analysis in \cite{ramsey_witness}) depend quite delicately on the algebraic structure of the pattern $\{b_i, b_i^2 + b_j : i < j\}$, and it is not at all clear how to produce a coloring that would also avoid monochromatic shifts of such configurations.

\bigskip

Our second question concerns the extent to which we can recover a density statement similar to Theorem \ref{thm: density P} in the integers.

\begin{question}
	Let $A \subseteq \N$.
	Which of the following conditions (from strongest to weakest) are sufficient to guarantee a configuration $\{b_i^2 + b_j : i < j\} \subseteq A$ for some infinite set $B = \{b_n : n \in \N\}$?
	\begin{itemize}
		\item	$A$ is a totally ergodic set (i.e., has a totally ergodic Furstenberg system).
		\item	$A$ is relatively totally ergodic (Conjecture \ref{conj: trivial projection to rational factor}).
		\item	There exists $\Phi$ such that $d_{\Phi}(A \cap k\N) > 0$ for all $k \in \N$.
		\item	There exists $\Phi$ such that for every $k \in \N$, there exists $n \in \N$ with
        \begin{equation*}
            d_{\Phi} \left( A \cap (k\N + n^2 + n) \right) > 0.
        \end{equation*}
	\end{itemize}
\end{question}

The first bullet has a positive answer by Theorem \ref{thm: TE}.
We note that the third bullet (and consequently also the fourth) is a partition regular condition.
That is, if $\N = \bigcup_{i=1}^r C_i$, then one of the sets $C_i$ will satisfy $d_{\Phi}(C_i \cap k\N) > 0$ for every $k \in \N$.
This offers a possible density reason for the fact (which follows from Ramsey's theorem) that given such a partition, one of the sets $C_i$ contains $\{b_i^2 + b_j : i < j\}$ for some infinite set $B = \{b_n : n \in \N\}$.


\end{document}